\documentclass[12pt]{amsart}
\usepackage{amsmath,amssymb,amsfonts} 
\usepackage{graphics}                 
\usepackage[square, numbers]{natbib}
\usepackage[english,french]{babel}
\usepackage{tikz} 
\usepackage[utf8]{inputenc}
\usepackage[all]{xy}	
\usepackage[T1]{fontenc}
\usepackage{pifont}
\usepackage{hyperref}
\usepackage[margin=1.2in]{geometry}
\usepackage{mathrsfs}       

\newcommand{\U}{\mathcal{U}}

\newcommand{\N}{\mathbb{N}}

\newcommand{\E}{\mathcal{E}}

 \DeclareMathOperator{\Mat}{Mat}
  \DeclareMathOperator{\Gal}{Gal}
   \DeclareMathOperator{\GL}{GL}

\newcommand{\Spec}{\rm{Spec}}

\renewcommand{\O}{\mathcal{O}}

\renewcommand{\P}{\mathbb{P}}

\theoremstyle{definition}
\newtheorem{definition}{Definition}[subsection]
\newtheorem{construction}[definition]{Construction}
\newtheorem{example}[definition]{Example}

\newtheorem{notation}[definition]{Notation}

\theoremstyle{plain}
\newtheorem{proposition}[definition]{Proposition}
\newtheorem{theorem}[definition]{Theorem}
\newtheorem{corollary}[definition]{Corollary}
\newtheorem{lemma}[definition]{Lemma}
\newtheorem{fact}[definition]{Fact}
\newtheorem{question}[definition]{Question}

\newtheorem{claim}{Claim}

\newtheorem*{claimproof}{Claim}

\theoremstyle{exampstyle}
\newtheorem{thmx}{Theorem}

\theoremstyle{exampstyle}

\numberwithin{equation}{section}
\theoremstyle{remark}
\newtheorem{remark}[definition]{Remark}

\def\Ind{\setbox0=\hbox{$x$}\kern\wd0\hbox to 0pt{\hss$\mid$\hss}
\lower.9\ht0\hbox to 0pt{\hss$\smile$\hss}\kern\wd0}
\def\Notind{\setbox0=\hbox{$x$}\kern\wd0\hbox to 0pt{\mathchardef
\nn=12854\hss$\nn$\kern1.4\wd0\hss}\hbox to
0pt{\hss$\mid$\hss}\lower.9\ht0 \hbox to
0pt{\hss$\smile$\hss}\kern\wd0}
\def\ind{\mathop{\mathpalette\Ind{}}}
\def\nind{\mathop{\mathpalette\Notind{}}}

\renewcommand{\Spec}{\mathrm{Spec}}

\newtheoremstyle{exampstyle}
{3pt} 
{3pt} 
{\itshape} 
{} 
{\bfseries} 
{.} 
{.5em} 
{} 

\makeindex

\title[Strongly minimal vector fields]{On the density of strongly minimal  algebraic vector fields}
\author{Rémi Jaoui}
\address{Université Claude Bernard Lyon 1, CNRS UMR 5208, Institut Camille Jordan, 43 blvd. du 11 novembre 1918, 69622 Villeurbanne, France}
\email{jaoui@math.univ-lyon1.fr}
\date\today
\subjclass[2020]{03C45, 12H05, 34M45, 32M25}
\begin{document}

\maketitle

\selectlanguage{english}

\begin{abstract} 
Two theorems witnessing the abundance of geometrically trivial strongly minimal autonomous differential equations of arbitrary order are shown. The first one states that a generic algebraic vector field of degree $d\geq 2$  on the affine space of dimension $n \geq 2$ is strongly minimal and geometrically trivial. The second one states that if $X_0$ is the complement of a smooth hyperplane section $H$ of a smooth projective variety $X$ of dimension $n$, then for $d$ large enough, the system of differential equations associated with a generic vector field on $X_0$ with a pole of order at most $d$ along $H$ is strongly minimal and geometrically trivial. 

This produces the first examples of meromorphic functions that are new in the sense of Painlevé and satisfy autonomous differential equations of order $n \geq 4$.
\end{abstract} 

%
%

\setcounter{tocdepth}{1}
\tableofcontents

\section{Introduction}

Following the pioneering work of Hrushovski \cite{Hrushovski-MordellLang} on Manin kernels, a large body of results has emerged in the model-theoretic literature concerning the study of the algebraic relations between solutions of algebraic {\em nonlinear} differential equations for which the classical Picard-Vessiot Galois theory is not applicable. 

The effectiveness of these model-theoretic techniques has been demonstrated both from a practical and from a theoretical perspective. On the practical side, they led to a full understanding  of the algebraic relations between solutions of various classical families of algebraic differential equations. This includes the families of {\em Painlevé equations} \cite{Nagloo-Pillay1, Nagloo-Pillay2, Freitag-Nagloo}, families of {\em Schwarzian equations} \cite{Freitag-Scanlon, Casale-Freitag-Nagloo, BCFN} and certain families of {\em Liénard equations} \cite{FJMN}. From a theoretical perspective, they provided  new insights on the possible algebraic relations between solutions of arbitrary algebraic differential equations \cite{Freitag-Jaoui-Moosa}.

\subsection{The model-theoretic strategy} Given a definable family of algebraic differential equations depending on some complex parameters $\alpha \in \mathbb{C}^n$  written as
$$E(\alpha) : F_\alpha(y,y',y'', \ldots, y^{(n)}) = 0,$$
the core of the model-theoretic strategy is to first establish a ``{\em{rough classification''}} of the algebraic relations among solutions of the equations $E(\alpha)$. This rough classification describes the structures of the solution sets of the equations $E(\alpha)$ for fixed values of the parameters $\alpha$ as well as the behavior of the relations between different equations in the family. It is phrased in the language of geometric stability theory in the sense of  \cite{Pillay} and called the {\em{semi-minimal analysis}} of the family.

Among all the possible semi-minimal analysis, a key feature of nonlinear differential algebra is that, after discarding a ``small subset'' $\mathcal E$ of parameters, one of them appears to have an {\em ubiquitous character} shared, for instance, by all the classical families of equations  above. This ubiquitous semi-minimal analysis can be summarized as follows:

\begin{enumerate} 
\item {\em strong minimality}: outside of the exceptional set $\mathcal E$ of parameters,  the equations $E(\alpha)$ are  strongly minimal: for every solution $y$ of $E(\alpha)$ and every differential field $L$ containing the parameters of the equation $E(\alpha)$,
$$ \mathrm{trdeg}(L\langle y \rangle / L )  = 0 \text{ or } n$$
where $n$ denotes the (common) order of the equations $E(\alpha)$ and $L\langle y \rangle$ denotes the differential field generated by $L$ and $y$.

\item {\em geometric triviality:} for every $r \geq 2$, every parameters $\alpha_1,\ldots,\alpha_r$ lying outside of the exceptional set $\mathcal E$ and every solutions $y_1,\ldots, y_r$ of $E(\alpha_1),\ldots, E(\alpha_r)$ respectively, if $\mathrm{trdeg}(L\langle y_1,\ldots, y_r \rangle /L) < r \cdot n$ then 
$$ \text{ for some } i \neq j, \mathrm{trdeg}(L\langle y_i,y_j\rangle/L) = 0 \text{ or } n, $$
where $L$ denotes any differential field containing the parameters of all the equations $E(\alpha_i)$. This condition is equivalent to geometric triviality of each individual fiber  outside of the exceptional set $\mathcal E$ of parameters.

\item {\em multidimensionality:} for every pair $(\alpha,\beta)$ of parameters outside of an exceptional set $\mathcal E_2$ of pairs of parameters, the equations $E(\alpha)$ and $E(\beta)$ are {\em orthogonal}: 
$$ \mathrm{Corr}(E(\alpha),E(\beta)) = \emptyset$$ 
where $\mathrm{Corr}(E(\alpha),E(\beta)) $ denotes the set of generically finite-to-finite correspondences between the solution sets of $E(\alpha)$ and $E(\beta)$.

\end{enumerate} 
 
The problem addressed in this article is to confirm the ubiquitous character of this semi-minimal analysis by showing that  the families of algebraic differential equations satisfying the properties (1) and (2) are indeed abundant in the universe of {\em autonomous} algebraic differential equations. The conjectural property (3) as well as other related problems are discussed in subsection 1.5. 

\subsection{State of the art} Before Nagloo and Pillay clarified in \cite{Nagloo-Pillay1} the relationship between (1) and classical ideas of Painlevé from \cite[Vingt-et-unième leçon, pp.487]{Painleve} on ``highly irreducible''  differential equations, the scope of this model-theoretic strategy was essentially limited to order one equations. Nevertheless, it was already clear since the beginning of the eighties that all the properties (1)-(3) should hold somewhat generically for algebraic differential equations of higher order as well. For instance, after showing that the Poizat equation $y''\cdot y = y'$ and $y' \neq 0$ satisfies (1), Poizat conjectures in \cite{Poizat} that {\em for every $n \geq 2$, a ``general'' order $n$ algebraic (nonlinear) differential equation should be strongly minimal}.

The first result of this kind was established in \cite{jaoui-ANT} and only came after substantial progress was made on the study of the classical examples mentioned above. It states that the family $\Xi(2,d)$ of two-dimensional systems of the form
$$ (S): \begin{cases} x' = f(x,y) \\ y' = g(x,y) \end{cases} \text{ where } f,g \in \mathbb{C}[X,Y]_{\leq d},$$
satisfies the properties (1) and (2) as soon as $d \geq 3$.  Allowing the coefficients of the equations to be differentially transcendental functions, Devilbiss and Freitag \cite{Devilbiss-Freitag} then established a similar result valid in arbitrary dimension. Their result states that for every $n \geq 1$ and every $d \geq 6$, the ``highly non-autonomous'' equations 
$$F(y,y',\ldots, y^{(n)}) = 0$$ where $F \in \mathcal M(U)[Y_0,\ldots, Y_n]$
is a (multivariate) polynomial of total degree $d$ with  {\em differentially independent and differentially transcendental} coefficients are strongly minimal. Together with the results of this article, they provide a complete answer to Poizat's conjecture.

\subsection{Statement of the main results}
The first result of this article concerns systems of polynomial differential equations on the affine space of arbitrary dimension. In contrast with the previous results, it already applies to quadratic systems of polynomial differential equations.

\begin{thmx} \label{theoremA}
Let $n,d \in \mathbb{N}$. Consider the family $\Xi(n,d)$ of systems of differential equations of the form 
$$ (S): \begin{cases} y'_1 & =  f_1(y_1,\ldots, y_n) \\  &\vdots  \\  y'_n & =   f_n(y_1,\ldots, y_n)  \end{cases} $$ 
parametrized by $n$-tuples of complex polynomials $f_1,\ldots, f_n$ of degree bounded by $d$. If $n,d \geq 2$ then the family $\Xi(n,d)$ satisfies the properties (1) and (2): 
\begin{itemize} 
\item[(1)] a {\em generic} system $(S)$ from $\Xi(n,d)$ is strongly minimal,
\item[(2)] if $y(1),\ldots, y(r)$ are $r$  ($n$-tuples of) solutions of {\em generic} systems $(S_1),\ldots , (S_r)$ from $\Xi(n,d)$ and $L$ is a differential field extension of $\mathbb{C}$ then 
$$\mathrm{trdeg}(L (y(1),\ldots,y(r)) /L) = r \cdot n$$
unless for some $i \neq j$,  $\mathrm{trdeg}(L(y(i),y(j))/L) = 0 \text{ or } n$.
\end{itemize}  
\end{thmx}

The conclusion of Theorem A asserts that the properties (1) and (2) are properties of ``a generic system from the family $\Xi(n,d)$'', namely that they hold for all the systems from $\Xi(n,d)$ {\em outside of an exceptional set of parameters formed by a countable union of proper Zariski-closed subsets} of the complex vector space of parameters of $\Xi(n,d)$. Since this family is defined over $\mathbb{Q}$, a standard argument implies that these Zariski-closed subsets can be taken to be  defined over $\mathbb{Q}$ as well. 

This theorem implies the following statement reflected by the title of the article: for every system $(S_0)$ from $\Xi(n,d)$ defined by $(f_0,\ldots, f_n)$ and every $\varepsilon > 0$, there exists a {\em strongly minimal geometrically trivial system} $(S_\varepsilon)$ from $\Xi(n,d)$ defined by $(f^\varepsilon_0,\ldots, f^\varepsilon_n)$ such that
$$ \mathrm{max}_{0 \leq i \leq n} \mid f_i - f^\varepsilon_i \mid < \varepsilon $$ 
where $\mid-\mid$ denotes any norm on the complex vector space of polynomials of degree $\leq d$. Finally, Theorem \ref{theoremA} is also equivalent to its {\em measure-theoretic version} asserting the vector fields from $\Xi(n,d)$ satisfying the properties (1) and (2) have full Lebesgue measure in the parameter space of $\Xi(n,d)$, see Remark \ref{remark-Lebesguemeasure}.  \\

The second theorem provides a similar conclusion for algebraic vector fields living on the complement of an hyperplane section in an arbitrary smooth projective complex algebraic variety.

\begin{thmx} 
Let $X$ be a smooth projective complex algebraic variety of dimension $n$, let $H_X$ be a smooth hyperplane section of $X$ and set $X_0 = X \setminus H_X$. Consider the family $\Xi(X_0,d)$ of vector fields on $X_0$ with a pole of order at most $d$ along $H_X$. There exists $d_0 \geq 0$ such that for all $d \geq d_0$, the family $\Xi(X_0,d)$ satisfies the properties (1) and (2): 
\begin{itemize} 
\item[(1)] a {\em generic} vector field $(X_0,v)$ from $\Xi(X_0,d)$  is strongly minimal,
\item[(2)] if $y_1,\ldots, y_r$ are  solutions of systems of differential equations associated to {\em generic} vector fields from $\Xi(X_0,d)$ and $L$ is a differential field extension of $\mathbb{C}$ then
$$ \mathrm{trdeg}(L(y_1,\ldots, y_r)/L) = r \cdot n $$
unless for some $i \neq j$, $\mathrm{trdeg}(L( y_i,y_j)/L) =  0 \text{ or } n$.
\end{itemize}
\end{thmx}  

In this case, the proper Zariski-closed subsets forming the exponential set of parameters can be taken to be defined over any field of definition of the algebraic variety $X$ and of the hyperplane section $H_X$. Since the $\Xi(X_0,d)$ when $d$ varies among the nonnegative integers form an increasing filtration 
$$\Xi(X) = \Xi(X_0,0) \subset \ldots \subset \Xi(X_0,d) \subset \ldots $$
of the space $\Xi(X_0)$ of regular vector fields on $X_0$, Theorem B implies that strongly minimal and geometrically trivial algebraic vector fields are dense in the space of all regular algebraic vector fields on the complement of a smooth hyperplane section in any smooth projective complex-algebraic variety. \\

Even after putting aside the effectiveness condition on the degree, Theorem A can not be deduced Theorem B applied to $X = \mathbb{P}^n$ because of the well-known discrepancy between the degree of an affine vector field and the projective degree of the associated foliation. A common feature of both settings established by Coutinho and Pereira \cite{Coutinho-Pereira, Coutinho} is  the existence of algebraic vector fields with {\em ``strongly isolated'' non-resonant singular points}: complex points where the algebraic vector field vanishes, the eigenvalues of the linear part of the vector field are $\mathbb{Q}$-linearly independent and which are not contained in any proper positive-dimensional irreducible subvariety tangent to the algebraic vector field. The subvarieties with the latter property are called {\em invariant} as they are preserved under the complex-analytic flow of the vector field.

\begin{thmx}
Let $(X,v)$ be a complex algebraic vector field with reduced singularities of dimension $n \geq 2$.  Assume that  $(X,v)$ admits a singular point $x_0 \in X$ satisfying:
\begin{itemize}
\item[(a)] the singularity $x_0$ is  {\em non-resonant} in the sense that the eigenvalues $\lambda_1,\ldots, \lambda_n \in \mathbb{C}$ of the linear part of $v$ at $x_0$ generate a free $\mathbb{Z}$-module of rank $n$ in $\mathbb{C}$,
\item[(b)]  the singularity $x_0$ is not contained in any {\em invariant} proper positive-dimensional irreducible algebraic subvariety of $(X,v)$.
\end{itemize}
Then the generic type of $(X,v)$ is minimal. 
\end{thmx} 

The algebraic vector field  $(X,v)$ has {\em reduced singularities} if its singular locus $\mathrm{Sing}(v)$ consisting of the points where $v$ vanishes has codimension greater or equal to two in $X$. Note that an algebraic vector field may admit nontrivial invariant subvarieties  and satisfy the assumptions of Theorem C. This is for example the case for the quadratic planar Lotka-Volterra systems to which this theorem is applied in \cite{DEJ}. In dimension two and dimension three, the assumption (a) that the singularity is non-resonant can be removed, see Remark \ref{remark-theoremminimality2}.

The conclusion describes the differential-algebraic behavior of the {\em generic solutions} of  $(X,v)$, those which are Zariski-dense in $X$. The existence of such analytic solutions is guaranteed by the properties (a) and (b) and any two such generic solutions satisfy the same quantifier-free formulas in the language of differential rings. The collection of all such formulas is called the {\em generic type} of the algebraic vector field. Hence, the generic types of algebraic vector fields play the same role for systems of autonomous algebraic differential equations in differential algebra as the generic points of integral schemes do for systems of polynomial equations in algebra.  The conclusion of Theorem C asserts that this type is {\em minimal}. This is a standard  model-theoretic weakening of the notion of strong minimality tailored to the description of types rather than definable sets. In the model-theoretic terminology, a type is minimal if it is nonalgebraic but every forking extension is, see also Fact \ref{fact-Jaoui-Moosa} (iii) below for an equivalent geometric characterization. In contrast with the stronger notion, this notion of minimality is {\em preserved under generically finite-to-finite correspondences} of algebraic vector fields and hence amenable to the birational methods of the theory of algebraic foliations.  

Although the articulation between minimality and strong minimality can be rather subtle in general \cite{Hrushovski-Scanlon,Pila-Scanlon}, it is rather mild in the setting of Theorem A and Theorem B, see Lemma \ref{lemma-strongminimality}. \\

As an application, we produce the first examples of meromorphic functions which are new in the sense of Painlevé and satisfy autonomous algebraic differential equations of order $n \geq 4$. They are the meromorphic functions satisfying an (autonomous) order $n$ algebraic differential equation which are algebraically independent from all the meromorphic functions that one can obtain from the solutions of algebraic differential equations of order $< n$ by successive resolutions of linear differential equations and compositions with abelian functions, see Definition \ref{definition-newmeromorphic}. The simplest examples we obtain are constructed as follows. Consider, in any dimension $n \geq 2$, a system of algebraic differential equations
$$ (E): Y' = A_0 \cdot Y + \varepsilon(Y) $$
obtained by perturbing a linear system $(L): Y' = A_0 \cdot Y$ defined by a matrix $A_0 \in \Mat_n(\mathbb{C})$ with $\mathbb{Q}$-algebraically independent eigenvalues by a homogeneous quadratic perturbation $\varepsilon(Y) \in \mathbb{C}[Y_1,\ldots,Y_n]_{= 2}$. 

Both the linear system $(L)$ and its pertubation $(E)$ admits the origin as a non-resonant singular point which is contained in $n$ hyperplanes of $\mathbb{C}^n$ {\em invariant for the linear system $(L)$}. Crucially, the results of Coutinho and Pereira \cite{Coutinho-Pereira,Coutinho} ensure that one can choose the quadratic perturbation $\varepsilon(Y)$ so that $(E)$ no longer admits any proper positive-dimensional invariant subvariety. Using this strong form of the hypothesis (b) from Theorem C, we can conclude that the perturbed system $(E)$ is even {\em strongly minimal}. Applying the model-theoretic machinery developed by Nagloo and Pillay \cite{Nagloo-Pillay1} in their study of the Painlevé equations, this ensures that any meromorphic function $z \mapsto f(z) = g (y_1(z),\ldots, y_n(z))$
built as a composition of a nonconstant rational function $g \in \mathbb{C}(Y_1,\ldots, Y_n) \setminus \mathbb{C}$ with a nonconstant solution $y_1(z),\ldots, y_n(z)$ of the system $(E)$ is a new meromorphic function of order $n$ in the sense of Painlevé.

Note that in this construction, only the choice of the quadratic perturbation is not constructive as it relies on the nonconstructive results of \cite{Coutinho-Pereira}. The details of this construction as well as similar constructions based on Theorem A and Theorem B are described in subsection 5.5.

\subsection{Description of proof methods and organization of the paper} The proof of these results uses a combination of {\em model-theoretic techniques} concerning the structure of strongly minimal sets in the theory $\mathrm{DCF}_0$ of differentially closed fields and of {\em geometric techniques} describing the dynamical properties of generic foliations on smooth projective algebraic varieties. We refer the reader to \cite{Bouscaren, Marker, Pizarro} for background on the model-theoretic terminology and quickly describe the articulation between these techniques in the rest of the article.

\subsubsection*{Algebraic vector fields without nontrivial algebraic factors.} The starting point is the following description of algebraic vector fields without nontrivial algebraic factors obtained by Moosa and the author.
\begin{fact}[{\cite[Theorem 1.3]{Jaoui-Moosa}}] \label{fact-Jaoui-Moosa}
Let $(X,v)$ be an algebraic vector field. One of the following three mutually exclusive cases occur: 

\begin{itemize} 
\item[(i)] the algebraic vector field $(X,v)$ admits {\em a nontrivial algebraic factor}: after passing to  generically finite dominant rational cover $\rho: (X',v') \dashrightarrow (X,v)$, the algebraic vector field admits a nontrivial rational factor 
 $$\pi: (X',v') \dashrightarrow (Y,w) \text{ with } 0 < \mathrm{dim}(Y) < \mathrm{dim}(X).$$ 
\item[(ii)] the algebraic vector field $(X,v)$ is {\em{a cover of a simple  abelian equation}}: there exist a simple abelian variety $A$ of dimension $\geq 2$, a nonzero translation-invariant vector field $v_A$ on $A$ and a generically finite dominant rational morphism 
$$ \phi : (X,v) \dashrightarrow (A,v_A).$$

\item[(iii)] the generic type of the algebraic vector field $(X,v)$ is {\em minimal}: for all algebraic vector fields $(Y,w)$, every proper invariant subvariety of $$(X \times Y,v \times w)$$ that projects dominantly on both factors is a generically finite cover of $(Y,w)$.
\end{itemize}   
\end{fact} 

Here, a morphism of algebraic vector fields denoted $\phi : (X,v) \rightarrow (Y,w)$ is a morphism of algebraic varieties $\phi : X \rightarrow Y$ whose differential $d\phi$ sends the vector field $v$ to the vector field $w$. 

The first result established in Section 3 is a corollary of Fact \ref{fact-Jaoui-Moosa} suitable to be analyzed by geometric means. To any algebraic vector field $(X,v)$, we associate in subsection 3.2 its {\em first projective prolongation} which is an algebraic vector field $\P(v^{[1]})$ supported by the projectivization $\P(T^\ast X)$ of the cotangent bundle of $X$. It is defined by a commutative diagram 
$$\xymatrix{ (T^\ast X, \setminus \lbrace 0\text{-section} \rbrace ,v^{[1]})  \ar[d]^\pi \ar[r] & (\P(T^\ast X), \P(v^{[1]})) \ar[ld]^p \\ 
(X,v)
 }$$
where $v^{[1]}$ is an algebraic vector field on the cotangent bundle $T^\ast X$ linear on the fibers. The algebraic vector field $v^{[1]}$ is constructed ``by prolongation of $v$'' using  the partial connection along $v$ on the sheaf $\Theta_X$ of vector fields on $X$ defined by the Lie-bracket of a vector field with our distinguished vector field $v$. The relevance of these prolongations for model-theoretic purposes goes back to the proof of the canonical base property in $\mathrm{DCF}_0$ from  \cite{Pillay-Ziegler}. They also play an important role in the definition of the Malgrange pseudogroup of an algebraic vector field, see \cite{Casale-Davy} for example.

The main result of Section 3 is the following minimality criterion based on the absence of  {\em invariant horizontal irreducible subvarieties} in $\P(T^\ast X)$. They are  the  closed  $\P(v^{[1]})$-invariant irreducible subvarieties of $\P(T^\ast X)$.

\begin{thmx}[Theorem \ref{theorem-minimality}]\label{theoremC}
Let $(X,v)$ be an algebraic vector field where $X$ is a smooth algebraic variety of dimension $n \geq 2$. Assume that  either $n \geq 3$ or $n = 2$ and $(X,v)$ does not admit a nontrivial rational first integral. 

If the first prolongation of $(X,v)$ is {\em quasi-irreducible} then the generic type of $(X,v)$ is minimal. 
\end{thmx} 

Here, we say that the first prolongation of $(X,v)$ is quasi-irreducible if $\P(T^\ast X)$ admits {\em exactly one} proper $\P(v^{[1]})$-invariant horizontal subvariety. In that case, it is always an hypersurface that we call the {\em canonical invariant hypersurface} of $(\P(T^\ast X), \P(v^{[1]}))$ which is described in greater details below.  We also obtain a Galois-theoretic version of Theorem D (Theorem \ref{theorem-C-Galois}) and show that, in dimension greater or equal to three, the converse of Theorem D does not hold. 

\subsubsection*{Foliations by curves vs vector fields.} To study these invariant horizontal subvarieties in Section 4, our primary ingredient is the description of {\em the characteristic foliation of a generic foliation by curves} obtained by Pereira \cite{Pereira-JEMS}.  

A difficulty that we have to overcome  is to account for the subtle articulation between the notions of {\em foliation} and {\em vector field} in the model-theoretic analysis of an autonomous differential equation. For instance, it shown in \citep[Example 4.14]{FJMN} that the two vector fields 
$$ \begin{cases} 
x' = y \\
y' = \frac y x \end{cases} \text{ and } \begin{cases} x' = xy \\
y' = y \end{cases} $$
are tangent to the same foliation by curves but have {\em opposite} semi-minimal analysis: the first one is strongly minimal geometrically trivial while the second one is two-step analyzable in the constants.

It turns out that this articulation for an algebraic vector field $(X,v)$ is materialized by the {\em canonical invariant hypersurface $H(v)$} of $\P(T^\ast X)$.  Outside of the singular locus, this hypersurface is defined as the projectivization of the conormal bundle of the one-dimensional foliation $\mathcal F$ tangent to the vector field $v$. In other words, its fiber over a nonsingular  point $x \in X \setminus \mathrm{Sing}(v)$ is the projectivization of the vector space 
$$ N^\ast\mathcal F_x = \lbrace \omega \in TX^\ast_x \mid \omega(v(x)) = 0 \rbrace \subset T^\ast X_x.$$  
There are then two important types of $\P(v^{[1]})$-invariant horizontal subvarieties: the ones contained in $H(v)$ only depend on the foliation $\mathcal F$ and the ones disjoint from $H(v)$ are characteristic features of the vector field $v$ itself.

Studying these two kinds of horizontal subvarieties separately  under the assumptions (a) and (b) of Theorem C, we deduce Theorem C from Theorem D in Section 4. On the one hand, the horizontal subvarieties which only depends on the foliation $\mathcal F$ can be controlled directly using the results of Pereira \cite{Pereira-JEMS}. This is done in subsections 4.1 and 4.2. On the other hand, the horizontal subvarieties  characteristic of the vector field $v$ correspond to invariant multidistributions of codimension one transverse to the foliation $\mathcal F$. They are studied in subsections 4.3 and 4.4 using the tangency locus of an invariant distribution of codimension one with the foliation $\mathcal F$ as was done in \cite{jaoui-ANT} in dimension two.

\subsubsection*{The proofs of Theorem A and Theorem B} Finally, we deduce Theorem A and Theorem B from Theorem C in Section 5. Starting with a family of algebraic vector fields, say a family $\Xi$ of algebraic vector fields  on a fixed smooth algebraic variety $X_0$ parametrized by an algebraic variety $S$, the property 
\begin{center} 
$\mathcal P(s)$: \textit{the solution set of $(X_0,v(s))$ is strongly minimal} 
\end{center}
only depends on the type $\mathrm{tp}(s/k_0)$ (in the field language) where $k_0$ is a countable field of definition for $\Xi,S$ and $X_0$. It follows that every generic vector field from $\Xi$ is strongly minimal if and only if some generic vector field from $\Xi$ is. Furthermore, the existence of {\em any} vector field from $\Xi$ with a non-resonant singularity ensures the existence of a generic vector field from $\Xi$ with the same property. This follows from the stability (in the sense of dynamical systems) of non-resonant singularities. The remaining hypothesis of Theorem C is provided by the following result.
  
\begin{fact}[{\citep[Theorem 1.1]{Coutinho-Pereira}}]\label{fact-Pereira-Coutinho-int}
Let $\mathcal L$ be an ample line bundle on a smooth projective complex algebraic variety $X$. There exists $k_0 = k_0(X,\mathcal L) \geq 0$ such that for all $k \geq k_0$, there exists a countable union of proper Zariski-closed subsets of the complex vector space $\Xi(X,\mathcal L^{\otimes k})$ of {\em $\mathcal L^{\otimes k}$-twisted vector fields}
$$\mathcal E_1(k) = \bigcup_{n \in \N} Z_n(\mathbb{C}) \subsetneq \Xi(X,\mathcal L^{\otimes k})$$ 
such that for all $v \in  \Xi(X,\mathcal L^{\otimes k}) \setminus \mathcal E_1(k)$, the {\em twisted algebraic vector field $(X,v)$} does  not leave invariant any proper positive-dimensional algebraic subvariety of $X$. 
\end{fact}

The notion of {\em twisted (algebraic) vector field} appearing in Fact \ref{fact-Pereira-Coutinho-int} is a natural extension of the usual notion of algebraic vector field adapted to the projective setting. It is described in subsections 2.2 and 2.3. Note that such an extension is required as a theorem of Buium \cite{Buium-projective} states that the solution set of a projective algebraic vector field $(X,v)$ is always internal to the constants, that is, its general solution can be parametrized algebraically by constants and finitely many particular solutions. This witnesses the failure of both (1) and (2) for projective algebraic vector fields, see also Theorem 6.11 in \cite{Jaoui-Moosa}.

In the twisted case, we show that if $H_X$ is a smooth hyperplane section of $X$ and $X_0$ is the complement of this hyperplane section as in Theorem B, then we have an isomorphism of complex vector spaces given by restriction: 
$$\Xi(X, \mathcal O_X(dH_X)) \simeq \Xi(X_0,d)$$
between vector space of $\mathcal O_X(dH_X)$-twisted vector field on $X$ and the vector space  $\Xi(X_0,d)$ of regular vector fields on $X_0$ with a pole of order at most $d$ along the hypersurface $H_X$, see Proposition \ref{presentation- twisted-vector field-hyperplanesection}. This isomorphism is used in subsection 5.3 to combine Fact \ref{fact-Pereira-Coutinho-int} with Theorem C and obtain the ``strong minimality part''  (1) of Theorem B. A similar argument based on \cite[Theorem 4.3]{Coutinho} is described to prove the first part of Theorem A in subsection 5.2.  

The ``geometric triviality part'' (2) of Theorem A and Theorem B then uses the trichotomy theorem of the unpublished manuscript of Hrushvoski-Sokolovic \cite{Hrushovski-Sokolovic}, see also \cite[Section 5.1]{Casale-Freitag-Nagloo} for a published account of their results. In the setting of autonomous differential equations of dimension greater or equal to two, the trichotomy boils down to the following result.
\begin{fact}[Hrushovski-Sokolovic, \cite{Hrushovski-Sokolovic}]\label{Hrushovski-Sokolovic}
If the solution set of an algebraic vector field $(X,v)$ is strongly minimal and if the dimension of the algebraic variety $X$ is greater or equal to two then the solution set of $(X,v)$ is also geometrically trivial. 
\end{fact}  


\subsection{Open questions} We left several open questions about the structure of these family of strongly minimal equations. We describe some of them below. 

\subsubsection*{Strong minimality and specialization} The first of these questions concerns the structure of the exceptional set of parameters. For all the classical examples of families of algebraic differential equations mentioned above 
$$ E(\alpha): F_\alpha(y,y',\ldots, y^{(n)}) = 0$$
the exceptional set of parameters where the equation fails to be strongly minimal is an at most countable union of proper Zariski-closed subset of $\mathbb{C}^n$. On the other hand, the results of this article (Theorem A and Theorem B) only ensures that the exceptional set is {\em contained} in such a subset. It seems reasonable to expect that the exceptional set is in fact always of the form observed for the classical families. This would follow from the following specialization (or generalization) result: 
\begin{question}[Under some reasonable smoothness and irreducibility assumptions]
Is it true that for a family of differential equations as above, if for some value of the parameter the differential equation is strongly minimal then for generic values of the parameter the equation is strongly minimal? 
\end{question} 

For the weaker version of {\em orthogonality to the constants}, the corresponding specialization result has been established in \cite{jaoui-Bulletin}. Furthermore, the results of \cite{Casale-Davy} concerning the specialization of the Malgrange pseudogroup strongly suggest that a positive answer to this question is expected.

\subsubsection*{Multidimensionality and trivial internal structure} All these classical families also satisfy the property (3) of multidimensionality. Furthermore, while certain Schwarzian equations give rise to strongly minimal structures with trivial forking geometry but with a rich (not $\aleph_0$-categorical) binary internal structure, one expects that the opposite happens in the generic case: 
\begin{itemize}
\item[(4)] \textit{trivial internal structure}: for $\alpha \in \mathbb{C}^n$ generic, 
$$\mathrm{Corr}(E(\alpha), E(\alpha)) = \lbrace \Delta_{E(\alpha)} \rbrace$$ 
\end{itemize} 
where $\Delta_{E(\alpha)}$ denotes the diagonal of the solution set of $E(\alpha)$.
A natural extension of the results presented here is to show that the properties (3) and (4) are also abundant for families of autonomous differential equations:  
\begin{question}
Is it true that the families $\Xi (n,d)$ for $n,d \geq 2$ and $\Xi(X_0,d)$ for $d$ large enough of Theorem A and Theorem B are multidimensional with trivial internal structure? 
\end{question} 
\subsubsection*{Abundance relatively to a foliation} The articulation between the notions of foliation and vector field plays an important role in this article and several recent developments of differential algebra.  As a consequence of the results of this article, we obtain that a ``generic'' one-dimensional foliation in the sense of \cite{Coutinho-Pereira} supports a vector field with a strongly minimal solution set. The Poizat equation and related examples also shows that certain foliations on surfaces with a Liouvillian (non-algebraic) first integral also support vector fields with the same property. 

A natural question is therefore \textit{which foliation support a vector field with a (strongly) minimal solution set?} In its simplest form (i.e. in dimension two), the only known obstruction is the existence of a nontrivial rational first integral:
\begin{question}
Is it true that every non algebraically-integrable foliation on a projective surface supports (on a Zariski-open set) a strongly minimal algebraic vector field? 
\end{question} 
\subsection{Acknowledgments} The author would like to thank J. V. Pereira for several discussions during the summer school \textit{Feuilletages et Géométrie Algébrique} in Grenoble in 2019 as well as during the workshop \textit{Foliations in Algebraic and Birational Geometry} in Lausanne in 2022  on his paper \cite{Pereira-JEMS} and on several additional  arguments used in this article. The author is also very grateful to Jean-Benoît Bost who suggested to him during his PhD  to study this problem using tools from foliation theory. The author was partially supported by ANR-DFG AAPG2019 GeoMod.
 

\section{Preliminaries} 

This preliminary section is organized as follows. In subsection 2.1, we recall the correspondence between  algebraic vector fields and their {\em{solution sets in a differentially closed field}}, see \cite{Hrushovski-Itai,Freitag-Jaoui-Moosa,Jaoui-Moosa}. The subsections 2.2 to 2.4 introduce some vocabulary on {\em twisted algebraic vector fields} from \cite{Coutinho-Pereira, Coutinho,Pereira-JEMS} that will be relevant in Section 5. Finally, in subsections 2.5 and 2.6, we recall some classical terminology on {\em distributions and foliations} in the setting of differential algebra. We refer to \cite[Sections 2 and 3]{Pizarro} for background on the standard model-theoretic terminology used in this section.

\subsection{Algebraic vector fields and autonomous equations} Let $k$ be an algebraically closed field of characteristic $0$ fixed once and for all. By an algebraic variety, we mean a (not necessarily irreducible) algebraic variety over $k$.

\begin{notation} 
Let $X$ be a smooth irreducible algebraic variety. We always denote by $TX \rightarrow X$ the {\em tangent bundle} of $X$, by $\Theta_X$ the sheaf of sections of $TX$, that is, the sheaf of {\em vector fields} on $X$ and by $\Xi(X) = \Theta_X(X)$ the (possibly infinite-dimensional) $k$-vector space of global regular vector fields on $X$.  Dually, $T^\ast X \rightarrow X$ denotes the {\em cotangent bundle of $X$} and $\Omega^1_X$ the sheaf of {\em one-forms} on $X$.   
\end{notation} 

\begin{definition} 
The category of {\em algebraic vector fields} (over $k$) is the category
\begin{itemize} 
\item whose objects, the algebraic vector fields, are the pairs $(X,v)$ where $X$ is a smooth irreducible algebraic variety and $v$ is a (global regular) vector field on $X$,
\item whose morphisms $\phi: (X,v) \rightarrow (Y,w)$ are the morphisms of algebraic varieties $\phi: X \rightarrow Y$ such that $d\phi_x(v(x)) = w(\phi(x))$ for all $x \in X$.
\end{itemize} 
\end{definition} 

\begin{fact}[{\citep[Section 1]{Hrushovski-Itai}}]\label{fact-solutionfunctor}
Fix  $ \mathcal U$ a saturated differentially closed field extending $k$ endowed with the trivial derivation. We have a functor called the {\em solution functor} 
$$(X,v) \mapsto (X,v)^\U : = \lbrace x \in X(\U) \text{ solution of } (X,v) \text{ in } \U \rbrace$$  
from the category of algebraic vector fields to the category of $k$-definable sets of the theory $\mathrm{DCF}_0$ (with morphisms the $k$-definable functions).
\end{fact} 
The $k$-definable set $(X,v)^\U$ is called the {\em solution set} of $(X,v)$. It follows from the irreducibility of $X$ that there exists a unique  type $p = p_{(X,v)} \in S(k)$ such that for some/every realization $x$ of $p$, 
\begin{equation}
\mathrm{trdeg}(k\langle x \rangle / k) = \mathrm{dim}(X).
\end{equation}
It is called the {\em{generic type}} of $(X,v)$. Since $k$ is algebraically closed, $p$ is always a stationary type, that is, $p$ admits a unique nonforking extension to any differential field extending $k$.

{\remark[Algebraic vector fields and autonomous equations] \label{remark-vectorfields-autonomous} Let $(X,v)$ be an algebraic vector field. We can associate to $(X,v)$ an autonomous differential equation as follows. 

The vector field $v$ defines a derivation $\delta_v : k(X) \rightarrow k(X)$ on the function field of $X$ and every $f \in k(X)$ defines a rational morphism of algebraic varieties
$$ \sigma_f: X \dashrightarrow \mathrm{Jet}^n(\mathbb{A}^1)$$
given by $x \mapsto ( f(x), \delta_v(f)(x), \ldots, {\delta^{(n)}_v(f)}(x) )$.
We say that $f$ is {\em primitive} if $\sigma_f$ is birational onto its image. In that case, since $X$ is $n$-dimensional and $\mathrm{Jet}^n(\mathbb{A}^1)$ is smooth affine of dimension $n + 1$, the Zariski-closure of the image of $\sigma_f$ is defined by a single equation of the form 
$$ (E): F(y,y',\ldots, y^{(n)}) = 0 $$
where $y,y',\ldots, y^{(n)}$ are the coordinates on $\mathrm{Jet}^n(\mathbb{A}^1)$. We then say that $(X,v)$ is a {\em geometric model of the autonomous equation $(E)$}. By the extension of Kolchin's primitive element theorem from \cite{Pogudin}, primitive rational functions do exist assuming $(X,v)$ admits no nonconstant rational first integral. Conversely, two geometric models of the same autonomous equation are birationally equivalent in the following sense.}

\begin{definition} 
A {\em rational} morphism of algebraic vector fields denoted $\phi : (X,v) \dashrightarrow (Y,w)$ 
is a rational morphism $\phi: X \dashrightarrow Y$ such that on some nonempty open set $U$ of $X$ (or equivalently on the open set of definition of $\phi$), 
$$d\phi_x(v(x)) = w(\phi(x)) \text{ for all } x \in U.$$  Moreover, we say that $\phi$ is {\em birational}, {\em dominant} or {\em{generically finite}} if the underlying morphism $\phi$ has the corresponding property.  
\end{definition} 

Given an algebraic vector field $(X,v)$,  the dominant rational morphisms $\phi : (X,v) \dashrightarrow (Y,w)$ with source $(X,v)$ are called the {\em rational factors} of $(X,v)$. 

\begin{definition}\label{definition-noalgebraicfactors}
Let $(X,v)$ be an algebraic vector field. An {\em algebraic factor} of  $(X,v)$ is a rational factor $\phi: (X',v') \dashrightarrow (Y,w)$
of some generically finite dominant cover $\rho : (X',v') \dashrightarrow (X,v)$ of $(X,v)$. 

\noindent We say that $(X,v)$ admits {\em no nontrivial algebraic factor} if every algebraic factor of $(X,v)$ satisfies $\mathrm{dim}(Y) = \mathrm{dim}(X)$. Otherwise, we say that the algebraic vector field $(X,v)$ admits a nontrivial algebraic factor.
\end{definition} 

In the terminology of \cite{Freitag-Jaoui-Moosa, Jaoui-Moosa}, an algebraic vector field admits no nontrivial algebraic factor if and only if its generic type admit no proper almost fibration.
\subsection{Twisted vector fields} 

\begin{definition} 
A {\em{twisted algebraic vector field}} is a triple $(X,\mathcal L, v)$ where $X$ is a smooth irreducible algebraic variety, $\mathcal L$ is an invertible sheaf on $X$ and $v$ is a global section of $\Theta_X \otimes \mathcal L$. If $X$ and $\mathcal L$ are fixed, we say that a global section of $\Theta_X \otimes \mathcal L$ is  an {\em $\mathcal L$-twisted vector field on $X$} and we denote by 
$$ \Xi(X,\mathcal L) = \mathrm{H}^0(X, \Theta_X \otimes \mathcal L)$$
the space of $\mathcal L$-twisted vector fields on $X$.
\end{definition} 
This generalizes the usual notion of algebraic vector field when $\mathcal L = \mathcal O_X$. Fix $(X,\mathcal L)$ be a smooth irreducible algebraic variety endowed with an invertible sheaf. The space of $\mathcal L$-twisted vector fields on $X$ will be identified with the space of morphisms of $\mathcal O_X$-modules:
$$ \Xi(X,\mathcal L)  \simeq \mathrm{Hom}_{\mathcal O_X}(\Omega^1_X , \mathcal L)$$
via $v \mapsto i_v$ given, for every open set $U$ of $X$, by:
\begin{equation}
 i_v: \begin{cases} 
 \Omega^1_X(U) \rightarrow \mathcal L(U)  \\ 
\omega \mapsto \omega(v) 
\end{cases}.
\end{equation}
Denoting by $\mathrm{Der}_k(\mathcal O_X,\mathcal L)$ the space of derivations of $\mathcal O_X$ into $\mathcal L$ which are trivial on $k$, we also have another identification:  
$$ 
\Xi(X,L)  \simeq \mathrm{Der}_k(\mathcal O_X, \mathcal L) $$ 
via $v \mapsto \delta_v$ given, for every open set $U$ of $X$, by:
\begin{equation}\label{derivation-definition}
\delta_v : \begin{cases} \mathcal O_X(U) \rightarrow \mathcal L(U)  \\
f \mapsto i_v(df)
\end{cases}.
\end{equation}
If $U$ is an open set of $X$, we have a natural {\em restriction ($k$-linear) map}
$$-_{\mid U} : \Xi(X,\mathcal L) \rightarrow \Xi(U, \mathcal L_{\mid U})$$ 
which is injective if $U$ is dense in $X$.  Moreover, if $U$ is affine then the invertible sheaf $\mathcal L$ is trivial on $U$. After fixing a nowhere vanishing section $s \in \mathcal L(U)$ inducing an isomorphism 
$\mathcal O_U \simeq \mathcal L_{\mid U}$
given by $ f\mapsto fs$, the restriction map produces a vector field on $U$ and is denoted:
$$-_{\mid U,s}: \Xi(X,\mathcal L) \rightarrow \Xi(U).$$ 

\begin{remark} \label{remark-counterexample}
If $s'$ is another nonvanishing section of $\mathcal L$ and $v \in \Xi(X,\mathcal L)$ then the vector fields obtained on $U$ by restriction relatively to $s$ and $s'$ are related by  
\begin{equation}\label{equation-multiplevectorfields} v_{\mid U,s} = g \cdot v_{\mid U,s'} \text{ for some } g \in \mathcal O^{\ast}_X(U)
\end{equation}
where $\mathcal O^\ast_X(U)$ denotes the group of invertible elements in the ring $\mathcal O_X(U)$. As an example, the two vector fields 
$$ \begin{cases} 
x' = y \\
y' = \frac y x \end{cases} \text{ and } \begin{cases} x' = xy \\
y' = y \end{cases} $$
differ by multiplication by the function $x$ which is invertible on the open set $x \neq 0$. In this case, Poizat \cite{Poizat} showed that the solution set of the first one is strongly minimal and the second one as it admits $y' = y$ as a rational factor. This example witnesses that (\ref{equation-multiplevectorfields}) does {\em not} guarantee that the model-theoretic properties of solution sets of  $(U,v_{\mid U,s})$  and $(U, v_{\mid U,s'})$ are the same. 
\end{remark}

\begin{lemma} \label{lemma-canonical-isomorphism}
Let $X$ be a smooth irreducible projective variety, let $D = \sum_{i = 1}^n n_i Z_i$ be a divisor on $X$ and denote by $X_0$  the complement of the support of $D$ in $X$. Consider $s$ a rational section of $\mathcal O_X(D)$ such that $\mathrm{div}(s) = D$ and $v \in \Xi(X, \mathcal O_X(D))$. The property:
$$(\mathcal P): \text{ the solution set of }(X_0,v_{\mid X_0,s}) \text{ is strongly minimal and geometrically trivial.} $$
only depends on $v$ and $D$ and not on the choice of the section $s$.  
\end{lemma}
Here, $\mathrm{div}(s)$ denotes the divisor on $X$ whose restriction to every open set $U$ of $X$ for which $\mathcal L_{\mid U}$ is trivial is given by $\mathrm{div}(s)_{\mid U} = \mathrm{div}(f)_{\mid U}$ where $s = f\cdot s_0$  and $s_0$ is a generating section of $\mathcal L_{\mid U}$. In particular, if $D$ is an effective divisor and $\mathrm{div}(s) = D$ then $s$ is a regular section of $\mathcal O_X(D)$.

\begin{proof} 
If $s'$ is another rational section of $\mathcal L$ such that $\mathrm{div}(s') = D$ then $s = \lambda s'$ for some $\lambda \in k^\ast$ so that 
$$ (X_0,v_{\mid X_0,s}) = (X_0,\frac{1}{\lambda} v_{\mid X_0,s'}).$$
Denoting by $\U_{1/\lambda}$ the differentially closed field deduced form $\mathcal U$ by dividing the derivation by $\lambda$, it follows that
 $$ (X_0,v_{\mid X_0,s})^\U = (X_0,\frac{1}{\lambda} v_{\mid X_0,s'})^{\U_{1/\lambda}} \subset X_0(\U)$$ 
Since the differentially closed fields $\mathcal U$ and $\mathcal U_{1/\lambda}$ have the same definable sets, the lemma follows. 
\end{proof} 

Hence, if $D$ is a divisor on $X$ and $X_0$ denotes the complement of the support of $D$ in $X$ then we can identify $\Xi(X,\mathcal O_X(D))$ with a finite-dimensional $k$-vector subspace of the space $\Xi(X_0)$ of regular vector fields on $X_0$ and we write 
\begin{equation}
-_{\mid X_0} : \Xi(X,\mathcal O_X(D)) \rightarrow \Xi(X_0)
\end{equation}
for the restriction map (without explicit mention of the choice of the section $s$). To describe the image of this map, we use the following notion of pole along an hypersurface for rational algebraic vector fields.

\begin{definition} 
Let $H$ be an hypersurface of $X$ and let $v$ be a rational vector field on $X$, that is, a section of $\Theta_X \otimes \mathbb{C}(X)$. We say that the vector field $v$ has {\em a pole of order at most $d$ along $H$} if there exists an affine open set $U$ such that $U \cap H \neq \emptyset$ and 
$$ h^d \cdot v = v_0 \in \mathrm{H}^0(X,\Theta_X \otimes \mathbb{C}(X)) $$
where $v_0$ is a {\em regular} vector field on $U$ and $h \in \mathcal O_X(U)$ is a (reduced) equation of $H \cap U$.
\end{definition}

{\remark Let $v$ be a rational vector field on $X$ and let $H$ be an hypersurface. Consider $U$ an affine open set of $X$ such that $U \cap H \neq \emptyset$,  $v_1,\ldots, v_n$ a basis of the free $\mathcal O_X(U)$-module  $\Theta(U)$ and write
$$ v = \sum_{i = 1}^n \lambda_i \cdot v_i \text{ where } \lambda_1,\ldots,\lambda_n \in \mathbb{C}(X).$$
Then the vector field $v$ has a pole of order at most $d$ along $H$ if and only if for all $i = 1,\ldots, n$, we have $v_H(\lambda_i) \geq - d$ where $v_H$ denotes the valuation associated to $H$ on $\mathbb{C}(X)$. In particular, every rational vector field has a pole of finite order along $H$.}

\begin{proposition} \label{presentation- twisted-vector field-hyperplanesection}
Let $X$ be a smooth  subvariety of the projective space $\P^N$, let $H_X = H \cap X$ be a smooth hyperplane section of $X$ and denote by $X_0$ be the complement of $H_X$ in $X$. The image of the restriction map 
$$-_{\mid X_0}: \Xi(X,\mathcal O_X(dH_X)) \rightarrow  \Xi(X_0) $$
is the $k$-vector space $\Xi(X_0,d)$ of regular vector fields on $X_0$ with a pole of order at most $d$ along the hypersurface at infinity $H_X$.  
\end{proposition} 

\begin{proof}
Denote by $\mathcal L = \mathcal O_X(H_X)$ so that $\mathcal L^{\otimes d} = \mathcal O_X(dH_X)$, fix a regular section $s$ of $\mathcal L$ on $X$ such that $\mathrm{div}(s) = H_X$ and define the restriction map
$$-_{\mid U_0}: \Xi(X,\mathcal L^{\otimes d}) \rightarrow \Xi(X_0)$$
relatively to the section $s^{\otimes d}$ of $\mathcal L^{\otimes d}$. Consider $\mathcal U = \lbrace U_0,\ldots, U_l \rbrace$ an affine open cover of $X$ such that each $U_i$ is equipped with a nowhere vanishing section $s_i$ of $\mathcal L_{\mid U_i}$. We may assume that $U_0 = X_0$ and that $s_0$ is the restriction of $s$ to $X_0$. By definition of the divisor of a section of line bundle, for every $i \leq l$, we can write 
$$ s_{\mid U_i} = h_i \cdot s_i$$ 
where $h_i \in \mathcal O_X(U_i)$ vanishes at order one on $H_X$ and hence is a local equation for $H_X$ on $U_i$ (so in particular $h_0 = 1$). It follows that the transition functions of $\mathcal L$ are given by 
$$\phi_{ij} = \frac {h_{i \mid U_i \cap U_j}} {h_{j \mid U_i \cap U_j}} \in \mathcal O^\ast _X(U_i \cap U_j).$$
They are indeed invertible since both $h_i$ and $h_j$ vanish only on $H$ and with the same order. Hence, an $\mathcal L^{\otimes d}$-twisted vector field can be identified with a collection 
$$ v:  \begin{cases} \lbrace (U_i,v_i) \mid i = 0,\ldots, l \rbrace \text{ of regular vector fields on each }U_i \\ 
\text{ such that } v_i = \phi_{ij}^d \cdot v_j \text{ on } U_i \cap U_j. \end{cases}
$$
We first show that the image of the restriction map is contained in $\Xi(X_0,d)$. Consider $v \in \Xi(X, \mathcal L^{\otimes d})$ as described above and consider $i$ such that $U_i \cap H \neq \emptyset$. On the open set $U_i \cap U_0$ we have 
$$h_i^d \cdot v_0 = \frac {h_i^d} {h_0^d} \cdot v_0 = \phi_{i0}^d \cdot v_0 = v_i$$
which is a regular vector field on $U_i$. It follows that $v_0$ which is the restriction of $v$ to $U$ has a pole of order at most $d$ along $H_X$ as required. To prove that the restriction map is surjective on $\Xi(X_0,d)$, we need the following claim: 

\begin{claimproof} 
Let $U$ be an affine open set of $X$ and let $H_U$ be an hypersurface of $U$. Assume that $v$ is a regular vector field on $U^\ast = U \setminus H_U$ with a pole of order at most zero along $H_U$. Then $v$ extends to a regular vector field $\overline{v}$ on $U$. 
\end{claimproof}

\begin{proof}[Proof of the claim] 
Denote by $\Theta_X(U)$ the free $\mathcal O_X(U)$-module of vector fields on $U$ and by $h$ an equation for $H_U$. Since $\Theta_X$ is a coherent sheaf, we have that $\Theta(U^\ast) = \Theta(U)_h$
is the localization of $\Theta(U)$ by the multiplicative set generated by $h$. It follows that 
$$h^n \cdot v  = v_1 \in \Theta_X(U)$$ 
extends to a regular vector field on $U$. Similarly, since $v$ has a pole of order at most zero along $H$, there is an irreducible $g \in \mathcal O_X(U)$  not associated with $h$ such that 
$$ g^m \cdot v = v_2 \in \Theta_X(U)$$
extends to a regular vector field on $U$ and we conclude that
$g^m\cdot v_1 = h^n \cdot v_2$.
Since  $X$ is smooth and $U$ is affine, $\mathcal O_X(U)$ is a unique factorization domain and we obtain that $g^m$ divides $v_2$ and that $h^n$ divides $v_1$ in $\Theta(U)$. This precisely means that $v$ extends to a vector field $\overline{v}$ on $U$. 
\end{proof} 
To finish the proof of surjectivity, consider $v_0 \in \Xi(X_0,d)$. 
For every $i = 1,\ldots, l$, the vector field $v_i = h_i^d \cdot v_0$ is a regular vector field on $U_i^\ast = U_i \setminus H_X$ with a pole of order at most zero along $H_X$. By the claim above, $v_i$ extends to a regular vector field $\overline{v_i}$ on $U_i$. It is now easy to see that the collection $(U_i, \overline{v_i})$ that we have constructed defines an $\mathcal L^{\otimes d}$-twisted vector field with restriction $v_0$: indeed, on $U_i \cap U_j \cap U_0$, we have
$$ \phi_{ij} \cdot \overline{v_j} = \frac{h_i^d}{h_j^d} \cdot \overline{v_j} =  h_i^d \cdot v_0 = \overline{v_i}$$  
and therefore the same equality holds in $U_i \cap U_j$. This finishes the proof of the proposition.
\end{proof}

{\example[Twisted algebraic vector fields on $\P^n$] \label{example-full-P^n} Set $X = \mathbb{P}^n$ with $n \geq 2$, let $H$ be an hyperplane and $d \geq 0$. The tangent sheaf on $\mathbb{P}^n$ is described by Euler's exact sequence 
$$ 0 \rightarrow \mathcal O_{\P^n} \rightarrow \mathcal O_{\P^n}(1)^{n+1} \rightarrow \Theta_{\P^n} \rightarrow 0.$$ 
After tensoring with $\mathcal O_{\P^n}(d)$, we obtain the exact sequence 
$$ 0 \rightarrow \mathcal O_{\P^n}(d) \rightarrow \mathcal O_{\P^n}(d+1)^{n+1} \rightarrow \Theta_{\P^n} \otimes \O_{\P^n}(d) \rightarrow 0.$$
Since $\mathrm{H}^1(\P^n, \mathcal O_{\P^n}(d)) = 0$ as $d \geq 0$, we can conclude that 
\begin{equation} \label{equation-twistedvectorfields}
\Xi(\P^n, \mathcal O_{\P^n}(d)) \simeq \mathrm{H}^0(\P^n, \mathcal O_{\P^n}(d+1)^{n+1} )/\mathrm{H}^0(\P^n, \mathcal O_{\P^n}(d)).
\end{equation}
In projective coordinates $(X_0: \cdots : X_n)$, this isomorphism can be interpreted as follows: every homogeneous vector field $v$ of degree $d +1$ on $\mathbb{C}^{n+1}$ written as
$$ v = \sum_{i = 0}^n F_i(X_0,\ldots,X_n) \cdot \partial_{X_i},$$
defines a $\mathcal O_{\P^n}(d)$-twisted vector field on $\mathbb{P}^n$. It is defined as the derivation $\delta_v: \mathcal O_{\P^n} \rightarrow \mathcal O_{\P^n}(d)$  which on the affine open set given by $X_i \neq 0$ with coordinates $(X_j/X_i)_{j \neq i}$ is given by:

\begin{equation}\label{equation-derivationassociated}
X_j/X_i \mapsto \frac{F_j(X_0,\ldots, X_n)} {X_i} -  \frac{X_j F_i(X_0,\ldots, X_n)} {X_i^2}.
\end{equation} 
The  Euler vector field $E = \sum_{i = 0}^n X_i \cdot \partial_{X_i}$
of degree one vanishes on all homogeneous functions and hence every multiple of $E$ induces the trivial $\mathcal O_{\P^n}(d)$-twisted vector field on $\mathbb{P}^n$. The equality (\ref{equation-twistedvectorfields}) then expresses that two homogeneous vector fields $v_1,v_2 \in \Xi_{h}(n+1,d + 1)$ of degree $d + 1$ on $\mathbb{C}^{n+1}$ define the same $\mathcal O_{\P^n}(d)$-twisted vector field on $\mathbb{P}^n$ if and only if their difference is a multiple of Euler vector field by a homogeneous function $f$ of degree $d$.}

\subsection{Closed invariant subvarieties} If $Z$ is a closed subvariety of $X$, we have an exact sequence
$$ 0 \rightarrow \mathcal I_Z \rightarrow \mathcal O_X \rightarrow \mathcal O_Z \rightarrow 0$$
which after tensoring with an invertible sheaf $\mathcal L$ gives: 
$$ 0 \rightarrow \mathcal I_Z \otimes \mathcal L \rightarrow \mathcal L \rightarrow \mathcal O_Z \otimes \mathcal L \rightarrow 0.$$
This exact sequence identifies $\mathcal I_Z \otimes \mathcal L$ with the subsheaf $\mathcal I_Z \cdot \mathcal L$ of $\mathcal L$ and $\mathcal O_Z \otimes \mathcal L$ with the quotient sheaf $\mathcal L / \mathcal I_Z \cdot \mathcal L$.

\begin{definition} 
Let $(X,\mathcal L,v)$ be a twisted algebraic vector field. A closed subvariety $Z$ of $X$ is called a {\em{closed invariant subvariety}} of $(X,\mathcal L,v)$ if
$\delta_v(\mathcal I_Z) \subset \mathcal I_Z\cdot\mathcal L$
where $\mathcal I_Z$ denotes the sheaf of ideals defining $Z$. We also say that $Z$ is {\em $v$-invariant} or that the twisted vector field $(X,\mathcal L,v)$ {\em leaves $Z$ invariant}.
\end{definition}

This definition obviously extends the notion of invariance of differential algebra and coincides with the notion of invariance of \citep[Section 2.2]{Coutinho-Pereira}:
\begin{lemma} 
Let $(X,\mathcal L,v)$ be a twisted algebraic vector field. A closed subvariety $Z$ of $X$ is $v$-invariant if and only if we have a factorization: 
$$\xymatrix{\Omega^1_X \otimes \mathcal O_Z \ar[r]^{j^\ast} \ar[d]_{i_v\otimes \mathcal O_Z}  & \Omega^1_Z \ar@{..>}[ld] \\ \mathcal L \otimes \mathcal O_Z}$$ 
where $j^\ast: \Omega^1_X \otimes \mathcal O_Z \rightarrow \Omega^1_Z$ is the morphism given by the pullback of one-forms.
\end{lemma} 

\begin{proof} 

Let $Z$ be a closed subvariety of $X$ associated with a sheaf of ideals $\mathcal I_Z$. By  \cite[Proposition 8.4A]{Hartshorne}, we have an exact sequence of $\mathcal O_Z$-modules
$$ \mathcal I_Z /\mathcal I_Z^2 \overset{\epsilon}{\rightarrow} \Omega^1_X \otimes \mathcal O_Z \overset{j^\ast}{\rightarrow} \Omega^1_Z \rightarrow 0$$
where $\epsilon$ is given by $\epsilon(f) = df \otimes 1$. Since $\mathcal L \otimes \mathcal O_Z \simeq \mathcal L / \mathcal I_Z \cdot \mathcal L$, we have a factorization as in the lemma if and only if for every local section $f \in \mathcal I_Z(U)$,
$i_v(df) \in \mathcal I_Z\cdot \mathcal L.$
The lemma then follows from the equality $i_v(df) = \delta_v(f)$
given by (\ref{derivation-definition}).
\end{proof}
\noindent The following classical properties of closed invariant subvarieties follow from the corresponding statements for algebraic vector fields. 

{\remark \label{remark-invariant-restriction} Let $(X,\mathcal L,v)$ be a twisted vector field and $U$ a dense open set of $X$.
\begin{itemize} 
\item[(i)] If $Z$ is a closed invariant subvariety of $(X,\mathcal L, v)$ then $Z \cap U$ is a closed invariant subvariety of $(U,\mathcal L_{\mid U},v_{\mid U})$.
\item[(ii)] Conversely, if $Z$ is a closed invariant subvariety of $(U,\mathcal L_{\mid U},v_{\mid U})$ then the Zariski closure $\overline{Z}$ of $Z$ in $X$ is a closed invariant subvariety of $(X,\mathcal L,v)$ 
\item[(iii)] The irreducible components of a closed invariant subvariety are also closed invariant subvarieties.
\item[(iv)] The intersection of two closed invariant subvarieties is also a closed invariant subvariety.
\end{itemize}}

\begin{proposition} \label{proposition-affine-degree}
Let $n \geq 1$ and $d \geq 1$. Consider $H$ a hyperplane in $\P^n$ and $\mathbb{A}^n$ the affine space with the hyperplane $H$ at infinity. The space  
$$\overline{\Xi}(n,d) = \lbrace v \in \Xi(\P^n, \mathcal O_X((d-1)H)) \mid  \text{ the hyperplane } H \text{ is invariant under } v \rbrace$$
is a $k$-vector subspace of $\Xi(\P^n,\mathcal O_X((d-1)H)$ and the image under the restriction map 
$$-_{\mid \mathbb{A}^n}: \overline{\Xi}(n,d) \rightarrow \Xi(\mathbb{A}^n)  $$
is the space $\Xi(n,d)$ of vector fields of (affine) degree $d$ (in the sense of Theorem A).
\end{proposition}

\begin{proof} 
Clearly, $\overline{\Xi}(n,d)$ is a $k$-vector subspace. To prove the second part, fix $(X_0: \cdots: X_n)$ projective coordinates on $\P^n$ such that the hyperplane $H$ is given by $X_0 = 0$ and the coordinates on $\mathbb A^n$ are the $x_i = X_i/X_0$
for $i = 1,\ldots, n$. 

\begin{claimproof} 
$\overline{\Xi}(n,d)$ is isomorphic to the $k$-vector subspace $\Xi$ of homogeneous vector fields of degree $d$ on $\mathbb{A}^{n+1}$ of the form 
$v = \sum_{i = 1}^n F_i \cdot \partial_{X_i}$, namely without terms in $\partial_{X_0}$.
\end{claimproof}
\begin{proof}[Proof of the claim] 
Denote by $\rho : \Xi(n+1,d) \rightarrow \Xi(\mathbb{P}^n, \mathcal O_{\P^n}(d))$ the $k$-linear map which sends a homogeneous vector fields of degree $d$ on $\mathbb{C}^{n+1}$ to an $\mathcal{O}_{\P^n}(d-1)$-twisted vector field on $\P^n$ described in Example \ref{example-full-P^n}. Since no vector field in $\Xi$ is a multiple of the Euler vector field, its restriction
$$ \rho_{\mid \Xi}: \Xi \rightarrow \Xi(\P^n,\mathcal O_{\P^n}(d-1))$$
is injective. Moreover, by definition, every $v \in \Xi$ satisfies $\delta_v(X_0) = 0$. Hence, the hyperplane $H$ is invariant under $\rho(v)$ so that the map $\rho_{\mid \Xi}$ takes values in $\overline{\Xi}(n,d)$. It remains to show that the image of $\rho_{\mid \Xi}$ is indeed $\overline{\Xi}(n,d)$: As described in Example \ref{example-full-P^n}, every vector field $v \in \overline{\Xi}(n,d)$ can be written as $\rho(v)$ where
$$ v = \sum_{i = 0}^n F_i \cdot \partial_{X_i} \in \Xi_h(n+1,d).$$ 
Since the hyperplane $H$ is invariant, we have that 
$F_0 = \delta_v(X_0) = h \cdot X_0$ 
for some homogeneous function $h$ of degree $d-1$ so that  
$$ v - h \cdot E = \sum_{i = 1}^n F_i \cdot \partial_{X_i} \in \Xi$$
and $\rho(v - h\cdot E) = \rho(v)$. This shows that $ \rho_{\mid \Xi}$ is surjective as required.
\end{proof} 
Set $U = \mathbb{A}^n$ for the complement of $X_0 \neq 0$. We now claim that the restriction map
$-_{\mid \mathbb{A}^n}: \Xi \rightarrow \Xi(\mathbb{A}^n) $
(relatively to the section $X_0^{d-1}$ of $\mathcal O_{\P^n}(d-1)$ nowhere vanishing on $U$) can be computed as:
$$ v = \sum_{i = 1}^n F_i \cdot \partial_{X_i} \mapsto v_0 = \sum_{i = 1}^n f_i \cdot \partial_{x_i}   $$ 
where $f_i(x_1,\ldots, x_n)$ is the polynomial of degree $\leq d$ given by dehomogenization of the homogeneous polynomial $F_i(X_0, X_1,\ldots, X_n)$ of degree $d$, that is,
$$f_i(X_1/X_0,\ldots, X_n/X_0) = \frac {F_i(X_0,X_1,\ldots, X_n)} {X_0^d}.$$

\noindent This follows easily from (\ref{equation-derivationassociated}) since after fixing the isomorphism 
$\mathcal O_{\P^n}(U) \simeq \mathcal O_{\P^n}(d-1)(U)$
given by $f \mapsto f\cdot X_0^{d-1}$, the restriction $v_0$ of $v$ to $U$ is the vector field associated to the derivation of $\mathcal O_{\P^n}(U)$ satisfying
$$x_i = X_i/X_0 \mapsto \frac {F_i(X_0,\ldots, X_n)} {X_0^d} = f_i(x_1,\ldots, x_n).$$
 We have therefore shown that the restriction map defines an isomorphism of $k$-vector spaces between $\overline{\Xi}(n,d)$ and $\Xi(n,d)$ as required. 
\end{proof}

\subsection{Singular points} Let $(X,\mathcal L,v)$ be a twisted algebraic vector field.

\begin{lemma} \label{lemma-singularity}
Let $x$ be a closed point of $X$. The following are equivalent: 
\begin{itemize}
\item[(i)] $\lbrace x \rbrace$ is a closed invariant subvariety of $(X,\mathcal L,v)$,
\item[(ii)] the morphism $i_v: \Omega^1_X \rightarrow \mathcal L$ is not surjective at $x$. 
\item[(iii)] the regular section of $v: X \rightarrow TX \otimes \mathcal L$ vanishes at $x$. 
\end{itemize} 
\end{lemma} 

\begin{proof} 
$(i) \Leftrightarrow (ii)$: Denote by $m_x$ the  maximal ideal corresponding to $x$. Since $\Omega^1_{\lbrace x \rbrace}$ is trivial, $\lbrace x \rbrace$ is a closed invariant subvariety if and only if the morphism of $k$-vector space
$$ \Omega^1_{X,x}/ m_x \Omega^1_{X,x} \rightarrow \mathcal L_{X,x} / m_x \mathcal L_{X,x}$$
is identically zero if and only if 
$i_v(\Omega^1_{X,x}) \subset m_x \cdot \mathcal L_{X,x}$
if and only if $i_v$ is not surjective at $x$ (as $m_x$ is the unique maximal ideal). 

$(ii) \Leftrightarrow (iii)$: we can consider an affine open set endowed with a system of étale coordinates $(U; x_1,\ldots, x_n)$ around $x$ (see subsection 3.1) and write 
$v = \sum_{i = 1}^n f_i \partial_{x_i} \otimes s$
where $s$ is a non-vanishing section of $\mathcal L$ on $U$.
It follows that $i_v(dx_i) = f_i \otimes s$ and hence that $i_v$ is not surjective if and only $f_1(x) = \ldots = f_n(x) = 0$.
\end{proof} 

\begin{definition} 
We say that a closed point $x$ of $X$ is {\em a singular point (or a zero of $v$)} if the equivalent conditions of Lemma \ref{lemma-singularity} hold. We denote by 
$$ \mathrm{Sing}(v) = \mathrm{Sing}(X,\mathcal L,v)$$
the closed (and obviously invariant under $v$) subvariety of $X$ whose closed points are the singular points of $v$. 
\textit{We say that a (twisted) algebraic vector field has {\em reduced singularities} if the singular locus $\mathrm{Sing}(v)$ has codimension $\geq 2$ in $X$.} 
\end{definition}

\begin{construction} 
Let $x_0 \in X$ be a singular point of an algebraic vector field $(X,v)$. The vector field $v$ defines a derivation 
$\delta_{v}: \mathcal O_{X,x_0} \rightarrow \mathcal O_{X,x_0}$
of the local ring $\mathcal O_{X,x_0}$. Since $x_0$ is a singular point, the maximal ideal $m_{x_0}$ is invariant under the derivation $\delta_v$. Hence, the ideal $m_{x_0}^2$ is also invariant under $\delta_v$ so that the derivation $\delta_v$ defines a $k$-linear map
$L: m_{x_0}/m_{x_0}^2 \rightarrow m_{x_0}/m_{x_0}^2$. Using the identification $TX_{x_0} = (m_{x_0}/m_{x_0}^2)^\ast$, the transpose of $L$ defines a $k$-linear map
$$L^\ast: TX_{x_0} \rightarrow TX_{x_0}$$
called \textit{the linear part of $v$ at $x_0$}. \end{construction}

\begin{lemma} \label{lemma-linearpart}
Let $(X,v)$ be an algebraic vector field with a singular point $x_0 \in X$. Consider an affine open $(U;x_1,\ldots,x_n)$ endowed with étale coordinates (see subsection 3.1) containing $x_0$ and write 
$$ v_{\mid U} = \sum_{i = 1}^n a_i \cdot \partial_{x_i}\text{ with } a_i \in \mathcal O_X(U).$$ 
\begin{itemize} 
\item[(i)] The matrix of the linear part of $v$ at $x_0$ in the basis of $TX_{x_0}$ defined by $\partial_{x_1}, \ldots, \partial_{x_n}$ is given by 
$A = (\partial_{x_j}(a_i))_{ 1\leq i,j \leq n }$. 
\item[(ii)] Consider the evaluation morphism associated to $v$ given on $U$ by 
$$e : \begin{cases} U   \rightarrow TU \simeq U \times TU_{x_0}  \rightarrow TU_{x_0} \\
x  \mapsto v(x)   \mapsto pr_2(v(x))
\end{cases} 
$$ 
Then $de_{x_0}: TU_{x_0} \rightarrow TU_{x_0}$ is the linear part of $v$ at $x_0$.
\end{itemize} 
\end{lemma}

\begin{proof} 
(i). Since $x_1,\ldots, x_n$ are étale coordinates, the germs of $dx_1,\ldots, dx_n$ at $x_0$ form a basis of $m_{x_0}/m_{x_0}^2$ and 
$$L(dx_i) = d (v(x_i)) = da_i = \sum_{ j = 1}^n \partial_{x_j}(a_i)dx_j $$
It follows that in the dual basis defined by $\partial_{x_1}, \ldots, \partial_{x_n}$,  the transpose of $L$ is given by the matrix $(\partial_{x_j}(a_i))_{1 \leq i,j \leq n}$. 

(ii). Indeed, with the identification $TU_{x_0} \simeq \mathbb{A}^n$ given by the basis $\partial_{x_1},\ldots, \partial_{x_n}$, the morphism $e$ is given by $e(x) = (a_1(x),\ldots, a_n(x))$.
\end{proof} 

\begin{definition} 
Let $(X,v)$ be an algebraic vector field and let $x_0 \in X$ be a singular point of $v$. We say that $x_0$ is a {\em non-resonant singularity of $v$} if the eigenvalues $\lambda_1,\ldots, \lambda_n$ of the linear part of $v$ at $x_0$ generate a free $\mathbb{Z}$-module of rank $n$ in $k$. 
\end{definition}

{\remark This condition is stronger than the classical non-resonance condition of the Poincaré-Dulac linearization theorem which only requires that the eigenvalues $\lambda_1,\ldots, \lambda_n \in k$ of the linear part of $v$ at $x_0$ do not satisfy any $\mathbb{Z}$-linear relation of the form 
$$ \lambda_j = \sum_{i = 1}^n m_i \cdot \lambda_i \text{ where the } m_1,\ldots, m_n \in \mathbb{N} \text{ satisfy } \sum_{i = 1}^n m_i \geq 2. $$ 
A singularity is non-resonant in our sense (as in subsection 5.1 of \cite{Pereira-JEMS}) if the Galois group of the linear differential equation defined by the linear part of $(X,v)$ around $x_0$ computed over the differential field $(k,0)$ is the Galois group of a ``generic autonomous'' linear differential equation of dimension $n$, that is, the complex torus $\mathbb{G}_m^n$.}

\begin{proposition} \label{proposition-singularlocus}
Let $X$ be a smooth projective variety of dimension $n$, let $D$ be a divisor on $X$ and denote by $X_0$ the complement of the support of $D$ in $X$. 

\begin{itemize} 
\item[(i)] Assume that $\mathcal O_X(D)$ is generated by its global sections and that $\mathrm{dim}(\Xi(X,\mathcal O_X(D))) > \mathrm{dim}(\Xi(X))$. Then there exists a dense open set $U \subset \Xi(X,\mathcal O_X(D))$ such that every twisted vector field of $U$ has reduced singularities.  

\item[(ii)]  Assume that $D$ is ample and fix $x_0 \in X_0$. There exists $d_0 \geq 0$ such that for all $d \geq d_0$ and for every matrix $A \in \Mat_n(k)$, there exists a twisted vector field $v \in \Xi(X,\mathcal O_X(dD))$ such that 

\begin{itemize}
\item the restriction $v_{\mid X_0}$ is singular at $x_0$ and 
\item the linear part of $v_{\mid X_0}$  at $x_0$ is given by the matrix $A$.
\end{itemize}
\end{itemize}   
\end{proposition}

\begin{proof} 
(i) follows from Proposition 2.4 of \cite{Coutinho-Pereira} in the case where $m = 1$. (ii) follows for the beginning of the proof of Proposition 4.1 of \cite{Coutinho-Pereira} for two singular points that we repeat here for one singular point.  Fix $x_0 \in X$ and denote by $\Theta_{X,x_0}$ the coherent subsheaf of $\Theta_X$ formed by the vector fields vanishing at $x_0$. We have an exact sequence of coherent sheaves over $X$ 
$$ 0 \rightarrow \Theta_{X,x_0} \otimes m_{x_0} \rightarrow \Theta_{X,x_0} \overset{L}{\rightarrow} \mathrm{Hom}_k(TX_{x_0}, TX_{x_0}) \rightarrow 0.$$ 
where we identify $\mathrm{Hom}_k(TX_{x_0}, TX_{x_0})$ with the skycraper sheaf at $x_0$ with fiber this vector space and $L$ is the map which sends a vector field vanishing at $x_0$ to its linear part, see the proof of Proposition 4.1 of \cite{Coutinho-Pereira}.

Fix $s$ a section of $\mathcal  O_X(D)$ such that $\mathrm{div}(s) = D$. After tensoring with the line bundle $\mathcal O_X(dD)$, we obtain  the exact sequence
$$ 0 \rightarrow \Theta_{X,x_0} \otimes m_{x_0} \otimes \mathcal O_X(dD) \rightarrow \Theta_{X,x_0} \otimes \mathcal O_X(dD) \overset{L}{\rightarrow} \mathrm{Hom}_k(TX_{x_0}, TX_{x_0}) \rightarrow 0.$$ 
Since $D$ is ample, there exists $d_0$ such that for all $d \geq d_0$, 
$$\mathrm{H}^1(X, \Theta_{X,x_0} \otimes m_{x_0} \otimes \mathcal O_X(dD)) = 0.$$
Hence, for $d \geq d_0$, the morphism of $k$-vector spaces 
$\mathrm{H}^0(X, \Theta_{X,x_0} \otimes \mathcal O_X(dD))) \rightarrow \mathrm{Hom}_k(TX_{x_0}, TX_{x_0})$ is surjective.
This means that for every matrix $A \in \Mat_n(k)$, there is an $\mathcal O_X(dD)$-twisted vector field vanishing at $x_0$ and with linear part $A$. 
\end{proof}
\subsection{Some classical terminology on distributions and foliations} If $\mathcal E$ is a coherent sheaf on $X$, we write $\mathcal E^\vee = \mathrm{Hom}_{\mathcal O_X}(\mathcal E, \mathcal O_X)$. If $\mathcal E$ is the sheaf of sections of a vector bundle $V$ over $X$ then $\mathcal E^\vee$ is the sheaf of sections of the dual vector bundle $V^\ast$.  

\begin{definition} 
A \textit{predistribution} $\mathcal D$ of rank $r$ on $X$ is a coherent subsheaf $\Theta_\mathcal D$ of the tangent sheaf $\Theta_X$ of $X$ of rank $r$ called {\em the tangent sheaf of the distribution $\mathcal D$}. We say that a predistribution $\mathcal D$ is 
\begin{itemize} 
\item[(1)] a \textit{distribution} if the subcoherent sheaf $\Theta_D$ of $\Theta_X$ is saturated: that is if $\Theta_X/\Theta_D$ does not have torsion.
\item[(2)] a \textit{foliation} if it is a distribution which is also involutive: for every local sections $v,w \in \Theta_D(U)$, we have $[v,w] \in \Theta_\mathcal D(U) $ where $[-,-]$ is the Lie bracket of vector fields.  
\end{itemize} 
\end{definition}

A predistribution $\mathcal D$ defines an exact sequence 
\begin{equation} \label{equation-exactsequence-distribution}
0 \rightarrow \Theta_\mathcal D \rightarrow \Theta_X \rightarrow \Theta_X/\Theta_\mathcal D \rightarrow 0
\end{equation}
The {\em singular locus} $\mathrm{Sing}(\mathcal D)$ of the predistribution is the closed subset of $X$ formed by the points $x \in X$ such that (\ref{equation-exactsequence-distribution}) is not an exact sequence of locally free sheaves at $x$. 
It is well-known (see for example \cite[Lemma 2.1.3]{jaoui-Confluentes}) that the saturation condition (1) in this definition is equivalent to: 
\begin{itemize}
\item[(1)'] the singular locus of $\mathcal D$ has codimension $\geq 2$ in $X$.
\end{itemize}

\begin{notation} \label{notation-normal-bundle}
Let $\mathcal D$ be a distribution of rank $r$ on $X$ and set $U = U(\mathcal D)$ for the complement in $X$ of the singular locus of $\mathcal D$. The exact sequence (\ref{equation-exactsequence-distribution}) induces an exact sequence of vector bundles over $U$ 
\begin{equation} 0 \rightarrow T\mathcal D \rightarrow TU \rightarrow N\mathcal D \rightarrow 0
\end{equation} 
where the vector bundle $N\mathcal D$ over $U$ (of rank $n - r$) is called the {\em normal bundle} of the distribution $\mathcal D$. The dual vector bundle $N^\ast\mathcal D$ can be identified with the vector subbundle of $T^\ast U$ whose fiber $N\mathcal D_x$ over $x \in U$ is given by: 
$$N^\ast\mathcal D_x = \lbrace \omega \in T^\ast X_x \mid \omega(v) = 0 \text{ for all } v \in T\mathcal D_x \rbrace.$$ 
It is called the {\em conormal bundle of the distribution} $\mathcal D$. The conormal bundle $N^\ast\mathcal D$  determines the distribution $\mathcal D$ because of the following fact. 
\end{notation}

\begin{fact}[Extension of distributions] \label{fact-extension distribution}
Let $U$ be a dense open set of $X$. The restriction map 
$\mathcal D \mapsto \mathcal D_{\mid U}$
sending a distribution on $X$ to a distribution on $U$ is {\em one-to-one} and sends foliations to foliations. In other words, every distribution (resp. foliation) on $U$ can be uniquely extended to a distribution (resp. a foliation) on $X$. 
\end{fact} 

See for example the proof of Proposition 2.3.1 of \cite{jaoui-ANT}. \qed

{\example[Foliations of rank one]\label{example-foliations of rank one} Let $(X,\mathcal L)$ be a smooth algebraic variety endowed with a line bundle $\mathcal L$.  We have a natural isomorphism 
$$\Xi(X,\mathcal L)  = \mathrm{H}^0(X,\Theta_X \otimes \mathcal L)  \simeq \mathrm{Hom}(\mathcal L^\vee, \Theta_X)   $$ 
which to a twisted vector field $v$ associates the $\mathcal O_X$-linear morphim given, on affine open sets $U$ of $X$, by
$$f_v: \begin{cases} 
\mathcal L^\vee(U) \rightarrow \Theta_X(U) 
\\ 
g \mapsto g\cdot v \end{cases}.$$
The following properties are immediate:
\begin{itemize} 
\item[(i)] If $v\neq 0$ then $f_v$ is injective and the image $\mathcal D(v)$ of $f_v$ is distribution of rank one with tangent sheaf isomorphic to $\mathcal L^\vee$.
\item[(ii)] If $v \neq 0$, then $\mathrm{Sing}(v) = \mathrm{Sing}(\mathcal D(v))$. Hence, if $v$ has reduced singularities then $\mathcal D(v)$ is a foliation of rank one. 

\item[(iii)] If $v,w \neq 0$ are two $\mathcal L$-twisted algebraic vector fields then $\mathcal D(v) = \mathcal D(w)$ if and only if there exists $f \in \mathcal O_X^\ast (X)$ such that $v = f \cdot w$. 
\end{itemize} 
It follows that the set of (rank one) foliations on a {\em projective} variety $X$ with tangent sheaf isomorphic to $\mathcal L$ can be identified with the open set   
$$
\mathrm{Fol}(X,\mathcal L) \subset \P(\Xi(X,\mathcal L^\vee))
$$
consisting of the projective classes $[v]$ of $\mathcal L^\vee$-twisted vector fields with reduced singularities. In particular, the space of foliations of degree $d$ on $\P^n$ can be identified with the open set 
\begin{equation} \mathrm{Fol}(\P^n,d) \subset \P(\Xi(\P^n,\mathcal O_{\P^n}(d-1)))
\end{equation}
consisting of the twisted vector fields with reduced singularities \cite[Section 8.1]{Pereira-JEMS}.
}

\begin{example}[Algebraically integrable foliations] \label{example-algebraically-integrable}
Let $\phi: X \dashrightarrow Y$ be a rational dominant morphism of algebraic varieties. On the open set of definition of $\phi$, we have an exact sequence of coherent sheaves  
$$ 0 \rightarrow \Theta_{X/Y} \rightarrow \Theta_X \rightarrow \phi^\ast \Theta_Y \rightarrow 0.$$ There is a unique foliation on $X$ denoted $\mathcal F_\phi$ and called the {\em{foliation tangent to $\phi$ on $X$}} such that at the generic point $x \in X$, we have $\mathcal F_{\phi,x} = \Theta_{X/Y,x}$.
See Definition 2.3.4 of \cite{jaoui-Confluentes}. 
\end{example}

\begin{construction}[Pullback of distributions] Let $\phi: X \dashrightarrow Y$ a dominant generically finite rational morphism of smooth  irreducible algebraic varieties and let $\mathcal D_Y$ be a distribution on $Y$. It follows from Fact \ref{fact-extension distribution} that there exists a unique distribution $\mathcal D_X$ on $X$ with the property that
\begin{quote}
for every open set $U$ of $Y$ such that $\phi^{-1}(U) \rightarrow U$ is étale, the lifts  (see Example \ref{example-lift})  of the local sections of $\Theta_ {\mathcal D_Y}(U)$ are generating sections of $\Theta_{\mathcal D_X}(\phi^{-1}(U))$. 
\end{quote}
The distribution $\mathcal D_X$ is called the {\em pullback distribution of $\mathcal D_Y$ by $\phi$ and denoted $\phi^\ast \mathcal D_Y$}. Note that the distributions $\mathcal D_Y$ and $\phi^\ast \mathcal D_Y$ have the same rank and that $\phi^\ast \mathcal D_Y$ is a foliation if $\mathcal D_Y$ is one. 
\end{construction}
\subsection{Invariant distributions} Let $(X,v)$ be an algebraic vector field.

\begin{definition} 
We say that a distribution $\mathcal D$ is an {\em invariant distribution of $(X,v)$ } if for every local section  $w \in \mathcal D(U)$, we have $[v,w] \in \mathcal D(U)$. If $\mathcal D$ is a foliation which is an invariant distribution of $(X,v)$, we say that $\mathcal D$ is an \textit{invariant foliation of $(X,v)$}. 
\end{definition}

\begin{remark}[See Section 3.2 of \cite{jaoui-Confluentes}] The analytic interpretation of this definition for foliations goes as follows: a foliation $\mathcal F$ of rank $r$ defines a partition (equivalently, an equivalence relation) of $X$ into a singular set $\mathrm{Sing}(\mathcal F)$ and a partition of $U = X \setminus \mathrm{Sing}(\mathcal F)$ into analytic immersed subvarieties of dimension $r$ called the {\em leaves of the foliation}.The foliation $\mathcal F$ is an invariant for $(X,v)$ if and only if this partition is preserved by the complex-integral curves of $(X,v)$ in the sense that: 
\begin{itemize}
\item the singular set $\mathrm{Sing}(\mathcal F)$ is a closed invariant subvariety of $(X,v)$,
\item if $\gamma_1, \gamma_2: \mathbb{D} \rightarrow X$ are two integral curves of the algebraic vector field $(X,v)$ such that $\gamma_1(0), \gamma_2(0)$ lie in the same leaf then the same holds true for $\gamma_1(t),\gamma_2(t)$ for every $t \in \mathbb{D}$.
\end{itemize} 
\end{remark}

\begin{proposition} \label{proposition rational factors to invariant foliations}
Denote by $\mathcal F(v)$ the one-dimensional foliation on $X$ tangent to the algebraic vector field $(X,v)$ The following properties hold:
\begin{itemize} 
\item[(i)] The foliation $\mathcal F(v)$ is an invariant foliation of $(X,v)$. More generally, every foliation containing $\mathcal F(v)$ is an invariant foliation.
\item[(ii)] If $\phi: (X,v) \dashrightarrow (Y,w)$ is a rational dominant morphism of algebraic vector fields then the (algebraically integrable) foliation $\mathcal F_\phi$ is invariant for $(X,v)$.
\end{itemize} 
\end{proposition} 

\begin{proof} 
(i) follows from the definitions, (ii) is Proposition 3.1.5 of \cite{jaoui-Confluentes}. 
\end{proof}

\section{The first projective prolongation}
{\em{Let $X$ be a smooth irreducible algebraic variety of dimension $n \geq 2$ over some algebraically closed field $k$ of characteristic zero fixed once and for all}}.

\subsection{Systems of étale coordinates} We often need to work locally on $X$. This will be done by covering $X$ by affine open sets endowed with étale coordinates.

\begin{definition}
An {\em affine open set endowed with étale coordinates} of $X$ is a pair $(U; x_1,\ldots, x_n)$ where $U$ is a nonempty affine open set of $X$ and $(x_1,\ldots, x_n): U \rightarrow \mathbb{A}^n_k$ is an étale morphism. 
\end{definition}
Systems of étale coordinates correspond to uniformizing parameters in the sense of \cite{Mumford}. We refer to \cite[Chapter III, Section 6]{Mumford} for more details and an (easy) proof of the following fact:
\begin{fact}
Every smooth algebraic variety can be covered by affine open sets endowed with étale coordinates.
\end{fact}

Let $(U;x_1,\ldots, x_n)$ be an affine open set with étale coordinates of $X$. By definition of étale morphisms, $\Omega^1_X(U)$ is the free $\O_X(U)$-module generated by $dx_1, \ldots, dx_n$. We always denote by $\partial_{x_1}, \ldots, \partial_{x_n}$ the dual basis of $\Theta_X(U)$ so that a vector field $v$ on $U$ can always be uniquely written as  
$$v = \sum_{i = 1}^n a_i \cdot \partial_{x_i} \text{ where } a_1,\ldots, a_n \in \mathcal O_X(U).$$

\example[Lift by an étale morphism] \label{example-lift} Let $\phi: X \rightarrow Y$ be an étale morphism of smooth algebraic varieties. Every vector field $w$ on $Y$ has a unique lift to a vector field $v := \phi^\ast w$ on $X$ such that 
$$\phi: (X,v) \rightarrow (Y,w)$$
is a morphism of algebraic vector fields. If $(V; y_1,\ldots, y_n)$ is an affine open set of $Y$ endowed with étale coordinates, after setting $$U = \phi^{-1}(V) \text{ and } x_i = y_i \circ \phi \text{ for } i = 1,\ldots, n$$ then $(U; x_1,\ldots, x_n)$ is an affine open set of $X$ endowed with étale coordinates and the lift of 
$w = \sum_{i = 1}^n a_i \cdot \partial_{y_i}$ is 
$v = \sum_{i = 1}^n (a_i \circ \phi) \cdot \partial_{x_i}$.

\example[Presentation of the cotangent bundle] Let $(U;x_1,\ldots, x_n)$ be an affine open set with étale coordinates of $X$.
 For every $i \leq n$, $\partial_{x_i}$ defines a function $y_i$ on $\pi^{-1}(U)$ given by  
$$ y_i: (\omega, x) \mapsto \omega(\partial_{x_i}(x))$$
The open set $\pi^{-1}(U)$ is affine and we have a natural identification:  
$$\mathcal O_{T^\ast X}(\pi^{-1}(U)) = \mathcal O_X(U)[y_1, \ldots , y_n]$$
where the left-hand side is the (free) polynomial ring generated by the $y_i$. This ring is naturally equipped with a grading given by the total degree in $y_1,\ldots, y_n$ and we have that: 
$$\pi^{-1}(U) \simeq \Spec \text{ } \mathcal O_X(U)[y_1, \ldots , y_n] \text{ and } p^{-1}(U) \simeq \mathrm{Proj} \text{ } \mathcal O_X(U)[y_1, \ldots , y_n].$$
In particular, the $(\pi^{-1}(U); x_1,\ldots, x_n, y_1, \ldots, y_n)$ where $(U;x_1,\ldots, x_n)$ runs through a covering of $X$ by affine open sets equipped with étale coordinates gives a covering of $T^\ast X$ by affine open sets equipped with étale coordinates.

\subsection{Partial connection along an algebraic vector field} \label{section-differential structure} We fix $(X,v)$ an algebraic vector field and we denote by 
$$\delta_v : \mathcal O_X \rightarrow \mathcal O_X$$
the derivation defined by $v$ on the structure sheaf of $X$. For coherence with the rest of the text, this section has been written for algebraic vector fields but works equally well for (not necessarily autonomous) smooth finite-dimensional $D$-varieties in the sense of Buium \cite{Buium}.

\begin{definition} \label{definition-partialconnectionalongavectorfield}
Let $\mathcal E$ be a coherent sheaf on $X$. A {\em partial connection $\nabla_v$ on $\mathcal E$ along $v$} is an additive morphism of sheaves
$\nabla_v: \mathcal E \rightarrow \mathcal E$
satisfying the Leibniz rule: for every local sections $f \in \mathcal O_X(U)$ and  $\sigma \in \mathcal E(U)$,
$$\nabla_v(f\cdot \sigma) = \delta_v(f) \cdot \sigma + f \cdot \nabla_v(\sigma).$$

\end{definition}

\begin{lemma} \label{dual-connection} 
Assume that $\nabla_v$ is a partial connection on $\mathcal E$ along $v$. There is a unique partial connection $\nabla_v^\vee$ along $v$ on $\mathcal E^\vee = \mathrm{Hom}_{\mathcal O_X}(\mathcal E, \mathcal O_X)$  such that 
\begin{equation}\label{equation-dual}
\delta_v(\omega(\sigma)) = \nabla_v^\vee(\omega) (\sigma) + \omega(\nabla_v(\sigma))
\end{equation}
for every local sections $\omega \in \mathcal E^\vee(U)$ and $\sigma \in \mathcal E(U)$.
\end{lemma} 

\begin{proof} 
Since a section of $\omega \in \mathcal E^\vee(U)$ is determined by the $\omega(\sigma)$ where $\sigma \in \mathcal E(U)$, the previous formula determines at most one partial connection on $\mathcal E^\vee$ along $v$. Therefore it suffices to check that 
$$ \omega \mapsto   \nabla_v^\vee\omega: \sigma \mapsto \omega(\nabla_v(\sigma)) - \delta_v(\omega(\sigma))$$
is a well-defined connection. This follows from a direct computation.  
\end{proof} 

\begin{example} \label{standard}
The connection 
$$\nabla^{\mathrm{st}}_v : \xi \mapsto [v , \xi]$$
where the Lie bracket on the right-hand side is the Lie bracket of vector fields is a partial connection on $\Theta_X$ along $v$. This connection and its dual on $\Omega^1_X$ will be called the {\em partial standard connections associated to $(X,v)$}.
\end{example} 

If $\mathcal E$ is a locally free sheaf on $X$, we denote by $\pi: \mathbb{E}(\mathcal E) \rightarrow X$ the vector bundle  whose local sections are the sections of $\mathcal E$ and by $\mathbb{P}(\mathcal E)$ its projectivization. A local section $\sigma \in \mathcal E^\vee(U)$ can be identified with a function on $\pi^{-1}(U)$ given by 
$$\overline{\sigma}: (x,e) \mapsto \sigma_x(e)$$
and the correspondence $\sigma \mapsto \overline{\sigma}$ is $\mathcal O_X(U)$-linear.

\begin{proposition} \label{totalspace} 
Let $(\mathcal E, \nabla_v)$ be a locally free sheaf endowed with a partial connection along $v$. There are unique algebraic vector fields $\mathbb{E}(\nabla_v)$ on $\mathbb{E}(\mathcal E)$ and $\mathbb{P}(\nabla_v)$ on $\mathbb{P}(\mathcal E)$ respectively such that the following diagram of algebraic vector fields commutes: 

$$\xymatrix { (\mathbb{E}(\mathcal E) \setminus \lbrace 0\text{-section} \rbrace , \mathbb{E}(\nabla_v)) \ar[r] \ar[d]^\pi & (\P(\E), \P(\nabla_v)) \ar[ld]^p \\ (X,v) }$$
and  for every local section $\sigma \in \mathcal E^\vee(U)$, 
$\mathbb{E}(\nabla_v)(\overline{\sigma}) = \overline{\nabla_v^\vee(\sigma)}$
where the vector field $\mathbb{E}(\nabla_v)$ has been identified with the derivation it defines on $\mathcal O_{\mathbb{E}(\mathcal E)}$.
\end{proposition}

\begin{proof} 
We work locally on $X$: consider $U$ an affine open set of $X$, a basis $e_1,\ldots, e_n$ of $\mathcal E(U)$ and denote by $\sigma_1,\ldots, \sigma_n$ the dual basis of $\mathcal E^\vee(U)$. By construction of $\mathbb{E}(\mathcal E)$, the coordinate ring on $\pi^{-1}(U)$ is the polynomial ring 
$$ \mathcal O_{\mathbb{E}(\mathcal E)}(\pi^{-1}(U))  = \mathcal O_X(U)[\overline{\sigma_1}, \ldots, \overline{\sigma_n}].$$ 
Hence, there is a unique derivation $\mathbb{E}(\nabla_v)$ on $\mathcal O_{\mathbb{E}(\mathcal E)}(\pi^{-1}(U))$ extending the derivation induced by $v$ on $\mathcal O_X(U)$ and such that 
$$ \mathbb{E}(\nabla_v)(\overline{\sigma_i}) = \overline{\nabla_v^\vee(\sigma_i)}.$$
This already implies uniqueness of the algebraic vector field structure on $\pi^{-1}(U)$. If $\sigma \in \mathcal E(U)$ is any section written as $\sigma = \sum_{i = 1}^n f_i \sigma_i$, we see that:  
\begin{eqnarray*} 
\mathbb{E}(\nabla_v)(\overline{\sigma}) & = & \sum_{i = 1}^n (\delta_v(f_i) \overline{\sigma_i} + f_i \mathbb{E}(\nabla_v)(\overline{\sigma_i})) \\
& = & \sum_{i = 1}^n (\delta_v(f_i) \overline{\sigma_i} + f_i \overline{\nabla_v^\vee(\sigma_i)} \\
& = & \overline{\nabla_v(\sum_{i = 1}^n f_i \sigma_i)} = \overline{\nabla_v(\sigma)}.
\end{eqnarray*}
Hence, the construction produces a unique algebraic vector field structure on $\pi^{-1}(U)$ with the desired properties. 

To define the algebraic vector field structure on $p^{-1}(U)$, note that the previous derivation 
$$ \mathbb E(\nabla_v) : \mathcal O_X(U)[\overline{\sigma_1}, \ldots, \overline{\sigma_n}] \rightarrow \mathcal O_X(U)[\overline{\sigma_1}, \ldots, \overline{\sigma_n}] $$ 
is homogeneous of degree zero. It follows that $\mathbb{E}(\nabla_v)$ defines a unique algebraic vector field structure $\P(\nabla_v)$ on $p^{-1}(U)$ such that 
$$ (\pi^{-1}(U) \setminus \lbrace 0\text{-section} \rbrace,  \mathbb{E}(\nabla_v)) \rightarrow (p^{-1}(U)) , \P(\nabla_v))$$
is a morphism of algebraic vector fields. Indeed, on the open set $U_i$ defined by $\sigma_i \neq 0$, we just take: 
$$ \P(\nabla_v)(\sigma_j/\sigma_i) = \frac {\mathbb{E}(\nabla_v)(\sigma_i) \sigma_j - \mathbb{E}(\nabla_v)(\sigma_j) \sigma_i} {\sigma_i^2} \text{ for }j \neq i $$ 
which can be written as a function of the $\sigma_j/\sigma_i$ for $j \neq i$ by homogeneity of $\mathbb{E}(\nabla_v)$. Finally, the uniqueness of the constructions implies that this local constructions glue together to define global algebraic vector fields $\mathbb{E}(\nabla_v)$ on $\mathbb{E}(\mathcal E)$ and $\mathbb{P}(\nabla_v)$ on $\mathbb{P}(\mathcal E)$ respectively. 
\end{proof} 
\begin{definition} 
Let $(X,v)$ be an algebraic vector field. Consider the standard partial connection $\nabla^{\mathrm{st}}_v$ on $\Omega^1_X$ of Example \ref{standard} dual of the connection 
$\xi \mapsto [v, \xi]$.
The {\em{first prolongation of $(X,v)$}} and the {\em first projective prolongation of $(X,v)$} are respectively the algebraic vector fields $(T^\ast X, \mathbb{E}(\nabla^{\mathrm{st}}_v))$ and $(\P(T^\ast X), \P(\nabla^{\mathrm{st}}_v))$ associated by Proposition \ref{totalspace} to the standard connection $\nabla^{\mathrm{st}}_v$. They are denoted respectively
$(T^\ast X, v^{[1]}), (\P(T^\ast X), \P(v^{[1]})) $ in the rest of the paper.  
\end{definition}

{\remark[Local expression]\label{local-expression} Let $(U;x_1,\ldots, x_n)$ be an affine open set with étale coordinates of $X$. We can write the vector field $v$ on $U$ as
$$ v_{\mid U} = \sum_{i = 1}^n a_i \cdot \partial_{x_i}.$$
Setting $y_i = \overline{\partial_{x_i}}$, on the affine open set $(\pi^{-1}(U); x_1,\ldots, x_n,y_1,\ldots, y_n)$ of $T^\ast X$, we have 

$$ v^{[1]}_{\mid \pi^{-1}(U)} = \sum_{i = 1}^n a_i \cdot \partial_{x_i} - \sum_{i = 1}^n \Big( \sum_{j = 1}^n \partial_{x_i}(a_j)y_j \Big) \cdot \partial_{y_i}.$$
In particular, the vector field $v^{[1]}$ is always linear in the fibers although we did not require it explicitly in Proposition \ref{totalspace}. 
} 

%

\subsection{Invariant horizontal subvarieties of the first prolongation} From now on, $p : \mathbb{P}(T^\ast X) \rightarrow X$ always denotes the canonical projection of the projective cotangent bundle. If $(X,v)$ is an algebraic vector field then 
$$p : (\mathbb{P}(T^\ast X),\P(v^{[1]})) \rightarrow (X,v)$$
is a morphism of algebraic vector fields by Proposition \ref{totalspace}.

\begin{definition} 
A subvariety of $\P(T^\ast X)$ is called \textit{horizontal} if it is closed, equidimensional and all of its irreducible components project surjectively (or equivalently dominantly) on $X$. If $(X,v)$ is an algebraic vector field, {\em an invariant horizontal subvariety of $\P(T^\ast X)$} is a horizontal subvariety invariant under the vector field $\P(v^{[1]})$.
\end{definition}

The reason why we allow reducible closed subvarieties will become apparent in the next section to define a pullback operation by generically finite morphisms, see Example \ref{example-notirreducible}. 

\begin{lemma} \label{lemma-horizontal-extension}
Let $U$ be a dense open set of $X$. The restriction map
$$Z \mapsto Z_{\mid U} := Z \cap p^{-1}(U)$$
which to a horizontal subvariety of $\P(T^\ast X)$ associates a horizontal subvariety of $\P(T^\ast U)$ is a bijection. Moreover, if $(X,v)$ is an algebraic vector field  then  invariant horizontal subvarieties of $\P(T^\ast X)$ correspond to invariant horizontal subvarieties of $\P(T^\ast U)$ under this bijection.
\end{lemma}  

\begin{proof} 
The inverse is given by $Z \mapsto \overline{Z}$ where $\overline{Z}$ denotes the Zariski-closure of $Z$ in $\P(T^\ast X)$ since if $H$ is a horizontal variety of $\P(T^\ast X)$ then $H \cap p^{-1}(U)$ is Zariski-dense in $H$. The moreover part follows from Remark \ref{remark-invariant-restriction}.  
\end{proof}

\begin{notation} 
If $Z$ is a horizontal subvariety of $\P(T^\ast X)$, we denote by $p_Z: Z \rightarrow X$
the restriction of the canonical projection to $Z$. The {\em relative dimension of $Z$ over $X$} is the dimension of the generic fiber of $p_Z$, that is,  $\mathrm{dim}(Z) - \mathrm{dim}(X)$. 

\end{notation}

\begin{proposition}\label{proposition-distributions-horizontal} 
We have an injective map: 
$$
\begin{cases} 
 \text{distributions on } X \\
 \text{ of rank  }< n 
 \end{cases} \rightarrow \begin{cases} 
  \text{irreducible horizontal} \\
 \text{subvarieties of } \P(T^\ast X) 
 \end{cases} $$
denoted $\mathcal D \mapsto Z(\mathcal D)$ with the following properties: 
\begin{itemize} 
\item[(i)] The map is preserved after restriction on both sides: if $U$ is a dense open set, then $Z(\mathcal {D}_{\mid U}) = Z(\mathcal D)_{\mid U}$.  
\item[(ii)] If $\mathcal D$ is a nonsingular distribution then $Z(\mathcal D) = \P(N^\ast \mathcal D)$
is the projectivization of the conormal bundle $N^\ast \mathcal D$ of the distribution. 
\item[(iii)] If $\mathcal D$ is a distribution of rank $r$ then the relative dimension of $Z(\mathcal D)$ is $n - r - 1$.
\end{itemize} 
Moreover, if $(X,v)$ is an algebraic vector field and $\mathcal D$ is a distribution on $X$ then 
\begin{itemize} 
\item[(iv)] the distribution $\mathcal D$ is an invariant distribution if and only if  $Z(\mathcal D)$ is an invariant horizontal subvariety of $\P(T^\ast X)$. 
\end{itemize} 
\end{proposition}

\begin{proof}
Let $\mathcal D$ be a distribution of rank $r < n$ on $X$. Denote by $V(\mathcal D)$ the complement of the singular locus of $\mathcal D$ in $X$. We first cover $V(\mathcal D)$ by affine open sets $U$ endowed with étale coordinates $(x_1,\ldots, x_n)$ so that 
$$\mathcal O_X(T^\ast U) \simeq \mathcal O_X(U)[\overline{\partial_{x_1}}, \ldots, \overline{\partial_{x_n}}].$$
Since the restriction of $\mathcal D$ to $U$ is smooth, $\Theta_\mathcal D(U)$ is a free $\mathcal O_X(U)$-module and we can fix a basis $\xi_1,\ldots, \xi_r$ of $\Theta_\mathcal D(U)$. Denote by $\mathcal I(U)$ the ideal of  $\mathcal O_{T^\ast X}(T^\ast U)$ generated by $\overline{\xi_1},\ldots, \overline{\xi_r}$ and set 
$$N^\ast(\mathcal D)_{\mid U} = Z(\mathcal I(U)) \subset T^\ast U  $$
By definition, for every $x \in U$ the fiber of 
$N^\ast(\mathcal D)_{\mid U}$ over $x$ given by 
$$N^\ast(\mathcal D)_x = \lbrace \omega \in TX^\ast_{x} \mid \omega(\xi_i(x)) = 0 \text{ for } i = 1,\ldots r \rbrace $$ 
is a subvector space of $T^\ast U_x$ of constant rank $n-r$.  Hence, $N^\ast(\mathcal D)_{\mid U}$ is a vector subbundle of $T^\ast U$ of rank $n-r$ and its projectivization
$$Z(\mathcal D)_{\mid U} = \P(N^\ast(\mathcal D)_{\mid U}) \subset  \P(T^\ast U)$$
is an irreducible horizontal subvariety of $\P(T^\ast U)$ of dimension $n - r - 1$. This construction does not depend on the choice of the basis $\overline{\xi_1},\ldots, \overline{\xi_r}$ so that the $N^\ast(\mathcal D)_{\mid U}$ and the $Z(\mathcal D)_{\mid U}$ glue together to form a global subvector bundle $N^\ast(\mathcal D)_{\mathrm{reg}}$ of $T^\ast V(\mathcal D)$ and its projectivization $Z(\mathcal D)_{\mathrm{reg}}$ over {\em the smooth locus $V(\mathcal D)$ of $\mathcal D$}. Finally set 
$$ Z(\mathcal D) = \overline{Z(\mathcal D)_{\mathrm{reg}}} \subset \P(T^\ast X)$$
which is an irreducible horizontal closed subvariety of $\P(T^\ast X)$ by Lemma \ref{lemma-horizontal-extension}. It remains to check that the properties of Proposition \ref{proposition-distributions-horizontal} are fulfilled: 
\begin{itemize} 
\item \textit{injectivity}: if $\mathcal D$ is a distribution then for every $x \in V(\mathcal D)$,
$$ T\mathcal D_x = \lbrace v \in TX_x \mid \omega(v) = 0 \text{ for all } \omega \in T^\ast X_x \text{  with }  [\omega] \in Z(\mathcal D)\rbrace$$
\end{itemize} 
Hence, the restriction of $\mathcal D$ to $V(\mathcal D)$ is determined by $Z(\mathcal D)$ and so is $\mathcal D$ by Fact \ref{fact-extension distribution}. 

\begin{itemize}
\item (i) follows from Lemma \ref{lemma-horizontal-extension}, (ii) follows from the construction and (iii) is a consequence of (ii) and (i). 

\item Finally to show (iv), consider $\mathcal D$ a distribution on $X$. Using the previous properties, we may replace $X$ by an affine open set $U$ with étale coordinates  and assume that $\mathcal D$ is smooth. Denote by $\xi_1,\ldots,\xi_r \in \Theta_D(U)$ a system of generators of $\mathcal D$. As previously, $\overline{\xi_1},\ldots, \overline{\xi_r}$ generate the ideal $\mathcal I(U)$ in $\mathcal O_{T^\ast X}(T^\ast U)$ defining $N^\ast \mathcal D$ and $Z(\mathcal D) = \mathbb{P}(N^\ast \mathcal D)$. 
\end{itemize}
If $\mathcal D$ is invariant then we can write for $i = 1,\ldots, r$
$[v,\xi_i] = \sum_{j = 1}^r \lambda_{i,j} \xi_j$ for some $\lambda_{i,j} \in \mathcal O_X(U)$. 
By Proposition \ref{totalspace}, the derivation $D^{[1]}$ induced by $v^{[1]}$ on $\mathcal O_{T^\ast X}(T^\ast U)$ satisfies 
$$D^{[1]}(\overline{\xi_i}) = \overline{[v,\xi_i]} = \sum_{j = 1}^s \lambda_{i,j} \overline{\xi_j} \in \mathcal I(U).$$ 
It follows that the ideal $\mathcal I(U)$ is an invariant ideal. This means that $N^\ast \mathcal D$ is $v^{[1]}$-invariant so that $Z(\mathcal D)$ is $\mathbb{P}(v^{[1]})$-invariant as required. 

Conversely,  assume that $Z(\mathcal D)$ is $\P(v^{[1]})$-invariant.  It follows from the commutative diagram of Proposition \ref{totalspace} that $N^\ast \mathcal D$ is $v^{[1]}$-invariant so that the ideal $\mathcal I(U)$ is invariant. Hence we can write
$$\overline{[v,\xi_i]} = D^{[1]}(\xi_i) = \sum_{j = 1}^r \lambda_{i,j} \cdot \overline{v_i}$$ 
for some functions $\lambda_{i,j}$ on $T^\ast X$. The homogeneity of the derivation $D^{[1]}$ imply that actually the $\lambda_{i,j}$ homogeneous of degree zero, that is, functions on $X$. It follows that 
$$ \overline{[v,\xi_i]} = \overline{\sum_{j = 1}^r \lambda_{i,j} \cdot v_i} \text{ and hence } [v,\xi_i] = \sum_{j = 1}^r \lambda_{i,j} \cdot v_i \in \Theta_\mathcal D(U).$$
This means that $\mathcal D$ is an invariant distribution as required.
\end{proof} 

\subsection{Pullback and pushforward under generically finite morphisms} Let $\phi: X \rightarrow X'$ be an étale morphism of algebraic varieties. Since the differential is locally invertible at every point, we obtain a commutative diagram 
$$ \xymatrix{ \P(T^\ast X) \ar[d] \ar[r] & \P(T^\ast X') \ar[d] \\ X \ar[r] & X'}$$
which induces linear isomorphisms on the fibers. 

\begin{lemma}\label{lemma-prolongation}
Let $\phi: (X,v) \rightarrow (X',w)$ be an étale morphism of algebraic vector fields. Then we have a commutative diagram of algebraic vector fields: 
$$ \xymatrix{ (\P(T^\ast X),\P(v^{[1]})) \ar[d] \ar[r] & (\P(T^\ast X'),\P(w^{[1]})) \ar[d] \\ (X,v) \ar[r] & (X',w)}$$
\end{lemma}

\begin{proof} 
By Proposition \ref{totalspace}, it is enough to show that 
$$ \xymatrix{ (T^\ast X,v^{[1]}) \ar[d] \ar[r] & (T^\ast X',w^{[1]}) \ar[d] \\ (X,v) \ar[r] & (X',w)}$$
is a commutative diagram of algebraic vector fields. Let $(U;x'_1,\ldots, x'_n)$ be an affine open set with étale coordinates of $X'$ and denote $y'_i = \partial_{x_i}'$ so that
$$(\pi^{-1}(U), x'_1,\ldots, x'_n,y'_1,\ldots, y'_n)$$ is an open set with étale coordinates on $T^\ast X'$. Set $V = \phi^{-1}(U)$, $x_i = \phi \circ x'_i$ and $y_i = \overline{\partial_{x_i}}$ so that 
$$(\pi^{-1}(V), x_1,\ldots, x_n,y_1,\ldots, y_n)$$
is an open set with étale coordinates of $T^\ast X$.  By construction, these étale coordinates relate through the morphism $\phi_1: T^\ast X \rightarrow T^\ast X'$  by 
$$x_i = x'_i \circ \phi_1 \text{ and } y_i = y'_i \circ \phi_1 $$
Now write $w = \sum_{i = 1}^n \lambda_i \cdot \partial_{x_i'}$ so that
$v = \sum_{i = 1}^n (\lambda_i \circ \phi) \cdot \partial_{x_i}$.
Using the Remark \ref{local-expression}, we compute that

\begin{eqnarray*} 
v^{[1]} & = & \sum_{i = 1}^n (\lambda_i \circ \phi) \cdot \partial_{x_i} - \sum_{i = 1}^n \Big( \sum_{j = 1}^n \partial_{x_i}(\lambda_j \circ \phi)y_j \Big) \cdot \partial_{y_i} \\ 
& = & \sum_{i = 1}^n (\lambda_i \circ \phi) \cdot \partial_{x_i} - \sum_{i = 1}^n \Big( \sum_{j = 1}^n (\partial_{x'_i}(\lambda_j) \circ \phi) y_j \Big) \cdot \partial_{y_i} \\
 & = & \sum_{i = 1}^n (\lambda_i \circ \phi_1) \partial_{x_i} - \sum_{i = 1}^n \Big( \sum_{j = 1}^n (\partial_{x'_i}(\lambda_j)y'_j)  \Big) \circ \phi_1 \cdot \partial_{y_i}
\end{eqnarray*}
where, on the last line, the $\lambda_i$ are seen as functions on $T^\ast X'$ using the embedding $\mathcal O_{X'}(X') \rightarrow \mathcal O_{T^\ast X'} (T^\ast X')$. Hence, $v^{[1]}$ is the lift of $w^{[1]}$ as required.
\end{proof}

\begin{proposition} \label{proposition-pull-back-push forward}
Let $\phi: X \dashrightarrow Y$ be a generically finite dominant rational morphism. We have pullback and pushforward  maps
$$ Z \mapsto \phi^\ast Z \text{ and } Z \mapsto \phi_\ast Z  $$ 
sending respectively a horizontal subvarieties of $\P(T^\ast Y)$ (resp. of $\P(T^\ast X)$) to horizontal subvarieties of $\P(T^\ast X)$ (resp. of $\P(T^\ast Y)$) which satisfy the following properties: 
\begin{itemize} 
\item[(i)] If $\mathcal D$ is a distribution on $Y$ then $Z(\phi^\ast \mathcal D) = \phi^\ast Z(\mathcal D)$.

\item[(ii)] For every horizontal subvarieties $Z$ of $\P(T^\ast X)$, $\phi^\ast \phi_\ast Z = Z$.
In particular, $\phi^\ast$ is injective and $\phi_\ast$ is surjective. 

\end{itemize} 
Moreover, if $\phi: (X,v) \dashrightarrow (Y,w)$ is a morphism of algebraic vector fields then 
\begin{itemize} 
\item[(iii)] the maps $Z \mapsto \phi_\ast Z$ and $Z \mapsto \phi^\ast Z$ send invariant horizontal subvarieties to invariant horizonal subvarieties. 
\end{itemize} 
\end{proposition}

\begin{proof} 
By Lemma \ref{lemma-horizontal-extension}, it is enough to define the maps $Z \mapsto \phi_\ast Z$ and $Z \mapsto \phi^\ast Z$ after restricting $X$ and $Y$ to dense open sets. So we may assume that $\phi: X \rightarrow Y$
is étale and proper. Since the differential $d\phi$ is locally invertible at every point, we obtain a commutative diagram: 
$$ \xymatrix{ \P(T^\ast X) \ar[d]_{p_X} \ar[r]^{\phi_1} & \P(T^\ast Y) \ar[d]^{p_Y} \\ X \ar[r] & Y}$$
Since $\phi$ and $p_X$ are proper, the composition $p_Y \circ \phi_1$ is also proper. It then follows from the fact that $p_Y$ is separated that $\phi_1$ is a proper morphism. See \cite[II, Corollary 4.8]{Hartshorne}. Hence, the morphism $\phi_1$ is an étale cover too and we define 
$$\phi^\ast Z := \phi_1^{-1}(Z) \text{ and } \phi_\ast Z := \phi_1(Z).$$ 
This correspondence sends horizontal  subvarieties to horizontal  subvarieties: indeed, if $Z$ is horizontal in $\P(T^\ast Y)$ then the image under $\phi_1$ of an irreducible component $T$ of $\phi^\ast Z$ is an irreducible component of $Z$ and it follows that 
$\phi \circ p_X(T) = Y$. Since $X$ is irreducible, the only closed subvariety of $X$ projecting generically on $Y$ is $X$ so that $p_X(T) = X$. Conversely, if $Z$ is a  horizontal in $\P(T^\ast X)$, the irreducible components $H_Y$ of $\phi_\ast Z$ are the images under $\phi_1$ of the irreducible components $H_X$ of $Z$. It follows that 
$$p_Y(H_Y) = p_Y \circ \phi_1(H_Y) = \phi \circ p_X(H_X) = \phi(X) = Y.$$
We conclude that the maps
$Z \mapsto \phi^\ast Z  \text{ and } Z \mapsto \phi_\ast Z$
are well-defined. It remains to check that they satisfy the properties listed in the proposition: 

(i). Let $\mathcal D$ be a distribution on $Y$ of rank $r < n$. By Lemma \ref{lemma-horizontal-extension}, we may assume that $\phi$ is étale and proper and that $\mathcal D$ is nonsingular. Under these assumptions, $Z(\mathcal D) = \P(N^\ast (\mathcal D))$  and (i) follows from the equality: 
$$N^\ast (\phi^\ast \mathcal D) = \phi^\ast (N^\ast \mathcal D)$$
valid for the pullback $\phi^\ast \mathcal D$ of any distribution $\mathcal D$ by an étale morphism. 

(ii). follows from the the construction and Lemma \ref{lemma-horizontal-extension} and (iii) follows from Lemma \ref{lemma-prolongation}. This concludes the proof of the proposition.
\end{proof}

{\example \label{example-notirreducible} If $\phi: X \dashrightarrow Y$ and $Z$ is an irreducible horizontal subvariety of $\P(T^\ast Y)$ then $\phi^\ast Z$ might not to be irreducible. To obtain an example, consider $C$ a projective curve, a natural number $n \geq 2$ and set $X = C^n$ and $Y = S^n C$ and $\phi : X \rightarrow Y$ given by 
$$(x_1,\ldots, x_n) \mapsto \lbrace x_1,\ldots, x_n \rbrace$$
For every $i = 1,\ldots, n$, denote by ${\mathcal F_i}$ the foliation tangent to the projection $\pi_i: C^n \rightarrow C$ on the $i^{th}$ factor. Then for every $i \leq n$, 
$$\phi^\ast \phi_\ast Z(\mathcal F_i) = Z(\mathcal F_1) \cup \ldots \cup Z(\mathcal F_n)$$
is not irreducible but $\phi_\ast Z(\mathcal F_i)$
always is since $Z(\mathcal F_i)$ is irreducible. This example justifies our choices in the definition of a closed horizontal subvariety. }

\subsection{A minimality criterion based on the first projective prolongation} Fix $(X,v)$ an algebraic vector field of dimension $n \geq 2$ over some algebraically closed field of characteristic zero.

\begin{definition} \label{definition-canonical-invariant-hypersurface}
Assume that $(X,v)$ is not the trivial vector field on $X$. The {\em{canonical invariant horizontal hypersurface}} of $\P(T^\ast X)$ is the horizontal subvariety $ H(v):= Z(\mathcal F(v))$ associated to the rank one foliation $\mathcal F(v)$ tangent to the vector field $v$. By Proposition \ref{proposition-distributions-horizontal}, $H(v)$ is a $\mathbb{P}(v^{[1]})$-invariant horizontal hypersurface of $\P(T^\ast X)$. 
\end{definition} 

\begin{definition}  \label{definition-firstprolongationirreducible}
We say that {\em the first prolongation of $(X,v)$ is quasi-irreducible} if $(X,v)$ is not the trivial vector field on $X$ and if $H(v)$ is the unique proper horizontal subvariety of $\P(T^\ast X)$ which is $\P(v^{[1]})$-invariant. 
\end{definition}

\begin{remark} \label{remark-quasiminimal} Lemma \ref{lemma-horizontal-extension} ensures that if $U$ is a dense open set of $X$ then the first prolongation of $(X,v)$ is quasi-irreducible  if and only if the first prolongation of $(U,v_{\mid U})$ is. Hence ``having a quasi-minimal first prolongation'' is a property of the generic point of $(X,v)$.
\end{remark}

\begin{lemma} \label{finitetofinite correspondence}
Assume that $(X_1,v_1)$ and $(X_2,v_2)$ are algebraic vector fields in generically finite-to-finite correspondence. Then the first prolongation of $(X_1,v_1)$ is quasi-irreducible if and only if the first prolongation of $(X_2,v_2)$ is. 
\end{lemma} 

Recall that $(X_1,v_1)$ and $(X_2,v_2)$ are in generically finite-to-finite correspondence if there exists an irreducible closed invariant subvariety $Z$ of $(X_1\times X_2, v_1 \times v_2)$ such that the two projections
$\pi_{i \mid Z}: Z \rightarrow X_i $
are dominant and generically finite. 
\begin{proof} 
It suffices to show that the conclusion of the lemma holds if $\phi : (X_1,v_1) \dashrightarrow (X_2,v_2)$ is a generically finite dominant morphism of algebraic vector fields.

Assume first that the first prolongation of  $(X_1,v_1)$ is quasi-irreducible and consider $Z$ an invariant proper horizontal subvariety of $(\P(T^\ast X_2), \P(v^{[1]}_2))$. By Proposition \ref{proposition-pull-back-push forward}, $\phi^\ast Z$ is an invariant proper horizontal subvariety of $\P(T^\ast X_1)$ and hence $\phi^\ast Z = Z(\mathcal F(v')) = \phi^\ast Z(\mathcal F(v)).$
Since by Proposition \ref{proposition-pull-back-push forward}, $\phi^\ast$ is injective, we conclude that $Z$ is the canonical invariant hypersurface and hence that the first prolongation of $(X_2,v_2)$ is quasi-irreducible.

Conversely, assume that the first prolongation of $(X_2,v_2)$ is quasi-irreducible and consider $Z$ an invariant proper horizontal subvariety of  $(\P(T^\ast X_1),v_1^{[1]})$. By Proposition \ref{proposition-pull-back-push forward}, $\phi_\ast Z$ is an invariant proper horizontal subvariety of $\P(T^\ast X_2)$ and hence $\phi_\ast Z = Z(\mathcal F(v))$. 
It follows that $Z$ is an irreducible component of 
$\phi^\ast \phi_\ast Z = Z(\mathcal F(v'))$. Since $Z(\mathcal F(v'))$ is irreducible by Proposition \ref{proposition-distributions-horizontal}, we have that $Z  = Z(\mathcal F(v'))$ so that the first prolongation of $(X_1,v_1)$ is quasi-irreducible. 
 \end{proof} 
\begin{theorem}[Theorem D] \label{theorem-minimality}
Let $(X,v)$ be an algebraic vector field of dimension $\geq 2$.  Assume that
\begin{center} 
$(\ast)$: either $\mathrm{dim}(X) \geq 3$ or $\mathrm{dim}(X) = 2$ and $(X,v)$ does not admit a nontrivial rational integral. 
\end{center}
If the first prolongation of $X$ is quasi-irreducible then the generic type of $(X,v)$ is minimal. 
\end{theorem} 

Theorem \ref{theorem-minimality} will be deduced from Fact \ref{fact-Jaoui-Moosa} from the introduction. Recall also the terminology of algebraic factor of an algebraic vector field from Definition \ref{definition-noalgebraicfactors}.

\begin{lemma}[under the assumption $(\ast)$]\label{lemma-rational factors}
Assume that $(X,v)$ admits a nontrival algebraic factor. Then the first prolongation of $(X,v)$ is not quasi-irreducible. 
\end{lemma} 

\begin{proof} 
By Lemma \ref{finitetofinite correspondence}, it is enough to prove the lemma when $(X,v)$ admits a nontrivial rational factor. For the sake of a contradiction, assume that the first prolongation of $(X,v)$ is quasi-irreducible and that $\phi: (X,v) \dashrightarrow (Y,w)$ is a nontrivial rational factor of $(X,v)$ satisfying  $0 < \mathrm{dim}(Y) < \mathrm{dim}(X)$. By Proposition \ref{proposition rational factors to invariant foliations}, the foliation $\mathcal F_\phi$, tangent to the fibers of $\phi$, is an invariant foliation and hence by Proposition \ref{proposition-distributions-horizontal}, 
$ Z(\mathcal F_\phi)$
is an invariant proper horizontal subvariety of $\P(T^\ast X)$. By quasi-irreducibility of the first prolongation of $(X,v)$, we conclude that:  
$$ Z(\mathcal F_\phi) = Z(\mathcal F(v)) \text{ and hence } \mathcal F_\phi = \mathcal F(v).$$ 
It follows that $\mathrm{dim}(Y) = \mathrm{dim}(X) - 1$ and that $v$ lies in the kernel of $d\phi$ by definition of the foliation $\mathcal F_\phi$ (see Example \ref{example-algebraically-integrable}) so that:
\begin{center}
$(\ast \ast)$: every rational factor of $(X,v)$ is of the form
$\phi: (X,v) \dashrightarrow (Y,0)$
where $\mathrm{dim}(Y) = \mathrm{dim}(X) - 1$.
\end{center}

Using $(\ast)$ and $(\ast \ast)$, we obtain a contradiction:

\begin{itemize} 
\item either $n = \mathrm{dim}(X) \geq 3$. Then $Y$ has dimension greater than two. Pick $f$ a nonconstant rational function on $Y$ so that 
$$f \circ \phi: (X,v) \dashrightarrow (Y,0) \dashrightarrow (\mathbb{P}^1,0)$$
is a rational factor of dimension one contradicting $(\ast \ast)$. 

\item or $n = \mathrm{dim}(X) = 2$ and $(X,v)$ does not have nontrivial rational integral. Pick $f$ a nonconstant rational function on $Y$ so that the rational factor 
$$ f\circ \phi: (X,v) \dashrightarrow (\mathbb{P}^1,0)$$
is a nontrivial rational integral of $(X,v)$ leading again to a contradiction.
\end{itemize} 
Since we obtained a contradiction in both cases, this concludes the proof of the lemma.
\end{proof} 

\begin{lemma} \label{lemma-commutaitve invariant}
Assume that $v$ is a translation-invariant vector field on a commutative group $G$ of dimension $> 1$. Then the first prolongation of $(X,v)$ is not quasi-irreducible. 
\end{lemma} 

\begin{proof} 
Since $G$ is a commutative group the Lie algebra $\mathrm{Lie}(G)$ formed by the invariant vector fields is commutative and we can consider $v = e_1,\ldots, e_n$ a basis of $\mathrm{Lie}(G)$ with $n > 1$. 
Consider the foliation $\mathcal F_i$ generated by $e_i$. For all $i \leq n$, $[v,e_i] = 0$ and hence the foliation $\mathcal F_i$ in an invariant foliation of $(X,v)$. By Proposition \ref{proposition-distributions-horizontal}, $Z(\mathcal F_1),\ldots, Z(\mathcal F_n)$ are $n$ distinct invariant horizontal hypersurfaces of $\P(T^\ast X)$. Since $n \geq 2$, the first prolongation of  $(X,v)$ is not quasi-irreducible. 
\end{proof}

\begin{proof}[Proof of Theorem \ref{theorem-minimality}]
Working by contraposition, assume that the generic type of $(X,v)$ is not minimal. By Fact \ref{fact-Jaoui-Moosa}, only two possibilities can occur: either $(X,v)$ admits a nontrivial algebraic factor, or there exists a generically finite-to-finite correspondence between $(X,v)$ and $(A,v_A)$ where $A$ is a simple Abelian variety of dimension $\geq 2$ and $v_A$ is a translation-invariant vector field on $A$. The two previous lemmas show that, in any case, the first prolongation of $(X,v)$ is not quasi-irreducible.
\end{proof}

The following example shows that the converse of Theorem \ref{theorem-minimality} does not hold: there exist an algebraic vector field of dimension three with a minimal generic type but such that the projective first prolongation $(X,v)$ is not quasi-minimal. We do not know a similar example in dimension two so it might be the case that in dimension two, the implication of Theorem \ref{theorem-minimality} is actually an equivalence.

\begin{example} \label{example-Schwarzian}
Consider a Schwarzian differential equation of the form 
$$ \Big(\frac{y''}{y'}\Big)' - \frac 1 2 \Big(\frac{y''}{y'}\Big)^2 + (y')^2 R(y) = 0$$  where $R(y) \in \mathbb{C}(y)$ with set of poles $S$. We can write this equation as an algebraic vector field 
$$v = v_1 =  y' \partial_y + y'' \partial_y' + \Big( - R(y) (y')^3 + \frac 3 2 \frac {(y'')^2}{y'} \Big)  \partial_{y''} $$
on the open set $U$ of the affine space with coordinates $(y,y',y'')$ defined by $y' \neq 0$ and $y \notin S$. By \cite{Blazquez-Casale}, there are two other vector fields $v_2 \text{ and } v_3$ on $U$ namely 
$$ v_2 = y' \partial_{y'} + 2y'' \partial_{y''} \text{ and } v_3 = 2y' \partial_{y''}$$ 
such that
\begin{itemize} 
\item the vector field $v_1,v_2,v_3$ define a trivilization of the tangent bundle on $U$ 
\item the vector fields $v_1,v_2,v_3$ generates a Lie algebra isomorphic to $\mathfrak{sl}_2(\mathbb{C})$ and more precisely
$$[v_1,v_2] = v_1, [v_1,v_3] =  2v_2 \text{ and } [v_2,v_3] = - v_3.$$ 

\end{itemize} 

It follows easily that the distribution $\mathcal G = \mathcal D(v_1,v_2)$ generated by the vector field $v_1$ and $v_2$ is a foliation. Since $\mathcal F(v_1) \subset \mathcal G$, this foliation is an invariant foliation of $(X,v)$ by Proposition \ref{proposition rational factors to invariant foliations}.  It follows that the canonical invariant hypersurface of  $\mathbb{P}(T^\ast U,v^{[1]})$ always has an invariant rational section and hence that the first prolongation of $(U,v)$ is never quasi-irreducible. On the other hand, \cite{Casale-Freitag-Nagloo} shows that in many cases, the generic type of this equation is (strongly) minimal. 
\end{example}

\subsection{A Galois-theoretic version of Theorem \ref{theorem-minimality}} We finally describe a Galois-theoretic version of Theorem \ref{theorem-minimality}. For that matter, we assume that $(X,v)$ is an algebraic vector field without nontrivial rational integral. This corresponds to the classical assumption in differential Galois theory that the differential field $(k(X),\delta_v)$ admits no new constants, that is, that its field of constants is equal to $k$.

\begin{definition} \label{definition-Galoisgroup}
Let $(\mathcal E, \nabla_v)$ be a coherent sheaf over $(X,v)$ endowed with a partial connection $\nabla_v$ along $v$. The generic stalk of  $(\mathcal E, \nabla_v)$ is a differential module over $(k(X),\delta_v)$ and we define the {\em Galois group of $(\mathcal E, \nabla_v)$} as 
$$\Gal((\mathcal E,\nabla_v)/(X,v)) = \Gal((\mathcal E,\nabla_v)_\eta/k(X)^{alg})$$
where $\eta$ is the generic point of $(X,v)$. The Picard-Vessiot differential Galois theory provides, together with the Galois group, a faithful representation: 
$$\rho_\mathcal E: \Gal((\mathcal E,\nabla_v)/(X,v)) \rightarrow \GL_n(k)$$ 
which is well-defined up to conjugation in $\GL_n(k)$. In model-theoretic terms, it can be described as follows. Denote by $K = (K,\delta)$ the differential closure of $(k(X),\delta_v)$ which has constant field $k$. The solution set $V$ in $K$ of the differential module $(\mathcal E,\nabla_v)_\eta$ is a $k$-vector space of dimension $n$ and we have a natural $k(X)^{alg}$-definable action  
$$(\mathrm{Aut}_{k(X)^{alg}}(V/k), V)$$ 
where $\mathrm{Aut}_{k(X)^{alg}}(V/k)$ denotes the binding group of $V$ relatively to the constants, that is, the group of permutations of $V$ induced by automorphims of $K$ fixing $k(X)^{alg}$ pointwise. The fundamental theorem linking binding groups to Galois groups \citep[Appendix B]{Hrushovski-Galois} expresses that every basis $\mathcal B = (v_1,\ldots, v_n)$ of $V$ defines a definable isomorphism 
\begin{equation} \label{equation-Galois-binding}  
(\mathrm{Aut}_{k(X)^{alg}}(V/k),V) \simeq_\mathcal B (\Gal((\mathcal E,\nabla_v)_\eta/k(X)^{alg}),k^n).
\end{equation} 
This construction produces the required representation of $\Gal((\mathcal E,\nabla_v)/(X,v))$ and different choices of the basis $\mathcal B$ give rise to representations on the right hand side conjugated by an element of $\GL_n(k)$.
\end{definition} 
Note also that Definition \ref{definition-Galoisgroup} implies that $\Gal((\mathcal E,\nabla_v)/(X,v))$ is always a {\em connected} linear algebraic group and that an alternative definition is: 
$$\Gal((\mathcal E,\nabla_v)/(X,v)) = \Gal^0((\mathcal E,\nabla_v)_\eta/k(X)).$$
See for example Lemma 2.1 of \cite{Jaoui-Moosa}.
Moreover, the Galois groups of $(\mathcal E,\nabla_v)$ and $(\mathcal E^\vee, \nabla_v^\vee)$ are naturally isomorphic. 

\begin{definition} 
Let $(X,v)$ be an algebraic vector field. The {\em Galois group of the first prolongation} of $(X,v)$ denoted $G_1(X,v)$ is the Galois group of the standard connection on $\Omega^1_X$ defined in Example \ref{standard}, that is,
$$G_1(X,v) := \Gal((\Omega^1_X, \nabla^{\mathrm{st}}_v)/(X,v)) $$
We also denote by $\rho_X: G_1(X,v) \rightarrow \GL_n(k)$
the representation defined by differential Galois theory (well-defined up to conjugation).
\end{definition}

\begin{notation} \label{example-standard}
For the rest of the section,  we fix the following notation: if $(X,v)$ is an algebraic vector field without nontrivial rational integral, 
\begin{itemize} 
\item we write $(M_X,\nabla_v^{\mathrm{st}})$ for the generic stalk of the partial connection $(\Omega^1_X, \nabla^{\mathrm{st}}_v)$. 
\end{itemize}
This differential module can be described explicitly as follows: fix $x_1,\ldots, x_n$ a transcendence basis of $k(X)$ and write
$$ v = \sum_{i= 1}^n a_i  \cdot \partial_{x_i} \text{ where } a_1,\ldots, a_n \in k(X).$$ 
It follows from Remark \ref{local-expression} that in the basis  $dx_1,\ldots,dx_n$, the differential module $(M_X,\nabla_v^{\mathrm{st}})$ can be written as the linear differential equation 
$$L(X,v): Y' = A Y \text{ where } A = (\partial_{x_i}{a_j})_{1 \leq i,j\leq n} \in \Mat_n(k(X)).$$ 

\begin{itemize} 
\item we write $K$ for the differential closure of $k(X)$, $(M_X,\nabla_v^{\mathrm{st}})^\sharp$ for the solution set of $(M_X,\nabla_v^{\mathrm{st}})$ in $K$ and $B_1(X,v) = \mathrm{Aut}_{k(X)^{alg}}((M_X,\nabla_v^{\mathrm{st}})^\sharp/k)$ for the binding group of this definable set over $k(X)^{alg}$ relatively to the constants. 
\end{itemize}
\end{notation} 

\begin{lemma} \label{lemma-Galois group}
Let $(X,v)$ be $(Y,w)$ be algebraic vector fields in generically finite-to-finite correspondence.  We have an isomorphism: 
$$(G_1(X,v),\rho_X) \simeq (G_1(Y,w),\rho_Y)$$
up to conjugation in$\GL_n(k)$.
\end{lemma}

\begin{proof} 
It follows easily from the proof of  Lemma \ref{lemma-prolongation} that if $(X,v) \dashrightarrow (Y,w)$ is a generically finite dominant rational morphism of algebraic vector fields then 
$$ (M_X, \nabla^{\mathrm{st}}_v) \simeq  (M_Y, \nabla^{\mathrm{st}}_w) \otimes_{k(Y)} k(X).$$ 
Hence, if $(X,v)$ and $(Y,w)$ are in generically finite-to-finite correspondence then 
$$ (M_X, \nabla^{\mathrm{st}}_v) \otimes_{k(X)} k(X)^{alg} \simeq (M_X, \nabla^{\mathrm{st}}_v) \otimes_{k(Y)} k(Y)^{alg}. $$
Hence, $(M_X, \nabla^{\mathrm{st}}_v)^\sharp$ and $(M_Y, \nabla^{\mathrm{st}}_w)^\sharp$ are gauge-conjugated over $k(X)^{alg}$ and have the same binding group. The lemma then follows from the identity (\ref{equation-Galois-binding}).
\end{proof} 

\begin{definition} 
Let $G$ be a group and $\rho: G \rightarrow \GL(V)$ be a representation of $G$ with $\mathrm{dim}(V) \geq 2$. We say that $\rho$ is {\em{quasi-irreducible}} if there exists a unique proper positive-dimensional vector subspace of $V$  which is $G$-invariant. 
\end{definition} 

\begin{example} \label{example-comparison}
To the canonical invariant hypersurface $H(v)$ of $\P(T^\ast X)$ corresponds a differential submodule $H_v$ of $(M_X,\nabla^{\mathrm{st}}_v)$ of dimension $n - 1$. It is the generic stalk of the conormal bundle $N^\ast \mathcal F$ of the foliation $\mathcal F$ tangent to $v$, that is, the $k(X)$-vector subspace of $M_X$ given by 
$H_v = \lbrace \omega \in M_X \mid \omega(v) = 0  \rbrace$.
Indeed, by (\ref{equation-dual}) and by anti-symmetry of the Lie bracket, we have that 
$$ \delta_v(\omega(v)) = \nabla^{\mathrm{st}}_v(\omega)(v) + \omega(\nabla_v^{\mathrm{st}}(v)) =  \nabla^{\mathrm{st}}_v(\omega)(v).$$ 
Hence, $\omega(v) = 0$ implies $\nabla_v^{\mathrm{st}}(\omega)(v) = 0$ so that $H_v$ is a differential submodule of $(M_X,\nabla^{\mathrm{st}}_v)$. It follows that, as soon as $\mathrm{dim}(X) \geq 2$, the representation $(G_1(X,v),\rho_X)$ is not irreducible. 

On the other hand, if the first prolongation of $(X,v)$ is quasi-irreducible in the sense of Definition \ref{definition-firstprolongationirreducible} then the representation $(G_1(X,v),\rho_X)$ is quasi-irreducible.  

Indeed, otherwise, the representation  $(B_1(X,v), (M_X, \nabla^{st}_v)^\sharp)$ is not quasi irreducible either and we can consider another $B_1(X,v)$-invariant subspace $W \neq H_v^\sharp$. $B_1(X,v)$-invariance means that $W = V^\sharp$ for some differential submodule $V$ of $(M_X,\nabla_v^{\mathrm{st}}) \otimes_{k(X)} k(X)^{alg}$. After replacing $(X,v)$ by a finite cover (which preserves quasi-irreducibility of the first prolongation by Lemma \ref{finitetofinite correspondence}), we may assume that $V$ is defined over $k(X)$. On some open set $U \subset X$, $V$ extends to a vector subbundle $E(V)$ of $T^\ast X$ such that $E(V)_\eta = V$ and which is $v^{[1]}$-invariant. This contradicts the quasi-irreducibility of the first prolongation of $(U,v_{\mid U})$ since $\P(E(V))$ is a $\P(v^{[1]})$-invariant horizontal subvariety distinct from the canonical hypersurface. 
\end{example}

\begin{theorem}[Galois-theoretic version of Theorem D]\label{theorem-C-Galois}
Let $(X,v)$ be an algebraic vector field of dimension $\geq 2$ without nontrivial rational integral. If the representation $(G_1(X,v), \rho_X)$
is quasi-irreducible, then the generic type of $(X,v)$ is minimal. 
\end{theorem}

\begin{proof} 
Working by contraposition, assume that the generic type of $(X,v)$ is not minimal. By Fact \ref{fact-Jaoui-Moosa}, only two possibilities can occur: 
\begin{itemize} 
\item[(a)] either $(X,v)$ admits a nontrivial algebraic factor, 
\item[(b)] or there exists a generically finite-to-finite correspondence between $(X,v)$ and $(A,v_A)$ where $A$ is a simple Abelian variety of dimension $\geq 2$ and $v_A$ is a translation-invariant vector field on $A$.
\end{itemize}
We need to show that in both cases the representation of the Galois group is not quasi-irreducible.  In the first case, by Lemma \ref{lemma-Galois group}, we may assume that $(X,v)$ admits a rational factor $\phi: (X,v) \dashrightarrow (Y,w)$ satisfying $w \neq 0$ since $(X,v)$ has no nontrivial integral. By  \cite[Proposition 1.5.1]{jaoui-Confluentes}, the differential of $\phi$ defines a surjective morphism of differential modules over $(k(X), \delta_v)$ 
$$ d\phi: (M_X,\nabla^{\mathrm{st}}_v)^\vee \rightarrow (M_Y,\nabla^{\mathrm{st}}_w)^\vee \otimes_{k(Y)} k(X) $$
Taking duals, we obtain that 
$(M_Y,\nabla^{\mathrm{st}}_w)\otimes_{k(Y)} k(X)$ can be identified with a proper nontrivial differential submodule $V_Y$ of $(M_X,\nabla_v^{\mathrm{st}})$. This submodule $V_Y$ is the $k(X)$-vector space generated by the pullbacks of one-forms on $Y$. This implies that $V_Y \neq H_v$ since the equality $H_v  = V_Y$ implies that the pullback of every rational one-form on $Y$ vanishes on $v$ and hence that $w = 0$. Hence $(M_X, \nabla^{\mathrm{st}}_v)$ admits two proper differential submodules and it follows that the representation $(G_1(X,v), \rho_X)$ is not quasi-irreducible.

In the second case, the proof of Lemma \ref{lemma-commutaitve invariant} implies that $(M_X,\nabla_X)^\sharp$ admits a basis formed by the translation invariant vector fields defined over $k(X)^{alg}$ so that $G_1(X,v) = 0$. Since $\mathrm{dim}(M_X) \geq 2$, this contradicts quasi-irreducibility of $(G_1(X,v), \rho_X)$.  
\end{proof}

{\remark Let $(X,v)$ be an algebraic vector field of dimension $\geq 2$. As shown in Example \ref{example-comparison}, if the first projective prolongation is quasi-irreducible then the representation $(G_1(X,v), \rho_X)$ is also quasi-irreducible. Hence, under the assumption that $(X,v)$ does not admit nontrivial rational integral, Theorem \ref{theorem-C-Galois} implies Theorem \ref{theorem-minimality}.}

%
%

\section{Controlling the invariant horizontal subvarieties  using a singularity} 
\textit{In this section, we assume that $k = \mathbb{C}$ is the field of complex numbers in order to use analytic techniques.} One can still apply the results of this section to algebraic vector fields defined over any algebraically closed field $k$ of characteristic zero since by compactness and saturation of $\mathbb{C}$, one may assume that $k$ is a countable subfield of $\mathbb{C}$.


%

\subsection{Linear connection along a smooth foliation} Let $\pi: E \rightarrow X$ be a vector bundle of rank $p$ and let $\mathcal F$ be a {\em smooth} foliation on $X$. We denote by $\mathcal E$ the sheaf of sections of $E$.

\begin{definition} \label{definition-connectionalongafoliation}
A {\em linear connection $\nabla$} on $E$ along $\mathcal F$ is a biadditive morphism of sheaves  
$\nabla : \Theta_\mathcal F \times \mathcal E \rightarrow \mathcal E$
denoted $$(v,\sigma) \mapsto \nabla_v \sigma$$ such that: 
\begin{itemize} 
\item[(a)] $v \mapsto \nabla_v(\sigma)$ is $\mathcal O_X$-linear: for all local sections $v \in \Theta_{\mathcal F}(U), \sigma \in \mathcal E(U)$ and $f \in \mathcal O_X(U)$, 
$\nabla_{fv}(\sigma) = f \cdot \nabla_v(\sigma)$.
\item[(b)] $\sigma \mapsto \nabla_v(\sigma)$ satisfies the Leibniz rule: for all local sections $v \in \Theta_{\mathcal F}(U), \sigma \in \mathcal E(U)$ and $f \in \mathcal O_X(U)$, 
$\nabla_v(f \cdot \sigma) = \delta_v(f) \cdot \sigma + f \cdot \nabla_v(\sigma)$.
\end{itemize} 
Note that (b) means that for every local vector field $v \in \Theta_X(U)$, $\sigma \mapsto \nabla_v \sigma$ is a partial connection on $\mathcal E_{\mid U}$ along the vector field $v$ in the sense of Definition \ref{definition-partialconnectionalongavectorfield}. 
\end{definition}

\begin{example}[Bott's partial connection] \label{example-Bott}
Let $\mathcal F$ be a smooth foliation on $X$. Consider the exact sequence of vector bundles on $X$: 
$$ 0 \rightarrow T\mathcal F \rightarrow TX \overset{\rho}{\rightarrow} N\mathcal F \rightarrow 0$$
and denote by $\mathcal N_\mathcal F$ the sheaf of sections of $N\mathcal F$. The map $\nabla : \Theta_\mathcal F \times \mathcal N_\mathcal F \rightarrow \mathcal N_\mathcal F$ defined by
 $\nabla_w(\theta) =  \rho([w,\hat{\theta}])$
where $\hat{\theta}$ is any (local) lift of $\theta$ to $\Theta_X$ defines a partial connection on the normal bundle $N\mathcal F$ of $\mathcal F$ along the foliation $\mathcal F$ called \textit{Bott's partial connection}. See \cite[Section 4]{Pereira-JEMS}.
%
%
\end{example}

{\remark If $\nabla$ is a partial connection on a vector bundle $E$ along a foliation $\mathcal F$ then there is a unique dual connection denoted $\nabla^\ast$ on the dual $E^\ast$ of $E$ along $\mathcal F$ such that for every local sections $v \in \Theta_\mathcal F(U)$, $\omega \in \mathcal E^\vee(U)$ and $\sigma \in \mathcal E(U)$, we have:
$$ \delta_v(\omega(\sigma)) = \nabla^\ast_v(\omega)(\sigma) + \omega (\nabla_v(\sigma)).$$}

\begin{construction} 
Let $\pi: E \rightarrow X$ be a vector bundle over $X$ endowed with a partial connection $\nabla$ along a foliation $\mathcal F$ on $X$.  We define the $\mathcal O_X$-linear {\em{lifting operator of the partial connection $\nabla$}}
$$\mathcal L_\nabla : \Theta_\mathcal F \rightarrow \pi_\ast \Theta_E$$ as follows: if $v \in \Theta_\mathcal F$, then $\nabla_v$ is a partial connection along the vector field $v$ and we set 
$$\mathcal L_\nabla (v) := \mathbb{E}(\nabla_v) \in \Theta_E(\pi^{-1}(U))$$
where $\mathbb{E}(\nabla_v)$ is the vector field given by Proposition \ref{totalspace}.  The fact that $\mathcal L_\nabla$ is $\mathcal O_X$-linear follows from Definition \ref{definition-connectionalongafoliation} (a). Indeed, given $f \in \mathcal O_X(U)$ and $v,w \in \Theta_\mathcal F(U)$, applying (a) to the dual connection $\nabla^\ast$, we get in the notation of Proposition \ref{totalspace} that for every local section $\sigma$ of $E^\ast$,
$$\mathbb{E}(\nabla_{fv + w})(\overline{\sigma}) =  \overline{\nabla_{fv + w}^\ast (\sigma)} = \overline{f \nabla_v(\sigma) + \nabla_w(\sigma)} =  (f \mathbb{E}(\nabla_{v}) + \mathbb{E}(\nabla_{w})) (\overline{\sigma}).$$
Since $\mathbb{E}(\nabla_v)$ is characterized in Proposition \ref{totalspace} by the property that it projects on $v$ and that $\mathbb{E}(\nabla_v)(\overline{\sigma}) = \overline{\nabla_v^\ast (\sigma)}$ for every local section $\sigma$ of $E^\ast$, the vector field $(f \mathbb{E}(\nabla_{v}) + \mathbb{E}(\nabla_{w}))$ is the lift of $fv + w$. This shows that the lifting operator $\mathcal L_\nabla$ is $\mathcal O_X$-linear.

\end{construction} 

The lifting operator $v \mapsto v^{[1]}$ given by the first prolongation defined in Section 3.2 is not $\mathcal O_X$-linear (see Remark \ref{local-expression}) so is not of the form $\mathcal L_\nabla$. This obstruction is lifted after passing to the normal bundle of any foliation containing $v$.

\begin{proposition} \label{proposition-comparaison-Bott}
Let $\mathcal F$ be a smooth foliation on $X$. 
\begin{itemize} 
\item[(i)] For every local section $v$ of $\Theta_\mathcal F$, the conormal bundle $N^\ast \mathcal F$ is a closed invariant subvariety for the first prolongation $v^{[1]}$ of $v$.
\item[(ii)] The lifting operator
 $v \mapsto v^{[1]}_{\mid N^\ast \mathcal F}$
coincides with the lifting operator defined by the dual of Bott's partial connection on $N^\ast \mathcal F$  (in particular, it is $\mathcal O_X$-linear).
\end{itemize} 
\end{proposition}

\begin{proof}
(i). Note that since $v$ is a local section of $\mathcal F$ and $\mathcal F$ is a foliation, we have $[v, \Theta_\mathcal F] \subset \Theta_\mathcal F$. Hence, $\mathcal F$ is an invariant foliation of $(X,v)$ and hence (i) follows from Proposition \ref{proposition-distributions-horizontal}.  (ii) follows from the construction of the Bott's connection: working locally, we can assume that $X$ is an affine open set endowed with étale coordinates $x_1,\ldots, x_n$ so that 
$$ \mathcal O_{T^\ast X}(T^\ast X) = \mathcal O_X(X)[y_1,\ldots, y_n]$$
The  closed subvariety $N = N^\ast \mathcal F$ is defined by linear equations in $y_1,\ldots, y_n$ and the restriction morphism $ -_{\mid N}: \mathcal O_X(X)[y_1,\ldots, y_n] \rightarrow \mathcal O_N(N)$
satisfies for every section $w$ of $\Theta_X$:
$$ \overline{w}_{\mid N} = \overline{\rho(w)}$$
where $\overline{w}$ is the function induced by $w$ on $T^\ast X$, $\rho: \Theta_X \rightarrow \mathcal N_\mathcal F$ is the projection and $\overline{\rho(w)}$ is the function induced by $\rho(w)$ on $N^\ast \mathcal F$. By Proposition \ref{totalspace} and Example \ref{example-Bott}, we conclude that for every $v \in \mathcal F(X)$ and every $w \in \Theta_X(X)$, we have 
\begin{eqnarray*} v^{[1]}_{\mid N}(\overline{\rho(w)}) & = & v^{[1]}_{\mid N}(\overline{w}_{\mid N}) = (v^{[1]}(\overline{w}))_{\mid N} \text{ (by restriction to a closed subvariety) } \\ & = & \overline{[v,w]}_{\mid N} \text{ (by Proposition \ref{totalspace})} \\ & = & \rho(\overline{[v,w]}) = \nabla^{\mathrm{Bott}}_v(\overline{\rho(w)}) \text{
 (by definition of Bott's partial connection)} \end{eqnarray*}
It follows that $(N^\ast \mathcal F, v^{[1]}_{\mid N}) \rightarrow (X,v)$ is a morphism of algebraic vector field which satisfies the condition of Proposition \ref{totalspace} for the partial connection $\nabla^{\mathrm{Bott}}_v$ and hence that $v^{[1]}_{\mid N} = \mathcal L_{\nabla^{\mathrm{Bott}}}(v)$.
This concludes the proof of the proposition. 
\end{proof}

\subsection{Horizontal subvarieties of the canonical invariant hypersurface}

\begin{lemma} \label{lemma-connection-distribution}
Let $\pi: E \rightarrow X$ be a vector bundle over $X$ endowed with a partial connection $\nabla$ along a smooth foliation $\mathcal F$ on $X$. There is a unique smooth distribution $\mathcal H = \mathcal H(\nabla)$ on $E$ of the same rank of $\mathcal F$ such that 
$$ \mathcal L_\nabla : \Theta_\mathcal F \rightarrow \pi_\ast \Theta_E$$
takes values in  $\pi_\ast \mathcal H$. Moreover 
$$d\pi_{\mid \mathcal H} : T\mathcal H \rightarrow T\mathcal F \times_{X} E$$
is an isomorphism of vector bundles over $E$.
\end{lemma}

The distribution $\mathcal H(\nabla)$ is often called the horizontal distribution associated to the partial connection $\nabla$. We will refrain from using this terminology because of the conflict with the notion of ``horizontal subvarieties'' of $\P(T^\ast X)$ defined previously.  
\begin{proof}
Assume that the smooth foliation $\mathcal F$ has rank $r$. 
On an affine open set $U$ of $X$, $\mathcal F$ is generated by vector fields $v_1,\ldots, v_r \in \Theta_\mathcal F(U)$ and set $\mathcal H(\nabla)_{\mid U}$ to be the distribution of rank $r$ on $E_{\mid U}$ generated by 
$$w_i = \mathcal L_\nabla v_i, i = 1,\ldots, r.$$
Any distribution $\mathcal H$ satisfying the conclusion of the lemma must satisfy $\mathcal H(\nabla)_{\mid U} \subset \mathcal H_{\mid U}$ since both are distributions of rank $r$, they must be equal. Moreover, by $\mathcal O_X$-linearity of the lifting operator $\mathcal L_\nabla$, for every local vector field $w$ on $U$, $\mathcal L_\nabla(w)$ is a local section of $\pi_\ast \mathcal H(\nabla)_{\mid U}$ so that $\mathcal H(\nabla)_{\mid U}$ satisfies on $U$ the conclusion of the lemma. Finally, by uniqueness, the $\mathcal H(\nabla)_{\mid U}$ glue together to form a global smooth distribution on $X$. 

For the moreover part, it is enough to see that the morphism $d\phi_{\mid \mathcal H(\nabla)}$ is surjective since both are vector bundles of rank $r$ on $E$. This follows from the fact that for every $i$ 
$$\pi: (E,w_i = \mathcal L_\nabla(v_i)) \rightarrow (X,v_i) $$
is a morphism of algebraic vector fields and that the vectors $v_1(x),\ldots, v_r(x)$ generates $T\mathcal F_{x}$ for every point $x \in U$.
\end{proof}

\begin{definition} 
Let $\mathcal F$ be a smooth foliation on $X$. The foliation $\mathcal H(\nabla^{\mathrm{Bott}})$ on $N^\ast \mathcal F$ (of the same rank as $\mathcal F$) associated to the Bott partial connection is called {\em the first prolongation of the foliation $\mathcal F$} and denoted $\mathcal F^{[1]}$.
\end{definition}

\begin{proposition}\label{proposition-translation-Pereira}
Let $(X,v)$ be an algebraic vector field with reduced singularities. Denote by $U$ the complement of the singular locus of $v$ in $X$ and by $\mathcal F$ the nonsingular foliation of rank one on $U$ tangent to $v$. Consider $Z$ a closed horizontal subvariety of $\P(T^\ast X)$ contained in $H(v)$, denote by $Z_U$ its restriction to $\P(T^\ast U)$ and by $$g : N^\ast \mathcal F \setminus \lbrace 0\text{-section} \rbrace \rightarrow H(v) \cap p^{-1}(U) = \P(N^\ast \mathcal F)$$
the canonical projection. The following are equivalent: 
\begin{itemize} 
\item[(i)] the closed subvariety $Z$ of $\P(T^\ast X)$   is  $\P(v^{[1]}))$-{\em invariant},
\item[(ii)] the  homogeneous closed subvariety $\overline{g^{-1}(Z)}$ of $N^\ast \mathcal F$ is $v^{[1]}_{\mid N^\ast \mathcal F}$-{\em invariant}, 
\item[(iii)] the  homogeneous closed subvariety  $\overline{g^{-1}(Z)}$ of $N^\ast \mathcal F$ is $\mathcal F^{[1]}$-{\em invariant}.
\end{itemize} 
\end{proposition} 

\begin{proof} 
The equivalence between (i) and (ii) follows Proposition \ref{totalspace} asserting that
$$g : (N^\ast \mathcal F \setminus \lbrace 0\text{-section} \rbrace, v^{[1]}) \rbrace \rightarrow (H(v) \cap p^{-1}(U), \P(v^{[1]}))$$
is a morphism of algebraic vector fields. Hence, $g$ sends invariant subvarieties onto invariant subvarieties. The equivalence between (ii) and (iii) follows from the fact that the rank one foliation
$\mathcal F^{[1]}$ is generated by $v^{[1]}_{\mid N^\ast \mathcal F}$ by Proposition \ref{proposition-comparaison-Bott}.
\end{proof}

\begin{theorem}[Pereira]\label{theorem-Pereira}
Let $X$ be a smooth algebraic variety and $\mathcal F$ be a foliation of rank one on $X$.  Assume that the foliation $\mathcal F$ admits a singular point $x_0 \in X$ satisfying:
\begin{itemize}
\item[(a)] the singularity $x_0$ is non-resonant, 
\item[(b)] the singularity $x_0$ is not contained in any $\mathcal F$-invariant proper positive-dimensional irreducible subvariety of $X$.
\end{itemize}
Then every {\em proper} homogeneous closed subvariety  of $N^\ast \mathcal F$ invariant for the foliation $\mathcal F^{[1]}$ projects on a {\em proper} closed invariant subvariety of the foliation $\mathcal F$.
\end{theorem} 

We refer to the proof of Theorem 1 in \cite{Pereira-JEMS} for the proof of this theorem.

\begin{remark} \label{remark-dimension3}
If $\mathrm{dim}(X) = 2$, the statement is vacuous so that the first nontrivial case is the case of algebraic varieties of dimension three. In the case where $\mathrm{dim}(X) = 3$, Proposition 3.3 of \cite{Pereira-JEMS} shows that the assumption that the singularity is non-resonant is not necessary for the conclusion of the theorem to hold.   
\end{remark} 

\begin{corollary} \label{corollary-Pereira}
Let $X$ be a smooth algebraic variety and let $v$ be a vector field with reduced singularities which admits a singular point $x_0 \in X$ satisfying:
\begin{itemize}
\item[(a)] the singularity $x_0$ is non-resonant, 
\item[(b)] the singularity $x_0$ is not contained in any invariant proper positive-dimensional irreducible closed algebraic subvariety of the algebraic vector field $(X,v)$.
\end{itemize}
Then the canonical invariant subvariety $H(v)$ is a {\em minimal  invariant horizontal} subvariety of the algebraic vector field $(\P(T^\ast X), \P(v^{[1]}))$.
\end{corollary}

\begin{proof} 
Consider $Z$ a $\P(v^{[1]})$-invariant horizontal subvariety  contained in $H(v)$. With the notation of Proposition \ref{proposition-translation-Pereira}, this implies that 
$\overline{g^{-1}(Z)}$ is a homogeneous closed subvariety of  $N^\ast \mathcal F$ invariant under the foliation $\mathcal F^{[1]}$ which projects surjectively on $X$. Hence, Theorem \ref{theorem-Pereira} implies that $\overline{g^{-1}(Z)} = N^\ast \mathcal F$ so that $Z = H(v)$  as required. 
\end{proof} 

\subsection{Invariant multidistributions of codimension one}

\begin{definition} 
A {\em{multidistribution $W$ of codimension one  on $X$}} is a horizontal subvariety $W$ of $\P(T^\ast X)$ of relative dimension zero. The {\em{degree of the multidistribution $W$}} is the cardinal of the generic fiber of 
$ p_W : W \rightarrow X$.
If moreover $X$ is equipped with a vector field $v$ then we say that a multidistribution $W$ is {\em invariant under $v$} if $W$ is an invariant horizontal  subvariety of $(\P(T^\ast X), \P(v^{[1]}))$. 
\end{definition}

{\example \label{example-degreeone} If $\mathcal D$ is an (invariant) distribution of rank $n - 1$ then the  horizontal subvariety $Z(\mathcal D)$ given by Proposition \ref{proposition-distributions-horizontal} is a degree one (invariant) multidistribution of codimension one. All degree one (invariant) multidistributions are of this form.}

{\definition Let $W$ be a multidistribution of codimension one on $X$. The {\em singular locus} $\mathrm{Sing}(W)$ of $W$ is the image in $X$ of the critical locus of 
$p_W: W \rightarrow X$
namely 
$$\mathrm{Sing}(W) = p_W(W^{\mathrm{sing}} \cup \lbrace x \in W^{\mathrm{reg}} \mid dp_W \text{ is not locally invertible at } x\rbrace).$$
Since $p_W$ is a projective morphism, $\mathrm{Sing}(W)$ is always a closed subvariety of $X$.}

\begin{lemma} \label{lemma-criticalocus invariant}
Let $p: (Y,w) \rightarrow (X,v)$ be a smooth morphism of algebraic vector fields and $Z$ a closed invariant subvariety of $(Y,w)$ projecting surjectively on $X$. Then the critical locus $\mathrm{Cr}(Z)$ of $p_{\mid Z}: Z \rightarrow X$ defined by 
$$\mathrm{Cr}(Z) = \mathrm{Sing}(Z) \cup \lbrace z \in Z \mid dp_{\mid Z} \text{ is not locally invertible  at } z  \rbrace$$
 is a closed invariant subvariety.   
\end{lemma}

\begin{proof}
This is well-known and follows easily from the fact that the pieces of analytic flows $\phi_t$ and $\psi_t$ associated with $w$ and $v$ respectively around a point $z \in Z$ and its image $x = \phi(z)$ respectively are related by $$p \circ \phi_t = \psi_t \circ p$$
since $p$ is a morphism of algebraic vector fields.
\end{proof} 
%
%

\begin{proposition} \label{proposition-preparation-multidistribution}
Let $(X,v)$ be algebraic vector field and let $W$ be an invariant codimension one multidistribution on $(X,v)$ of degree $k \geq 1$.

\begin{itemize} 
\item[(i)] the singular locus $\mathrm{Sing}(W)$ is a closed invariant subvariety of $(X,v)$, 
\item[(ii)] Denote by $\mathcal I(W)$ the union of the irreducible components of $\mathrm{Sing}(W)$ of codimension one. There is a connected étale cover 
$$ \phi :Y \rightarrow X \setminus \mathcal I(W)$$ 
and  distributions $\mathcal D_1,\ldots, \mathcal D_k$ of codimension one on $Y$ invariant under the pullback $\phi^\ast v$ of $v$ such that 
$$ \phi^\ast W = Z(\mathcal D_1) \cup  \cdots \cup Z(\mathcal D_k).$$
\end{itemize}  
\end{proposition}

\begin{proof} 
The first part follows from Lemma \ref{lemma-criticalocus invariant} and the properness of $p : \P(T^\ast X) \rightarrow X$ and the second part is Proposition 1.3.1 (b) in \cite{Pereira-Pirio}. As the proof is left to the reader in the mentioned reference, we give a complete proof.

\begin{claim}
There is a connected étale cover 
$\phi :Y \rightarrow X \setminus \mathcal \mathrm{Sing}(W)$ 
and  distributions $\mathcal D_1,\ldots, \mathcal D_k$ of codimension one on $Y$ invariant under the pullback $\phi^\ast v$ of $v$ such that 
$$ \phi^\ast W = Z(\mathcal D_1) \cup  \cdots \cup Z(\mathcal D_k).$$
\end{claim}

\begin{proof}[Proof of the claim]
Set $U = X \setminus \mathrm{Sing}(W)$. By construction of $U$, $p_U: W_U = W \cap p^{-1}(U) \rightarrow U$
is étale and proper hence is a finite étale cover with fibers of cardinal $k$. Fix $x_0 \in U$ a base point so that we obtain an action by monodromy: 
$$\rho: \pi_1(U^{\mathrm{an}},x_0) \rightarrow \mathfrak S(W_{x_0})$$
By the Galois correspondence for unramified covering, we can find 
a pointed connected unramified cover $(C,y_0) \rightarrow (U,x_0)$ such that $\pi_1(C,y_0) = \mathrm{Ker}(\rho)$ as subgroups of $\pi_1(X,x_0)$.
The isomorphism between the étale fundamental group and the analytic fundamental group shows that we can write $C = Y^{\mathrm{an}}$ for some étale cover $\phi: Y \rightarrow X$ of $X$, see Theorem 5.7.4 in \cite{Szamuely}.  It follows from the construction of $Y$ that $W_Y := \phi^\ast W$ is a nonsingular multidistribution of codimension one on $Y$ and that the monodromy action  
$$ \rho_Y: \pi_1(Y^{\mathrm{an}},y_0) \rightarrow \mathfrak S(W_{Y,y_0}) =  \mathfrak S (W_{x_0})$$ 
is trivial. Since the connected components of a topological cover are in one-to-one correspondence with the orbits of the monodromy action, we can conclude that  $W_Y^{\mathrm{an}}$ has $k$ connected components and hence so does $W_Y$ since the connected components of $W^{\mathrm{an}}_Y$ are the analytification of the connected components of $W_Y$. Hence, the smooth algebraic variety $W_Y$ has $k$ irreducible components which all have degree one over $X$ and we can write
$$ W_Y = Z(\mathcal D_1) \cup \cdots \cup Z(\mathcal D_k)$$ where $\mathcal D_1,\ldots, \mathcal D_k$ are pairwise transverse distributions of codimension one on $Y$ by Example \ref{example-degreeone}. By Proposition \ref{proposition-pull-back-push forward}, $W_Y$ is an invariant multidistribution for the pullback $\phi^\ast v$ and so are its irreducible components. Hence, the $\mathcal D_i$ are invariant distributions for $\phi^\ast v$. 
\end{proof}

It remains to show that this decomposition can be extended to $X \setminus \mathcal I(W)$. By purity of the branch locus (see Corollary 5.2.14 in \cite{Szamuely}), the connected étale cover $\phi: Y \rightarrow U$ can be extended in a connected étale cover 
$$ \overline{\phi}: \overline{Y} \rightarrow X \setminus \mathcal I(W).$$ 
By Fact \ref{fact-extension distribution}, the distributions $\mathcal D_1,\cdots, \mathcal D_k$ can 	also be extended into (possibly singular) distributions $\overline{\mathcal D_1}, \ldots, \overline{\mathcal D_k}$ on $\overline{Y}$ and we have
$$\overline{\phi}^\ast W =  Z(\overline{\mathcal D_1}) \cup \ldots \cup Z(\overline{\mathcal D_k}) $$
since this equality holds true on some nonempty Zariski-open set. Finally, as $\mathcal D_1,\ldots, \mathcal D_k$ are invariant distributions for $\phi^\ast v$, the distributions $\overline{\mathcal D_1},\ldots, \overline{\mathcal D_k}$ are invariant under $\overline{\phi}^\ast v$. This concludes the proof of the proposition. 
\end{proof}

\subsection{A minimality criterion based on a singular point} The main result of this section is the following theorem.

\begin{theorem}[Theorem C]\label{theorem-minimality2}
Let $(X,v)$ be a complex algebraic vector field with reduced singularities of dimension $n \geq 2$. Assume that $(X,v)$ admits a singular point $x_0 \in X$
such that
\begin{itemize}
\item[(a)] the singularity $x_0$ is  non-resonant, 
\item[(b)]  the singularity $x_0$ is not contained in any invariant proper positive-dimensional irreducible subvariety of $(X,v)$.
\end{itemize}
Then the generic type of $(X,v)$ is minimal.  
\end{theorem}

{\remark \label{remark-theoremminimality2} \text{ }

\begin{itemize} 
\item A direct argument shows that the existence of a singular point of $(X,v)$ rules out the possibility of (ii) in Fact \ref{fact-Jaoui-Moosa}. Indeed, assume that $(X,v)$ admits a singular point $x_0$ and consider a generically finite dominant rational morphism 
$$ \phi:(X,v) \dashrightarrow (A,v_A)$$
as in (ii) of Fact \ref{fact-Jaoui-Moosa}. Because $A$ is an abelian variety, $\phi$ can be extended into a regular morphism $\overline{\phi}: X \rightarrow A$ and since $v_A$ is a global vector field on $A$, we obtain a regular morphism of algebraic vector fields 
$$ \overline{\phi}:(X,v) \rightarrow (A,v_A)$$
It follows that $\overline{\phi}(x_0)$ must be a singular point of $v_A$ which is a contradiction since $v_A$ is nowhere singular. 

\item The presence of a singular point (even non-resonant) does not rule out the possibility of (i) to hold in Fact \ref{fact-Jaoui-Moosa}. Indeed, if $(Y,w)$ is an algebraic vector field with a singular point $y_0$ then any algebraic vector field of the form $(X,v) = (Y,w) \times (\mathbb{A}^1, \lambda x \frac d {dx})$ admits $x_0 = (y_0,0)$ as as singular point. In that case, the algebraic vector field $(X,v)$ admits (at least) two nontrivial rational factors given by the two projections. In particular, the generic type of $(X,v)$ is not minimal. Note  that on any example of this form {\em this singular point is contained both in an invariant hypersurface and an invariant curve} so that the hypothesis (b) is violated in two different ways.

\item In dimension two and three, the hypothesis (a) that the singular point is non-resonant can be removed. This follows from the improvement of Theorem \ref{theorem-Pereira} in dimension three described in section 3 of \cite{Pereira-JEMS}, see Remark \ref{remark-dimension3}. In dimension $\geq 4$, this assumption will be needed in our analysis.
\end{itemize}}

\begin{lemma} \label{lemma-tangent}
Let $\mathcal D$ be an invariant distribution of codimension one and let $\mathcal F$ be the foliation tangent to $v$. Assume that $\mathcal D$ and $\mathcal F$ are generically transverse: on some nonempty open set $U$ of $X$ (where both $\mathcal D$ and $\mathcal F$ are nonsingular) we have 
$$ TU = T\mathcal F_{\mid U} \oplus T\mathcal D_{\mid U}$$
Then every singular point of $\mathcal F$ is contained in a closed invariant hypersurface. 
\end{lemma} 
\begin{proof} 
We may assume that $X$ is affine, that $\mathcal F$ (or equivalently $v$) has at least one singular point $x_0$ and that $\mathcal D$ is generated by a global one form $\omega \in \Omega^1_X(X)$. We claim that $g = i_v (\omega) \in \mathbb{C}[X]$ generates an invariant ideal of $\mathbb{C}[X]$: indeed, since $\mathcal D$ is invariant, we can write $\nabla_v^\vee(\omega) = h \cdot \omega$ for some $h \in \mathbb{C}[X]$ and it follows that 
$$\delta_v(g) = \nabla_v^\vee(\omega)(v) + \omega(\nabla_v(v)) = \nabla_v^\vee(\omega)(v) = h \cdot \omega(v) = h \cdot g .$$
Now, by construction, $g$ vanishes at every singular point of $\mathcal F$ and since $\mathcal F$ and $\mathcal D$ are generically transverse, it is not identically zero on $X$. So the irreducible components of $g = 0$ are invariant hypersurfaces of $(X,v)$ passing through every singular point.
%
\end{proof} 

\begin{proof}[Proof of Theorem \ref{theorem-minimality2}]
We show that the first prolongation of $(X,v)$ is quasi-irreducible and apply Theorem D. Denote by $H(v)$ the canonical invariant hypersurface and consider $Z$ a proper invariant irreducible horizontal subvariety of $(\P(T^\ast X), \P(v^{[1]}))$. We distinguish two cases: 

\begin{itemize} 
\item \textit{Case 1:} either the relative dimension of $Z$ is $> 0$. 
\end{itemize} 
In that case, at the generic point $x \in X$, we have $H(v)_x \cap Z_x \neq \emptyset$ since $H(v)_x$ and $Z_x$ are two irreducible subvarieties of $\P^{n-1}$ such that $\mathrm{dim}(H(v)_x) + \mathrm{dim}(Z_x) \geq n - 1$.
It follows that some irreducible component $Z'$ of $Z \cap H$ is a horizontal subvariety of $\P(T^\ast X)$ contained in $H(v)$. It is also $\P(v^{[1]})$-invariant as the intersection of two closed invariant subvarieties is also invariant. Now Corollary \ref{corollary-Pereira} implies that $Z' = H(v)$ and therefore that  $H(v) \subset Z$. Since $H(v)$ is an hypersurface and $Z$ is irreducible, we conclude that $H(v) = Z$ as required to show the the first prolongation of $(X,v)$ is quasi-irreducible. 

\begin{itemize} 
\item \textit{Case 2:} or the relative dimension of $Z$ is zero and hence $Z = W$ is an invariant multidistribution of codimension one. 
\end{itemize} 
By Proposition \ref{proposition-preparation-multidistribution}, the singular locus $\mathrm{Sing}(W)$ is an invariant subvariety of $(X,v)$ and if $\mathcal I(W)$ denotes the union of the codimension one irreducible components of $\mathrm{Sing}(W)$ then there is an étale morphism $\phi: Y \rightarrow X \setminus \mathcal I(W)$
and distinct invariant distributions $\mathcal D_1,\ldots, \mathcal D_k $ such that $$\phi^\ast W = Z(\mathcal D_1) \cup \cdots \cup Z(\mathcal D_k)$$
Now $x_0 \notin \mathcal I(W)$ since by (b), $x_0$ is not contained in any proper invariant hypersurface. Denote by $w$ the lift of $v$ to $Y$ and consider $y_0 \in \phi^{-1}(x_0)$. Note that since $\phi$ is étale, $y_0$ is a singular point of $w$ which is not contained in  any invariant proper positive-dimensional irreducible subvariety of $(Y,w)$. Hence, Lemma \ref{lemma-tangent} implies that the invariant codimension one distributions $\mathcal D_1,\ldots, \mathcal D_k$ are not generically transverse to $\mathcal F$  that is:
$$  Z(\mathcal D_i) \subset Z(\mathcal F(w))  \text{ for } i = 1,\ldots , k  $$
If follows that $\phi^\ast W \subset  \phi^\ast (H(v))$  and hence that $W \subset H(v)$. As in Case 1, we can now apply Corollary \ref{corollary-Pereira} and conclude that this is impossible if $\mathrm{dim}(X) \geq 3$ and that $W = H(v)$ if $\mathrm{dim}(X) = 2$.  This concludes the proof that the first prolongation of $(X,v)$ is quasi-irreducible.

If $\mathrm{dim}(X) \geq 3$, the hypothesis of Theorem D are fulfilled and we conclude that the generic type of $(X,v)$ is minimal as required. To handle the case where $\mathrm{dim}(X) = 2$, it remains to see that (b) ensures that every rational integral of $(X,v)$ is constant. This follows from the theorem of Camacho and Sad \cite{Camacho-Sad} asserting that every singular point admits at least one analytic separatrix, see the proof of Theorem 2.1.2 in \cite{jaoui-ANT} for additional details. 
\end{proof} 
\section{Abundance of strongly minimal autonomous differential equations}

In this section, we work in the first-order theory $\mathrm{DCF}_0$ of differentially closed field of characteristic zero and use freely in this context several important notions from geometric stability theory (orthogonality, binding groups, Zilber's trichotomy...). Classical references in the setting of differential algebra are \cite{Bouscaren, Marker}  and \cite{Pillay} in general. We also refer the reader to the recent \cite{Pizarro} for a self-contained introduction to these notions.

\subsection{From minimality to strong minimality} Fix $k$ an algebraically closed field of characteristic zero.

\begin{lemma} \label{lemma-strongminimality}
Let $(X,v)$ be a algebraic vector field of dimension $n \geq 2$. Assume that $(X,v)$ does not admit any proper positive-dimensional invariant algebraic subvariety
and denote by $p$ the generic type of $(X,v)$. Then 
\begin{itemize} 
\item[(i)] $\mathrm{Sing}(v)$ is finite, 
\item[(ii)] the solution set $(X,v)^\U$ can be decomposed as 
$$ (X,v)^\U = \mathrm{Sing}(v) \cup p(\U) $$
\item[(iii)] the solution set $(X,v)^\U$ is strongly minimal if and only if $p$ is a minimal type.
\end{itemize} 
\end{lemma} 

\begin{proof} 
(i) follows from the assumption and the fact that $\mathrm{Sing}(v)$ is a proper closed invariant subvariety.  To show (ii), note that by quantifier elimination in $\mathrm{DCF}_0$ \cite[Theorem 1.1]{Marker}, the set $\mathcal Q$  of types $q \in S(k)$ living on $(X,v)^\U$ is in bijection with the set of invariant irreducible subvarieties of $(X,v)$. Since the proper irreducible invariant subvarieties are all zero-dimensional, it follows that  
$$(X,v)^\U = \bigcup_{q \in \mathcal Q} q(\U) = p(\U) \cup \mathrm{Sing}(v).$$
To prove (iii), assume that $p$ is a minimal type and, for the sake of a contradiction, consider $D$ a definable subset of $(X,v)^\U$ defined over a differential extension $K$ of $k$ which is both finite and cofinite.  Since $\mathrm{Sing}(v)$ is finite, (ii) implies that both 
$$p(x) \cup \lbrace x \in D \rbrace \text{ and } p(x) \cup \lbrace x \notin D \rbrace $$ 
are infinite and contain nonalgebraic complete types $p_+ \in S(K)$ and $p_-\in S(K)$ respectively. Since $p$ is minimal both are nonforking extension of $p$. This leads to a contradiction as $p$ is a stationary type.  
\end{proof} 

\subsection{Case of the affine space}

\begin{theorem} \label{theorem-affine-sm}
Let $d \geq 2, n \geq 2$ and consider a system of autonomous differential equations: 

$$(S): \begin{cases} 
x_1'& =  f_1(x_1,\ldots, x_n) \\
& \vdots  \\ 
x'_n & = f_n(x_1,\ldots, x_n) 
\end{cases}$$
where $f_1,\ldots, f_n \in \mathbb{C}[X_1,\ldots, X_n]$ are polynomials of degree $d \geq 2$. If the coefficients of $f_1,\ldots, f_n$ are  $\mathbb{Q}$-algebraically independent transcendental complex numbers (in particular they are all nonzero) then the solution set of $(S)$ is strongly minimal and geometrically trivial. 
\end{theorem} 

\begin{remark} \label{remark-sharpdegree}
Even when $n = 2$, Theorem \ref{theorem-affine-sm} improves on the main theorem of \cite{jaoui-ANT} which assumes that the degree $d$ is greater or equal to $3$. 

\begin{itemize}

\item The requirement that the degree $d$ is greater or equal to $2$ is sharp: every system $(S)$ of degree $1$ is internal to the constants and the generic one has binding (or Galois) group the torus $\mathbb{G}_m^n$. This follows easily from the fact that the equation 
$$y' = ay+b$$
over $\mathbb{Q}(a,b)$ has Galois group $\mathbb{G}_m$. When $a = 0 \text{ and } b \neq 0$, then the equation is still internal to the constants but its binding group is the additive group $\mathbb{G}_a$.

\item In dimension one, Rosenlicht's theorem \cite{Rosenlicht} ensures that the conclusion of the theorem holds as long as $d \geq 3$ but not when $d = 2$. Indeed, when $d = 2$, the equation 
$$x' = a x^2 + b x + c$$
is an autonomous Riccati equation. Hence, the equation is internal to the constants: given three distinct particular solutions, any other solution is an algebraic combination of these solutions and constants. Hence the solution set is not geometrically trivial. 
\item In dimension three, it is possible to construct explicit examples of quadratic vector fields on $\mathbb{C}^3$ satisfying the conclusion of the theorem: the Halphen equations described in \cite{Guillot}: 
$$H(\alpha_1,\alpha_2,\alpha_3): \begin{cases}
x_1' = \alpha_1 x_1^2 +  (1 - \alpha_1)(x_1x_2 + x_1x_3 - x_2x_3) \\ 
x_2' = \alpha_2 x_2^2 +  (1 - \alpha_2)(x_1x_2 - x_1x_3 + x_2x_3) \\
x_3' = \alpha_3 x_3^2 +  (1 - \alpha_3)(- x_1x_2 + x_1x_3 + x_2x_3)
\end{cases}$$
for well-chosen values of $\alpha_1,\alpha_2,\alpha_3 \in \mathbb{C}^3$ are strongly minimal and geometrically trivial. This follows from the results of \cite{BCFN} and was kindly pointed out to the author by Ronnie Nagloo. In this case, this construction produces strongly minimal and geometrically trivial systems $(S)$ with coefficients in $\mathbb{Q}$. 
\end{itemize}

In higher dimension, we do not know how to produce strongly minimal systems with coefficients in $\overline{\mathbb{Q}}$. This is reminiscent of the situation of Matzat's conjecture in inverse differential Galois theory which was first established over algebraically closed field $k$ of infinite transcendence degree and only recently deduced over  $\overline{\mathbb{Q}}$ in \cite{Feng-Wibmer} using specialization techniques going back to \cite{Hrushovski-Galois}.
\end{remark}

In the proof of the theorem, the following result will be used. 

\begin{fact}[{\citep[Theorem 4.3]{Coutinho}}]\label{theorem-Coutinho}
Let $n,d \geq 2$ and let $H$ be an hyperplane of $\P^n$. A {\em generic} algebraic foliation of $\P^n$ of {\em degree $d \geq 2$} that leaves $H$ invariant does not leave invariant any other proper positive-dimensional  irreducible algebraic  subvariety of $\P^n$.  
\end{fact} 

Let's spell out explicitly what is a {\em generic} algebraic foliation of $\P^n$ of degree $d$ leaving $H$ invariant in the sense of \cite{Coutinho,Pereira-JEMS}. As described in Example \ref{example-foliations of rank one}, the space $\mathrm{Fol}(\P^n,d)$ of foliations of degree $d$ on $\P^n$ can be identified with the open set 
$$U \subset \P(\Xi(\P^n, \mathcal O_{\P^n}(d-1)))$$
consisting of the projective classes of twisted vector fields with reduced singularities. As described in Proposition \ref{proposition-affine-degree}, the space of twisted vector field preserving the hyperplane $H$ is a subvector space 
$$\overline{\Xi}(n,d) \subset \P(\Xi(\P^n, \mathcal O_{\P^n}(d-1))) $$ 
With this notation, Fact \ref{theorem-Coutinho} asserts that for $d \geq 2$ and $n \geq 2$, there exists a countable union of proper Zariski closed subsets $\mathcal E_1 = \bigcup_{i \in \mathbb{N}} Z_i(\mathbb{C})$ of $\overline{\Xi}(n,d)$ such that every $v \in \overline{\Xi}(n,d) \setminus \mathcal E_1$ satisfies: 
\begin{itemize} 
\item[(a)] the image $\P(v)$ of the twisted vector field $v$ lies in $U$ and defines a foliation of degree $d$ preserving the hyperplane $H$, 
\item[(b)] apart from the hyperplane $H$, the twisted vector field $v$ (and the associated foliation) does not leave invariant any nontrivial algebraic subvariety of $\P^n$. 
\end{itemize}

\begin{proof}[Proof of Theorem \ref{theorem-affine-sm}] 
Fix $n,d \geq 2$. There is a one-to-one correspondence between polynomial systems $(S)$ in $\Xi(n,d)$ and algebraic vector fields $v_S$ on the affine space $\mathbb{A}^n$ of degree $\leq d$ given by 
$$ (S) \mapsto v_S = \sum_{i = 1}^n f_i \cdot \partial_{x_i}.$$
We say that $v_S$ is a {\em generic vector field of degree $d$} if the coefficients of the $f_1,\ldots, f_n$ are $\mathbb{Q}$-algebraically independent transcendental complex numbers. By Proposition 3.2.3 of \cite{jaoui-ANT}, it is enough to prove that the conclusion of the theorem holds for \textit{some} generic vector field of degree $d$. 

\setcounter{claim}{0}

\begin{claim} 
Let $d \geq 2$. There exists a countable union of proper Zariski-closed subsets $ \mathcal E_1 = \bigcup_{n \in \mathbb N} Z_n(\mathbb{C})$ of $\Xi(n,d)$ such that every vector field $v \in \Xi(n,d) \setminus \mathcal E_1$ does not leave invariant any proper positive-dimensional closed subvariety of $\mathbb{A}^n$. 
\end{claim}

\begin{proof}[Proof of the claim]
By Proposition \ref{proposition-affine-degree}, we have an isomorphism of complex vector spaces: 
$$\mathrm{prol} : \Xi(n,d) \rightarrow \overline{\Xi}(n,d) \subset \Xi(\P^n, \mathcal O_{\P^n}(d - 1))$$
which to a vector field $w \in \Xi(n,d)$ associates $\overline{w}$ an $\mathcal O_{\P^n}(d - 1)$-twisted vector field leaving the hyperplane at infinity $H$ invariant and whose restriction to $\mathbb{A}^n$ is $w$.  By Fact \ref{theorem-Coutinho}, there is  a countable union of proper Zariski-closed subsets $\mathcal E_1 = \bigcup_{i \in \mathbb{N}} Z_i(\mathbb{C})$ of $\overline{\Xi}(n,d)$ such that every $v \in \overline{\Xi}(n,d) \setminus \mathcal E_1$ does not leave invariant any proper positive-dimensional algebraic subvariety of $\P^n$ apart from the hyperplane $H$ at infinity. We conclude that every vector field $v \notin \mathrm{prol}^{-1}(\mathcal E_1)$  does not leave invariant any proper positive-dimensional subvariety of $\mathbb{A}^n$. Since $\mathrm{prol}$ is an isomorphism, this concludes the proof of the claim.
\end{proof}

Let $v_0 = \sum_{i = 1}^n \lambda_ix_i\cdot \partial_{x_i}$
be a linear vector field with a non-resonant isolated singularity at zero that is such that $\lambda_1,\ldots, \lambda_n \in \mathbb{C}$ generate a $\mathbb{Z}$-module of rank $n$.

\begin{claim} 
There is an open disk $D$ around $0$ and a connected analytic neighborhood $U$ of $v_0$ in $\Xi(n,d)$ such that 
\begin{itemize} 
\item every vector field $v$ in $U$ has a unique singularity $ x_0(v)$ in $D$, 
\item there are analytic functions 
$$\lambda_i: U \rightarrow \mathbb{C}$$ 
such that $\lambda_1(v),\ldots, \lambda_n(v)$ are the eigenvalues of the linear part of $v$ at $x_0(v)$. 
\end{itemize} 
\end{claim} 

\begin{proof}[Proof of the claim] 
This follows from the analytic implicit function theorem. Consider the evaluation morphism:
$$e : \Xi(n,d) \times \mathbb{A}^n \rightarrow \mathbb{A}^n$$
given by $(v,x) \mapsto v(x)$. By Lemma \ref{lemma-singularity} (ii), the partial differential $d_2e_{v_0,0}$ along the second factor at $(v_0,0)$ is the linear part of $v_0$ at $0$ and hence invertible. By the analytic implicit function theorem, there are analytic disks $D$ around $0$ and $U$ around $v_0$ and an analytic function 
$$x_0 : \begin{cases} U \rightarrow D \\ v \mapsto x_0(v) 
\end{cases}$$
such that every vector field $v \in U$ has a unique singularity on $D$ which is $x_0(v)$. It follows that the function $v  \mapsto d_2e_{v,x(v)} \in \mathrm{Mat}_n(\mathbb{C})$ is also analytic. Since the eigenvalues of this matrix are distinct for $v = v_0$, by following analytically the roots of the characteristic polynomial of this matrix, there are (up to restricting $U$ again) analytic functions $\lambda_i: U \rightarrow \mathbb{C}$ for $i = 1,\ldots, n$ 
such that $\lambda_1,\ldots, \lambda_n$ are the eigenvalues of the linear part of $v$ at $x_0(v)$. 
\end{proof}

Set $$ \mathcal E_2 = \bigcup_{(k_1,\ldots, k_n) \in \mathbb{Z}^n \setminus \lbrace 0 \rbrace} \Big\{ v \in U \mid \sum_{i = 1}^n k_i \lambda_i(v) = 0 \Big\} $$  
for the set of vector fields in $U$ such that the singularity $x_0(v)$ is resonant and set $$ \mathcal E_3 = \bigcup_{P \in \mathbb{Q}[\overline{X}] \setminus \lbrace 0 \rbrace}\Big\{ v \in U \mid P(\mathrm{coeff}(v)) = 0 \Big\} $$ 
for the set of vector fields in $U$ whose coefficients satisfy an algebraic relation with coefficients in $\mathbb{Q}$. By definition 
$$\mathcal E = (\mathcal E_1 \cap U) \cup \mathcal E_2 \cup \mathcal E_3$$
is a countable union of proper analytic closed subsets of $U$. Hence, each of them has empty interior. It follows from Baire's category theorem that  $ \mathcal E$ has empty interior in $U$ and in particular is  a proper subset of $U$. So we can consider an algebraic vector field $v \in U \setminus \mathcal E$ so that
\begin{itemize} 

\item $v$ is a generic vector field from $\Xi(n,d)$ since $v \notin \mathcal E_3$, 
\item $x_0(v)$ is a non-resonant singularity since $v \notin \mathcal E_2$,
\item $v$ does not leave invariant any proper positive-dimensional subvariety of $\mathbb{A}^n$  since $v \notin \mathcal E_1$
\end{itemize}
It follows from Theorem C that the generic type of  $(\mathbb{A}^n,v)$  is minimal. Now applying Lemma \ref{lemma-strongminimality} (iii), we conclude that the solution set of  $(\mathbb{A}^n,v)$ is strongly minimal. Hrushovski-Sokolovic trichotomy theorem, that is, Fact \ref{Hrushovski-Sokolovic} then implies that this strongly minimal set is geometrically trivial. This concludes the proof of the theorem.  
\end{proof}

\subsection{Generic vector fields with poles on an ample divisor}

\begin{theorem} \label{theorem-sm-ample}
Let $D$ be an ample divisor on a smooth projective complex algebraic variety $X$. Denote by $X_0$ the complement in $X$ of the support of $D$  and consider the restriction morphism 
$$ -_{\mid X_0}: \Xi(X,\mathcal O_X(kD)) \rightarrow \Xi(X_0).$$
There exists $k_0 = k_0(X,D) \geq 0$ such that for $k \geq k_0$, the restriction on $X_0$ of a {\em generic} vector field $v \in \Xi(X,\mathcal O_X(kD))$ is strongly minimal and geometrically trivial. 
\end{theorem} 

\begin{remark} \label{remark-Lebesguemeasure}
Let $D$ be a divisor on a smooth complex projective surface $X$. Recall that $\Xi(X,\mathcal O_X(D))$ is always a finite-dimensional complex vector space. A priori, we have three competing notions that one can choose from to define the properties of a ``{\em generic }'' $\mathcal O_X(D)$-twisted vector field $v \in \Xi(X,\mathcal O_X(D))$: 
\begin{itemize} 
\item \textit{a measure-theoretic notion}: a property holds for a generic vector field if it holds for all vector fields $v \in \Xi(X,\mathcal O_X(D))$ outside of a measurable subset of $\Xi(X,\mathcal O_X(D))$ of Lebesgue measure zero. 

\item \textit{a topological notion: }a property holds for a generic vector field if it holds for all vector fields $v \in \Xi(X,\mathcal O_X(D))$ outside of a $F_\sigma$-set that is outside of a countable union of closed subsets of $\Xi(X,\mathcal O_X(D))$ (for the analytic topology) with empty interior. 

\item \textit{a model-theoretic notion:} a property holds for a generic vector field if it holds for some/every vector fields $v \in \Xi(X,\mathcal O_X(D))$ realizing the generic type of  $\Xi(X,\mathcal O_X(D))$ over some countable field of definition $l$ for the projective variety $X$ and the divisor $D$. 
\end{itemize} 
By Proposition 3.2.3 of \cite{jaoui-ANT}, it turns out that, for the property we are interested in,  all these notions are equivalent. We will work with the model-theoretic notion since it offers us the possibility to work a single vector field instead of the whole family. 
\end{remark} 

\begin{proof}[Proof of Theorem \ref{theorem-sm-ample}]
The proof follows the same strategy as the proof of Theorem \ref{theorem-affine-sm}. We can assume that $X$ and $D$ are defined over some countable field $l$. By the remark above, it is enough to prove that for $k \gg 0$, the conclusion of the theorem hold for some generic algebraic vector field $v \in \Xi(X,\mathcal O_X(kD))$ over $l$.

Since $\mathcal L = \mathcal O_X(D)$ is ample, by Fact \ref{fact-Pereira-Coutinho-int}, there exits $k_1 = k_1(X,\mathcal O_X(D)) \geq 0$ such that for all $k \geq k_1$, there exists a countable union of proper Zariski-closed subset of $\Xi(X,\mathcal L^{\otimes k})$
$$\mathcal E_1(k) = \bigcup_{i \in \N} Z_i(\mathbb{C}) \subset \Xi(X,\mathcal L^{\otimes k})$$ 
such that for all $v \in  \Xi(X,\mathcal L^{\otimes k}) \setminus \mathcal E_1(k)$, the twisted algebraic vector field $(X,v)$ does  not leave invariant any proper positive-dimensional algebraic subvariety of $X$. By Proposition \ref{proposition-singularlocus}, there exists $k_2 \geq 0$ such that for all $k \geq k_2$, there exists a dense open set 
$U_{\mathrm{red}} \subset \Xi(X,\mathcal L^{\otimes k})$ such that all vector fields in $U_{\mathrm{red}}$ have reduced singularities and such that and  $\Xi(X,\mathcal L^{\otimes k})$ contains some vector field $v_0$ whose restriction to $X_0$ admits a non-resonant singularity $x_0 \in X_0$. 

Set $k_0 = \mathrm{max}(k_1,k_2)$. Consider $k \geq k_0$ and $v_0$ a twisted vector field such that $v_{0 \mid X_0}$ admits a non-resonant singularity $x_0 \in X_0$.

\begin{claimproof} 
There is an open disk $D$ around $x_0$ and a connected analytic neighborhood $U$ of $v_0$ in $\Xi(X,\mathcal L^{\otimes k})^{an}$ such that 
\begin{itemize} 
\item every vector field $v$ in $U$ has a unique singularity $ x_0(v)$ in $D$, 
\item there are analytic functions 
$$\lambda_i: U \rightarrow \mathbb{C}$$ 
such that $\lambda_1(v),\ldots, \lambda_n(v)$ are the eigenvalues of the linear part of $v$ at $x_0(v)$. 
\end{itemize} 
\end{claimproof}

The claim follows from the analytic implicit function theorem similarly as Claim 2 of Theorem \ref{theorem-minimality}. Set $$ \mathcal E_2 = \bigcup_{(k_1,\ldots, k_n) \in \mathbb{Z}^n \setminus \lbrace0 \rbrace} \Big\{ v \in U \mid \sum_{i = 1}^n k_i \lambda_i(v) = 0 \Big\} $$  
for the set of vector fields in $U$ such that the singularity $x_0(v)$ is resonant and set $$ \mathcal E_3 = \bigcup_{Z \subsetneq \Xi(X,\mathcal L^{\otimes k}) \text{  subvariety def. over } l} \Big\{ v \in U \mid \mathrm{coeff}(v) \in Z(\mathbb{C})  \Big\} $$ 
for the set of vector fields in $U$ whose coefficients satisfy an algebraic relation over $l$. By definition 
$$\mathcal E = (\mathcal E_1 \cap U) \cup \mathcal E_2 \cup \mathcal E_3$$
is a countable union of proper analytic closed subsets of $U$. Since each of them has empty interior, it follows from Baire's category theorem that  $ \mathcal E$ also has empty interior in $U$ and in particular is  a proper subset of $U$. So we can consider a vector field $v \in U \setminus \mathcal E$ so that
\begin{itemize} 

\item $v$ is a generic vector field over $l$ of $\Xi(X,\mathcal L^{\otimes k})$ since $v \notin \mathcal E_3$, 
\item $x_0(v)$ is a non-resonant singularity since $v \notin \mathcal E_2$,
\item $v$ does not leave any proper positive-dimensional subvariety of $X$ invariant since $v \notin \mathcal E_1$
\end{itemize}
It follows from Theorem C that the generic type of  $(\mathbb{A}^n,v)$  is minimal. Now applying Lemma \ref{lemma-strongminimality} (iii), we conclude that the solution set of  $(\mathbb{A}^n,v)$ is strongly minimal. Hrushovski-Sokolovic trichotomy theorem, that is, Fact \ref{Hrushovski-Sokolovic} then implies that this strongly minimal set is geometrically trivial. . This concludes the proof of the theorem. 
\end{proof} 

\begin{corollary} 
Let $X$ be a smooth algebraic subvariety of the projective space $\mathbb{P}^N$, let $H_X$ be a smooth hyperplane section of $X$ and set $X_0 = X \setminus H_X$. Consider the family $\Xi(X_0,d)$ of (regular) vector fields on $X_0$ with a pole of order at most $d$ along $H_X$. There exists $d_0 \geq 0$ such that for all $d \geq d_0$, a generic vector field  from the family $\Xi(X_0,d)$ is strongly minimal and geometrically trivial. 
\end{corollary}

\begin{proof}
Since $H_X$ is an ample divisor on $X$, this follows from  Proposition \ref{presentation- twisted-vector field-hyperplanesection} and Theorem \ref{theorem-minimality2}.
\end{proof}  
\subsection{Geometric triviality of decoupled systems} In this section, we will go from geometric triviality of $n$ systems to geometric triviality of the decoupled system. This is well known to experts in differential algebra and model theory (for instance, this was used in \cite{FJMN}) but we include it here for completeness. Fix $k$ a differential field and $\mathcal U$ a saturated differential field relatively to $k$.

\begin{definition} 
Let $(S_1), \ldots, (S_r)$ be $r$ systems of algebraic differential equations defined $K$. The decoupled structure associated with  $(S_1), \ldots, (S_r)$  is the structure $\mathcal M$ with $n$ sorts $\mathcal S_1, \ldots, \mathcal S_n$ such that 
\begin{itemize} 
\item $\mathcal S_i(\mathcal M)$ consists of the solution set $X_i$ of $(S_i)$ in a saturated differentially closed field. 
\item the basic relations on $\prod \mathcal S_i^{n_i}$ are the $k$-definable subsets of $\prod X_i^{n_i}$ 
\end{itemize}

\end{definition} 
By construction, the decoupled structure  $\mathcal M$ is saturated and interpretable in the differentially closed field $\U$ (so in particular $\omega$-stable).
\begin{lemma} 
Assume that $(S_1),\ldots, (S_r)$ are strongly minimal systems of algebraic differential equations. 
Every (possibly imaginary) minimal type of the decoupled structure associated to $(S_1), \ldots, (S_r)$ is nonorthogonal to the generic type of one of the $(S_i)$. 
\end{lemma}

\begin{proof} 
Assume that $p \in S(A)$ is a (possibly imaginary) minimal type of the decoupled structure. Consider $a \models p$. By definition, there is $b = (b_1,\ldots, b_n) \in \prod (S_i)^{n_i}$  such that $a \in\mathrm{dcl}(b,A)$. Note that 
$$a \nind_A b_1,\ldots, b_n$$
as otherwise $a \in \mathrm{acl}(A)$ contradicting the minimality of $p$. Consider the largest $m < n$ such that 
$$a \ind_A b_1,\ldots b_m.$$ 
It follows by transitivity  that
$$a \nind_{A,b_1,\ldots,b_m} b_{m+1}.$$ 
In particular, the type $q = \mathrm{tp}(b_{m+1}/A,b_1,\ldots,b_m)$ is not algebraic. Since $b_{m+1}$ lives on $\mathcal S_i$ for some $i$, $q$ is the generic type of $\mathcal S_{i}$ and we have witnessed nonorthogonality.
\end{proof} 

\begin{definition} 
A theory $T$ of finite rank is \textit{trivial} if whenever $a,b,c$ are three tuples of elements in some model of $T$ and $A$ is a set of parameters, $a,b,c$ are pairwise independent over $A$ if and only if there are independent over $A$. 
\end{definition}

By the local character of forking, a complete theory $T$ is trivial if and only if some saturated model is. Moreover, if $a_1,\ldots, a_r$ are some tuples of elements of some model of a trivial theory $T$ then $a_1,\ldots, a_r$
are independent over $A$ if and only if they are pairwise independent over $A$.

\begin{corollary} \label{triviality-decoupledsystem}
Let $(S_1),\ldots, (S_r)$ be $r$ geometrically trivial strongly minimal sets of dimension $n$. If $y(1),\ldots, y(r)$ are solutions of $(S_1),\ldots, (S_r)$ respectively and $L$ is any differential field over which the systems $(S_1),\ldots, (S_r)$ are defined
$$\mathrm{trdeg}(L\langle y(1),\ldots, y(r)\rangle / L) = n\cdot r$$
unless for some $i \neq j$,
$\mathrm{trdeg}(L\langle y(i),y(j) \rangle/L) = 0 \text{ or } n.$
\end{corollary}

\begin{proof} 
It is well-known that a theory of finite rank is geometrically trivial if and only if all minimal types of $T^{eq}$ are geometrically trivial (see for example \cite{Goode}). The corollary therefore follows from the previous lemma. 
\end{proof}

\begin{proof}[Proof of Theorem A and Theorem B] 
Theorem A and Theorem B now follows from Corollary \ref{triviality-decoupledsystem} and Theorem \ref{theorem-affine-sm} and Theorem \ref{theorem-sm-ample} respectively. 
\end{proof}

\subsection{New meromorphic functions on Painlevé's hierarchy} We finally apply the previous results to the construction of meromorphic functions which are new in the sense of Painlevé \cite{Painleve}, Vingt-et-unième leçon. By {\em a class $\mathcal C$ of meromorphic functions}, we mean a collection 
$$\lbrace \mathcal C(U) \subset \mathcal M(U) \mid U \text{ ranges over the connected open subset of } \mathbb{C} \rbrace$$
of subsets of meromorphic functions which is stable under restriction and analytic continuation i.e. such that if $U \subset V$ are connected nonempty open subsets of $\mathbb{C}$ then 
$$ f \in \mathcal C(V) \text{ iff } f_{\mid U} \in \mathcal C(U).$$
For example, the class $\mathcal C_{rat}$ given by the restrictions to every connected subset $U$ of the rational functions on $\mathbb{C}$ is a class of meromorphic functions. Note that the intersection of classes (open set by open set) is also a class, so it makes sense to talk about the smallest class satisfying certain closure properties. The following definition is taken from \cite{Umemura}.

\begin{definition}[Painlevé's hierarchy]\label{definition-newmeromorphic}  The {\em{class $\mathcal C_0$ of classical meromorphic functions}} is the smallest class of meromorphic functions containing $\mathcal C_{rat}$ and stable under the following operations: 
\begin{itemize} 
\item[(A)] If $f \in \mathcal M(U)$ is in the algebraic closure of $f_1,\ldots, f_n \in \mathcal M(U)$ then
$$ f_1,\ldots, f_n \in \mathcal C_0(U) \Rightarrow f \in \mathcal C_0(U).$$

\item[(PI)] If $f$ satisfies a linear differential equation of the form $y^{(n)} = f_0y + \ldots + f_{n-1}y^{(n-1)}$ then 
$$f_0, \ldots, f_{n-1} \in \mathcal C_0(U) \Rightarrow f \in \mathcal C(U)$$

\item[(PII)] If $\Lambda$ is a lattice such that $A = \mathbb{C}^n/\Lambda$ is an abelian variety and $f \in \mathcal M(U)$ can be written as 
$$ f = \theta \circ  \pi \circ (f_1,\ldots, f_n)$$
where $\pi : \mathbb{C}^n \rightarrow \mathbb{C}^n/\Lambda$ is the projection, $\theta \in \mathbb{C}(A)$ is an abelian function and $f_1,\ldots, f_n \in \mathcal M(U)$ then 
$$ f_1,\ldots, f_n \in \mathcal C_0(U) \Rightarrow f \in \mathcal C_0(U).$$
\end{itemize} 

For $r \geq 1$, the {\em{class $\mathcal C_r$}} is the smallest class of meromorphic functions containing $\mathcal C_{rat}$ stable under the operations (A), (PI), (PII) and 
\begin{itemize} 
\item[(PIII$_r$)] If $f$ satisfies an algebraic differential equation of order $s \leq r$ 
$$ G(y,y',\ldots, y^{(s)}) = 0 \text{ where } G \in \mathcal M(U)[X_0,\ldots, X_s]$$
then 
$$ \mathrm{coeff}(f) \in \mathcal C_r(U) \Rightarrow f \in \mathcal C_r(U).$$
\end{itemize} 
\end{definition} 

{\remark\label{remarkPII} While the operation (PII) is phrased as a closure operation under postcomposition by abelian functions, it can also be presented as a closure operation under solving certain differential equations. 

Let $A$ be an abelian variety of dimension $n$.  The projection $\pi: \mathrm{Lie}(A) \rightarrow A$ is the exponential of the algebraic group $A$ so that if $F = (f_1,\ldots, f_n) : U \rightarrow \mathbb{C}^n$ is an analytic curve with coordinate functions $f_1,\ldots, f_n$ in some differential field $K \subset \mathcal M(U)$, we see that $y(t) = \pi \circ F(t)$ satisfies the differential equation  (with coefficients in $K$)
$$\mathrm{dlog}_A(y(t)) = (f'_1(t),\ldots, f_n'(t)) \in \mathrm{Lie}(A)(K)$$
where $\mathrm{dlog}_G$ denotes the logarithmic derivative of an algebraic group $G$ (see for example \cite{Pillay-Galois} for a general account on logarithmic derivatives). It follows easily that if $\theta \in \mathbb{C}(A)$ then $f = \theta \circ F$ satisfies a differential equation (with coefficients from $K$) of the from
$$G(y,\ldots, y^{(r)}) = 0 \text{ where } G \in K[X_0,\ldots, X_r]$$ 
of order $r \leq n$ which is internal to the constants and whose Galois group (over $K$) is a quotient of the abelian variety $A$. We refer to \cite{Umemura} for more details.}

\begin{definition} 
Let $r \geq 1$.  We say that a meromorphic function $f \in \mathcal M(U)$  is a {\em new meromorphic function of order $r$ in the sense of Painlevé} if it satisfies an algebraic differential equation of order $r$ defined over $\mathbb{C}(z)$ and does {\em not} lie in the class $\mathcal C_{r-1}$.
\end{definition} 

To give the model-theoretic interpretation of the previous definition, fix $\U$ a monster model of the theory $\mathrm{DCF}_0$ and denote by $\Pi_{< r}$ the set of differential equations (defined over $\U$) of order $< r$. By definition, $\Pi_{<r}$ is an $\emptyset$-invariant collection of definable sets of $\U$. 

\begin{proposition} \label{proposition-Umemura}
Let $r \geq 2$ and consider a complex autonomous differential equation  $$(E): F(y,y',\ldots, y^{(r)}) = 0$$ 
where the polynomial $F(X_1,\ldots, X_n) \in \mathbb{C}[X_0,\ldots, X_n]$ is irreducible. 
Assume that some generic solution of $(E)$ is {\em not} a new meromorphic function of order $r$ in the sense of Painlevé. 
Then the generic type of $(E)$ is $\Pi_{<r}$-analyzable. 
\end{proposition} 

\begin{proof}
Consider $(X,v)$ a geometric model of $(E)$. The assumption means that there is a chain of differential fields
$$  \mathbb{C}(t) = K_0 \subset K_1 \subset \ldots \subset K_n$$
each of them being generated (as a differential field) from the previous one using one of the operations (A), (PI), (PII) or (PIII$_r$) such that 
$(X,v)^\U \cap \mathrm{dcl}(K_n) $
contains a realization of the generic type $q$ of $(X,v)$. Note that we have written $\mathrm{dcl}(K_n)$ to stress that we are thinking about tuples. If $X$ is affine and embedded in $\U^N$ then $(X,v)^\U \cap \mathrm{dcl}(K_n)$ means $(X,v)^\U \cap K_n^N$.

Every $K_i$ has finite transcendence degree over the previous one. Indeed, this is clear for operations (A), (PI) and (PIII$_r$) and for operation (PII), this follows from Remark \ref{remarkPII}. It follows that each $K_i$ is {\em a finitely generated differential field over $\mathbb{C}$ of finite transcendence degree}. The proof now follows from induction on the length $n$ of the chain. 
\begin{itemize} 
\item if $n = 0$, then by definition, we have a realization of $q$ in 
$\mathrm{dcl}(\mathbb{C},t)$ and since $\mathrm{tp}(t/\mathbb{C})$ is internal to the constants, the type $q$ is also internal to the constants and in particular analyzable in $\Pi_{<r}$. 

\item Assume that $n \geq 1$ and that the results holds from all autonomous differential equations whose solutions can be expressed with chains of differential fields of $< n$ steps and consider $(E)$ with a generic solution which can be expressed in a chain of $n$ steps as above. 
\end{itemize}
By the primitive element of \cite{Pogudin} and the finite transcendence degree observation above, we may assume that $K_{n - 1} = \mathbb{C} \langle f \rangle$  where $f$ satisfies an autonomous equation $(E^\ast)$. The induction hypothesis implies that the generic type of $(E^\ast)$ is $\Pi_{<r}$-analyzable.

We distinguish cases according to the operation used to define $K_n$ from $K_{n - 1}$: 

\begin{itemize} 

\item[(A)] In this case, we have that 
$$(X,v)^\U \cap \mathrm{acl}(\mathbb{C},f)$$    
contains a realization of $q$. Since $\mathrm{tp}(f/\mathbb{C})$ is $\Pi_{<r}$-analyzable, so is $q$. 

\item[(PI)]In this case, there is a solution $g$ of a linear differential equation 
$$ y^{(n)}  = f_0 y  +  \ldots + f_{n-1}y^{(n-1)} $$
with $f_0,\ldots, f_r \in \mathrm{dcl}(\mathbb{C},f)$ such that  
$(X,v)^\U \cap \mathrm{dcl}(\mathbb{C},f,g)$   
contains a realization of $q$. Since $\mathrm{tp}(f/\mathbb{C})$ is $\Pi_{<r}$-analyzable and so is $\mathrm{tp}(g/f,\mathbb{C})$ (as this type is internal to the constants), it follows that $q$ is also $\Pi_{<r}$-analyzable. 

\item[(PII)] In this case, by the remark \ref{remarkPII}, there is a solution $g$ of an abelian logarithmic differential equation with coefficients in $K_{n-1} = \mathrm{dcl}(\mathbb{C}, f)$
$$ \mathrm{dlog}_A(y) = \alpha \in \mathrm{Lie}(A)(K_{n-1}) $$
such that $(X,v)^\U \cap \mathrm{dcl}(\mathbb{C},f,g)$ contains a realization of $q$. Since $\mathrm{tp}(f/\mathbb{C})$ and $\mathrm{tp}(g/f,\mathbb{C})$ are both $\Pi_{<r}$-analyzable (the latter one being internal to the constant), it follows that $q$ so is. 
\item[(PIII$_r$)]  In this case, there is a solution $g$ of an algebraic differential equation 
$$ G(y,y'\ldots, y^{(s)}) = 0 \text{ where } G \in K_{n-1}[X_0,\ldots, X_s]$$
such that
$(X,v)^\U \cap acl(\mathbb{C},f,g)$    
contains a realization of $q$. Since $\mathrm{tp}(f/\mathbb{C})$ is $\Pi_{<r}$-analyzable and so is $\mathrm{tp}(g/f,\mathbb{C})$ (as this type satisfies an order $s$ equation with $s < r$), it follows that $q$ is also $\Pi_{<r}$-analyzable. 
\end{itemize} 
This conclude the proof of the proposition.
\end{proof}

As a corollary of our results, we obtain procedures to produce new meromorphic functions of order $r$ in the sense of Painlevé for arbitrary large $r \geq 2$.

\begin{corollary} 
Let $n,d \geq 2$. Consider $f \in \mathbb{C}(x_1,\ldots, x_n) \setminus \mathbb{C}$ a rational function and 
$y : t \mapsto y(t) = (y_1(t),\ldots, y_n(t))$
a nonconstant integral curve of a generic vector field from $\Xi(n,d)$ defined over a complex domain $U$. Then the meromorphic function $$\phi = f \circ y: t \mapsto f(y(t)) \in \mathcal M(U)$$ is a new meromorphic function of order $n$ in the sense of Painlevé.
\end{corollary} 

\begin{corollary} 
Let $X$ be a smooth projective variety, let $H_X$ be a smooth hyperplane section of $X$ and let $f \in \mathbb{C}(X) \setminus \mathbb{C}$ be a rational function. If
$y : t \mapsto y(t)$
is a nonconstant integral curve of a generic vector field from $\Xi(X_0,d)$ defined over a complex domain $U$ then for $d \gg 0$, the meromorphic function 
$$\phi = f \circ y: t \mapsto f(y(t)) \in \mathcal M(U)$$ is a new meromorphic function of order $\mathrm{dim}(X)$ in the sense of Painlevé.
\end{corollary}

\begin{proof}[Proof of the corollaries] 
Theorem A and Theorem B ensure that the generic algebraic vector fields  from the families $\Xi(n,d)$ or $\Xi(X_0,d)$ for $d \gg 0$ respectively are strongly minimal. 

Let $(X,v)$ be a strongly minimal algebraic vector field, let $t \mapsto y(t)$ be a nonconstant integral curve of this vector field defined on a complex domain $U$ and let $f \in \mathbb{C}(X) \setminus \mathbb{C}$ be a rational function.  By strong minimality, every differential subfield of $(\mathbb{C}(X),\delta_v)$ containing properly $\mathbb{C}$ has transcendence degree $ n= \mathrm{dim}(X)$. It follows with the notation of Remark \ref{remark-vectorfields-autonomous} that $f$ is a {\em weakly primitive function} in the sense that 
$$\sigma_f : X \dashrightarrow \mathrm{Jet}^n(\mathbb{A}^1)$$ is generically finite on its image which has dimension $ n = \mathrm{dim}(X)$. Therefore, the Zariski-closure of the image is given by a single equation 
$$(E_f) : F(y,y',\ldots, y^{(n)}) = 0$$
which is an order $n$ differential equation satisfied by $\phi = f \circ y \in \mathcal M(U)$. It remains to show that $\phi$ does not lie in the class $\mathcal C_{r - 1}$. 

First, note that $(X,v)$ is in generically finite-to-finite correspondence with every geometric model of $(E_f)$ so that the generic type $q$ of $(E_f)$ is a minimal type of order $r$. It follows that $q$ is orthogonal to all the types living on $\Pi_{<r}$ so that $q$ is {\em not} $\Pi_{<r}$-analyzable. The corollary now follows from Proposition \ref{proposition-Umemura}: 
since $y$ is noncontant and $(X,v)$ is strongly minimal, the analytic curve $y$ is Zariski-dense in $X$ so that $\phi = f \circ y$ realizes the generic type $q$ of $(E_f)$ and the corollary follows from Proposition \ref{proposition-Umemura}.  
\end{proof}


\begin{thebibliography}{BSCFN20}

\bibitem[Bou98]{Bouscaren}
Elisabeth Bouscaren, editor.
\newblock {\em Model theory and algebraic geometry}, volume 1696 of {\em
  Lecture Notes in Mathematics}.
\newblock Springer-Verlag, Berlin, 1998.

\bibitem[BSC17]{Blazquez-Casale}
David Bl\'{a}zquez-Sanz and Guy Casale.
\newblock Parallelisms \& {L}ie connections.
\newblock {\em SIGMA Symmetry Integrability Geom. Methods Appl.}, Vol. 13 (2017), Paper No. 086, 28pp.

\bibitem[BSCFN20]{BCFN}
David Bl\'{a}zquez-Sanz, Guy Casale, James Freitag, and Joel Nagloo.
\newblock Some functional transcendence results around the {S}chwarzian
  differential equation.
\newblock {\em Ann. Fac. Sci. Toulouse Math. (6)}, Vol. 29 (2020), no. 5, pp. 1265--1300.

\bibitem[Bui86]{Buium-projective}
Alexandru Buium.
\newblock {\em Differential function fields and moduli of algebraic varieties},
 Vol. 1226 of {\em Lecture Notes in Mathematics}.
\newblock Springer-Verlag, Berlin, 1986.

\bibitem[Bui93]{Buium}
Alexandru Buium.
\newblock Geometry of differential polynomial functions. {I}. {A}lgebraic
  groups.
\newblock {\em Amer. J. Math.}, Vol. 115 (1996), no. 6, pp. 1385--1444.

\bibitem[CD22]{Casale-Davy}
Guy Casale and Damien Davy.
\newblock Sp\'{e}cialisation du groupo\"{\i}de de galois d'un champ de
  vecteurs.
\newblock {\em Ann. Inst. Fourier (Grenoble)}, Vol. 72 (2022), no. 6, pp.2399--2447.

\bibitem[CFN20]{Casale-Freitag-Nagloo}
Guy Casale, James Freitag, and Joel Nagloo.
\newblock Ax-{L}indemann-{W}eierstrass with derivatives and the genus 0
  {F}uchsian groups.
\newblock {\em Ann. of Math. (2)}, Vol. 192 (2020), no. 3, pp. 721--765.

\bibitem[Cou11]{Coutinho}
S.~C. Coutinho.
\newblock Foliations of multiprojective spaces and a conjecture of {B}ernstein
  and {L}unts.
\newblock {\em Trans. Amer. Math. Soc.}, Vol. 363 (2011), no. 4, pp. 2125--2142.

\bibitem[CP06]{Coutinho-Pereira}
S.~C. Coutinho and J.~V. Pereira.
\newblock On the density of algebraic foliations without algebraic invariant
  sets.
\newblock {\em J. Reine Angew. Math.}, Vol. 594 (2006), pp. 117--135.

\bibitem[CS80]{Camacho-Sad}
C\'{e}sar Camacho and Paulo Sad.
\newblock Sur l'existence de courbes analytiques complexes invariantes d'un
  champ de vecteurs holomorphe.
\newblock {\em C. R. Acad. Sci. Paris S\'{e}r. A-B}, Vol. 291 (1980), no. 4, pp. A311--A312.

\bibitem[DF23]{Devilbiss-Freitag}
Matthew DeVilbiss and James Freitag.
\newblock Generic differential equations are strongly minimal.
\newblock {\em Compos. Math.}, Vol. 159 (2023), no. 7, pp. 1387--1412.

\bibitem[DEJ25]{DEJ}
Yutong Duan, Christine Eagles, and Léo Jimenez.
\newblock Algebraic independence of solutions to multiple {L}otka-{V}olterra
  systems, 2025.
\newblock {\em preprint} arXiv:2507.17090v2


\bibitem[FJM22]{Freitag-Jaoui-Moosa}
James Freitag, R\'{e}mi Jaoui, and Rahim Moosa.
\newblock When any three solutions are independent.
\newblock {\em Invent. Math.}, Vol. 230 (2022), no. 3, pp. 1249--1265.

\bibitem[FJMN22]{FJMN}
James Freitag, Rémi Jaoui, David Marker, and Joel Nagloo.
\newblock {On the Equations of Poizat and Liénard}.
\newblock {\em Int. Math. Res. Not. IMRN}, no. 11 (2023), pp. 16478--16539.

\bibitem[FN22]{Freitag-Nagloo}
James Freitag and Joel Nagloo.
\newblock Algebraic relations between solutions of Painlevé equations,
\newblock {preprint (2022)}, arXiv:1710.03304v2.

\bibitem[FS18]{Freitag-Scanlon}
James Freitag and Thomas Scanlon.
\newblock Strong minimality and the {$j$}-function.
\newblock {\em J. Eur. Math. Soc. (JEMS)}, Vol. 20 (2018), no. 1, pp. 119--136.

\bibitem[FW24]{Feng-Wibmer}
Ruyong Feng and Michael Wibmer.
\newblock Differential {G}alois groups, specializations and {M}atzat's conjecture,
  2024.
\newblock arXiv:2209.01581v2, to appear in {\em Memoirs of the AMS}.

\bibitem[Goo91]{Goode}
John~B. Goode.
\newblock Some trivial considerations.
\newblock {\em J. Symbolic Logic}, Vol. 56 (1991), no. 2, pp. 624--631.

\bibitem[Gui07]{Guillot}
Adolfo Guillot.
\newblock Sur les \'{e}quations d'{H}alphen et les actions de {${\rm SL}_2({\bf
  C})$}.
\newblock {\em Publ. Math. Inst. Hautes \'{E}tudes Sci.}, Vol. 105 (2007), pp. 221--294.

\bibitem[Har77]{Hartshorne}
Robin Hartshorne.
\newblock {\em Algebraic geometry}.
\newblock Graduate Texts in Mathematics, No. 52. Springer-Verlag, New
  York-Heidelberg, 1977.

\bibitem[HI03]{Hrushovski-Itai}
E.~Hrushovski and M.~Itai.
\newblock On model complete differential fields.
\newblock {\em Trans. Amer. Math. Soc.}, Vol. 355 (2003), no. 11, pp. 4267--4296.

\bibitem[Hru96]{Hrushovski-MordellLang}
Ehud Hrushovski.
\newblock The {M}ordell-{L}ang conjecture for function fields.
\newblock {\em J. Amer. Math. Soc.}, Vol. 9 (1996), no. 3, pp. 667--690.

\bibitem[Hru02]{Hrushovski-Galois}
Ehud Hrushovski.
\newblock Computing the {G}alois group of a linear differential equation.
\newblock In {\em Differential {G}alois theory ({B}edlewo, 2001)}, 
  {\em Banach Center Publ.}, Vol. 58 (2002), pp. 97--138. 

\bibitem[HS96]{Hrushovski-Sokolovic}
Ehud Hrushovski and Zeljko Sokolovic.
\newblock Minimal subsets of differentially closed fields,
\newblock {preprint}, 1996.

\bibitem[HS99]{Hrushovski-Scanlon}
Ehud Hrushovski and Thomas Scanlon.
\newblock Lascar and {M}orley ranks differ in differentially closed fields.
\newblock {\em J. Symbolic Logic}, Vol. 64 (1999), no. 3, pp. 1280--1284.

\bibitem[Jao20a]{jaoui-Bulletin}
R\'{e}mi Jaoui.
\newblock Corps diff\'{e}rentiels et flots g\'{e}od\'{e}siques
  {I}---{O}rthogonalit\'{e} aux constantes pour les \'{e}quations
  diff\'{e}rentielles autonomes.
\newblock {\em Bull. Soc. Math. France}, Vol. 148 (2020), no. 3, pp. 529--595.

\bibitem[Jao20b]{jaoui-Confluentes}
R\'{e}mi Jaoui.
\newblock Rational factors, invariant foliations and algebraic disintegration
  of compact mixing {A}nosov flows of dimension 3.
\newblock {\em Confluentes Math.}, Vol. 12 (2020), no. 2, pp. 49--78.

\bibitem[Jao21]{jaoui-ANT}
R\'{e}mi Jaoui.
\newblock Generic planar algebraic vector fields are strongly minimal and
  disintegrated.
\newblock {\em Algebra Number Theory}, Vol. 15 (2021), no. 10, pp. 2449--2483.

\bibitem[JM25]{Jaoui-Moosa}
R\'{e}mi Jaoui and Rahim Moosa.
\newblock Abelian reduction in differential-algebraic and bimeromorphic
  geometry.
\newblock {\em Ann. Inst. Fourier (Grenoble)}, Vol. 75 (2025), no. 4, pp. 1811--1853.

\bibitem[MMP06]{Marker}
David Marker, Margit Messmer, and Anand Pillay.
\newblock {\em Model theory of fields}, volume~5 of {\em Lecture Notes in
  Logic}.
\newblock Association for Symbolic Logic,  second edition, 2006.

\bibitem[MP25]{Pizarro}
Amador Martin-Pizarro.
\newblock Model theory, differential algebra and functional transcendence,
\newblock {Séminaire Bourbaki no. 1246, 2025-2026}.

\bibitem[Mum99]{Mumford}
David Mumford.
\newblock {\em The red book of varieties and schemes}, Vol. 1358 of {\em
  Lecture Notes in Mathematics}.
\newblock Springer-Verlag, Berlin, expanded edition, 1999.

\bibitem[NP14]{Nagloo-Pillay1}
Joel Nagloo and Anand Pillay.
\newblock On the algebraic independence of generic {P}ainlev\'{e}
  transcendents.
\newblock {\em Compos. Math.}, Vol. 150 (2014), no.4, pp. 668--678.

\bibitem[NP17]{Nagloo-Pillay2}
Joel Nagloo and Anand Pillay.
\newblock On algebraic relations between solutions of a generic {P}ainlev\'{e}
  equation.
\newblock {\em J. Reine Angew. Math.}, Vol. 726 (2017), pp. 1--27.

\bibitem[Pai97]{Painleve}
Paul Painlev\'{e}.
\newblock Leçons, sur la théorie analytique des équations différentielles,
  professées à {S}tockholm (septembre, octobre, novembre 1895) sur
  l'invitation de S.M. le roi de {S}uède et de {N}orwège.
\newblock {\em Ann. Sci. \'{E}cole Norm. Sup. (3)}, 1897.

\bibitem[Per12]{Pereira-JEMS}
Jorge~Vit\'{o}rio Pereira.
\newblock The characteristic variety of a generic foliation.
\newblock {\em J. Eur. Math. Soc. (JEMS)}, Vol. 14 (2012), no. 1, pp. 307--319.

\bibitem[Pil96]{Pillay}
Anand Pillay.
\newblock {\em Geometric stability theory}, volume~32 of {\em Oxford Logic
  Guides}.
\newblock The Clarendon Press, Oxford University Press, New York, 1996.

\bibitem[Pil04]{Pillay-Galois}
Anand Pillay.
\newblock Algebraic {$D$}-groups and differential {G}alois theory.
\newblock {\em Pacific J. Math.}, Vol. 216 (2004), no.2, pp. 343--360.


\bibitem[PZ03]{Pillay-Ziegler}
Anand Pillay and Martin Ziegler.
\newblock Jet spaces of varieties over differential and difference fields.
\newblock {\em Selecta Math. (N.S.)}, Vol. 9 (2003), no. 4,pp. 579--599.

\bibitem[Pog15]{Pogudin}
Gleb~A. Pogudin.
\newblock The primitive element theorem for differential fields with zero
  derivation on the ground field.
\newblock {\em J. Pure Appl. Algebra}, Vol. 219 (2015), no. 9, pp. 4035--4041.

\bibitem[Poi83]{Poizat}
Bruno Poizat.
\newblock C'est beau et chaud.
\newblock In {\em Study group on stable theories ({B}runo {P}oizat), 3,
  1980/1982}, (1983),  Exp. no. 7, pp. 11.

\bibitem[PP15]{Pereira-Pirio}
Jorge~Vit\'{o}rio Pereira and Luc Pirio.
\newblock {\em An invitation to web geometry}, volume~2 of {\em IMPA
  Monographs}.
\newblock Springer, Cham, 2015.

\bibitem[PS21]{Pila-Scanlon}
Jonathan Pila and Thomas Scanlon.
\newblock Effective transcendental Zilber-Pink for variations of Hodge
  structures,
\newblock {preprint (2021)}, arXiv:2105.05845.


\bibitem[Ros74]{Rosenlicht}
Maxwell Rosenlicht.
\newblock The nonminimality of the differential closure.
\newblock {\em Pacific J. Math.}, Vol. 52 (1974), pp. 529--537.

\bibitem[She73]{Shelah}
Saharon Shelah.
\newblock Differentially closed fields.
\newblock {\em Israel J. Math.}, Vol. 16 (1973), pp. 314--328.

\bibitem[Sza09]{Szamuely}
Tam\'{a}s Szamuely.
\newblock {\em Galois groups and fundamental groups}, Vol. 117 (2009) of {\em
  Cambridge Studies in Advanced Mathematics}.

\bibitem[Ume90]{Umemura}
Hiroshi Umemura.
\newblock Second proof of the irreducibility of the first differential equation
  of {P}ainlev\'{e}.
\newblock {\em Nagoya Math. J.}, Vol. 117 (1990), pp. 125--171.

\end{thebibliography}

\end{document}